\documentclass[12pt]{article}
\usepackage{amssymb}
\usepackage{amsmath}
\usepackage{amsthm}
\usepackage{color}
\usepackage{comment}
\usepackage{geometry}		
\usepackage{graphicx}

\usepackage{natbib}
\usepackage{tcolorbox}
\usepackage[colorlinks=true, citecolor=blue]{hyperref}

\makeatletter
	\@addtoreset{equation}{section}
\makeatother

\DeclareFontFamily{OT1}{pzc}{}
\DeclareFontShape{OT1}{pzc}{m}{it}{<-> s * [1.10] pzcmi7t}{}
\DeclareMathAlphabet{\mathpzc}{OT1}{pzc}{m}{it}

\let\originalleft\left
\let\originalright\right
\renewcommand{\left}{\mathopen{}\mathclose\bgroup\originalleft}
\renewcommand{\right}{\aftergroup\egroup\originalright}

\begin{document}

\newcommand{\bO}{{\bf 0}}
\newcommand{\co}{\mathpzc{o}}
\newcommand{\rD}{{\rm D}}
\newcommand{\ee}{\varepsilon}
\newcommand{\ri}{{\rm i}}

\newtheorem{theorem}{Theorem}[section]
\newtheorem{corollary}[theorem]{Corollary}
\newtheorem{lemma}[theorem]{Lemma}
\newtheorem{proposition}[theorem]{Proposition}

\theoremstyle{definition}
\newtheorem{definition}{Definition}[section]
\newtheorem{assumption}[definition]{Assumption}




\title{Fractional order induced bifurcations in Caputo-type denatured Morris-Lecar neurons}
\author{Indranil~Ghosh$^1$, Hammed Olawale~Fatoyinbo$^2$\\\\
$^1$School of Mathematical and Computational Sciences,\\
Massey University,\\
Palmerston North, 4410,\\
New Zealand\\\\
$^2$Department of Mathematical Sciences,\\
School of Engineering, Computer and Mathematical Sciences,\\
Auckland University of Technology,\\
Auckland, 1142,\\
New Zealand
}

\maketitle


\begin{abstract}

We set up a system of Caputo-type fractional differential equations for a reduced-order model known as the {\em denatured} Morris-Lecar (dML) neurons. This neuron model has a structural similarity to a FitzHugh-Nagumo type system. We explore both a single-cell isolated neuron and a two-coupled dimer that can have two different coupling strategies. The main purpose of this study is to report various oscillatory phenomena (tonic spiking, mixed-mode oscillation) and bifurcations (saddle-node and Hopf) that arise with variation of the order of the fractional operator and the magnitude of the coupling strength for the coupled system. Various closed-form solutions as functions of the system parameters are established that act as the necessary and sufficient conditions for the stability of the equilibrium point. The theoretical analysis are supported by rigorous numerical simulations.

\end{abstract}

\section{Introduction}
\label{sec:intro}
Dynamics of neurons are typically modeled using integer-order time differential operators. In 1952 Alan Hodgkin and Andrew Huxley published their seminal work~\citep{HoHu52} on a mathematical model detailing the activation and deactivation of a spiking neuron, constructed from experimental results collected from giant squid axons. In 1961 the Fitzhugh Nagumo model~\citep{Fi61} was constructed by simplifying a two-dimensional version of the Hodgkin-Huxley model. Then in 1981, Catherine Morris and Harold Lecar~\citep{MoLe81} put forward a two-dimensional reduced model that was able to reproduce the dynamics observed in the Hodgkin-Huxley model, besides exhibiting both class I and II excitability along with a wide array of interesting complex behaviors. In their textbook, Ermentrout and Terman~\citep{ErTe10} did a rigorous bifurcation analysis of the Morris-Lecar neuron model, exhibiting rich bifurcation patterns and realizing the intricate dynamics this model has to offer.

The purpose of this paper is to work with a variant of the Morris-Lecar model which was first studied by Schaeffer and Cain in their textbook~\citep{ScCa18}. They called this model the {\em denatured} Morris-Lecar (dML) model, which looks like
\begin{equation}
\label{eq:dML}
  \begin{aligned}
    \dot{x} &= x^2(1-x) - y + I, \\
    \dot{y} &= Ae^{\alpha x} - \gamma y.
\end{aligned}  
\end{equation}
Here, $x$ represents the voltage and $y$ represents the recovery-like variable. Variable $x$ has a cubic nonlinearity and demonstrates a positive feedback to the neurons leading to them firing, whereas variable $y$ has an exponential behavior demonstrating a negative feedback to the neurons, modeling the refractory period. We have four local parameters $I$, $A$, $\alpha$, and $\gamma$. Parameter $I$ is the external stimulation current and can be both positive and negative. Other parameters are all positive constants. 

Reduced order models like~\eqref{eq:dML} are computationally efficient and methodically capture the complex features exhibited by the highly nonlinear conductance-based models which possess multiple interacting variables and timescales. Model~\eqref{eq:dML} demonstrates similar dynamics as the original Morris-Lecar model and is a major motivation for this paper. Furthermore,~\eqref{eq:dML} has a structural analogy to the Fitzhugh-Nagumo neuron model, with the difference lying in their $y$-nullclines. For the Fitzhugh-Nagumo model, the $y$-nullcline is a straight line corresponding to the linear structure of the recovery-like variable, whereas for~\eqref{eq:dML} the $y$-nullcline curve upward corresponding to the exponential term. In 2022 Fatoyinbo {\em et al.}~\citep{FaMu22} employed a codimension-one and -two bifurcation analysis of~\eqref{eq:dML} and reported the appearance of saddle-node and Hopf bifurcations. They also showed that for a particular parameter regime, the system exhibits bistability. Codimension-two bifurcations like cusp bifurcation and generalised Hopf bifurcation were also reported.

All these studies, however, are modeled with standard integer-order ordinary differential equations (ODEs). In this paper, we take a deviation from studying a system governed by classical integer-order ODEs and instead study a system governed by Caputo-type fractional-order differential equations to incorporate memory effects into the system. Specifically, we first analyze the fractional order version of a single cell dML model~\eqref{eq:dML}. We then extend our analysis to a two-coupled system of identical dML neurons connected by a bidirectional synaptic coupling. We employ two different coupling schemes which we will elaborate on later in this text. The single-cell model is two-dimensional whereas the two-coupled system is four-dimensional in structure. 

The motivation behind studying fractional-order differential equations modeling neuron dynamics is heavily backed by studies undertaken in this direction. Anastasio~\citep{An94} suggested that the oculomotor integrator in the brain that controls eye movements, might be fractional-order. Lundstrom {\em et al.}~\citep{LuHi08} argued that there exists a multiple time scale adaptation in single rat neocortical neurons which is consistent with fractional order differential equations. Lundstrom and Richner~\citep{LuRi23} utilized a novel theoretical study and computational neuron models to realize how neural adaptation and fractional dynamics govern neural excitability in human EEG recordings. Vazquez-Guerrero {\em et al.}~\citep{VaTu24} showed that super-capacitors having fractional-order derivative and memcapacitive properties can be used to implement fractional-order spiking neurons. 

A suite of recent research has been published on fractional-order models of neuron dynamics, for example, fractional-order Hodgkin-Huxley model~\citep{ShEl11, NaSw14, AbFo22}, Fitzhugh-Nagumo model~\citep{FaAl23, YaSu24, KuEr25}, Hindmarsh-Rose model~\citep{Ka17, XiKa14, LiZh21}, and Morris-Lecar model~\citep{ShWa14, ShMo23, ChAl23}. A major point of reference for the analysis of the single cell dML model is the work by Sharma {\em et al.}~\citep{ShMo23}, where the authors first dealt with a Caputo type fractional order dynamics of two-dimensional type I and II Morris-Lecar neuron. Then they continued their analysis to a three-dimensional slow-fast system of the Morris-Lecar neuron. Finally, they build a random network of the two-dimensional Morris-Lecar neurons and study the complex dynamics of the same. The theoretical framework of the dynamics of the single-cell dML model of this paper is built on the work by Sharma {\em et al.}~\citep{ShMo23}. Then for the analysis of the two-coupled dimer of dML neurons, we follow the framework set in the work by Mondal {\em et al.}~\citep{MoKa24} where they closely analyzed the emergent dynamics of a single cell and a two-coupled Wilson-Cowan neuron network. Authors in \citep{ChAl23} studied discrete Caputo-type fractional order systems of conductance based Morris-Lecar neurons (two-dimensional single-cell, three-dimensional slow-fast single-cell, and a Erd\"os-Renyi random network of the neurons) and reported a plethora of dynamical responses. The difference between the two-dimensional model considered in the above two works and the dML neuron model~\eqref{eq:DML_2D} is that the dML neuron model is a simple reduced model of the Morris-Lecar system which has structural similarity to a FitzHugh-Nagumo type model, driving the complexity of the analysis (for example, establishing various closed-form necessary and sufficient conditions as functions of the local parameters) relatively easier to handle.

This paper is organised as follows: In section~\ref{sec:Frac} we discuss the preliminaries of Caputo-type fractional order differential operator, and the motivation behind modeling the dynamics of dML neurons with such fractional operators. We also state the fractional linearization theorem and how we are going to utilize it further in our analysis. In section~\ref{sec:Model} we first establish a system of Caputo-type fractional order differential equations modeling the dynamics of a single dML neuron. Then we take two of these identical neurons and couple them via two different synapse schemes: first that couples a bidirectional linear flow between the neurons and second that couples a bidirectional sigmoidal flow. The single-cell model is two-dimensional and the two-coupled model is four-dimensional in structure. Note that we only consider commensurate fractional-order systems throughout this study. In section~\ref{sec:dML_2D} a qualitative analysis of the two-dimensional system is provided along with rigorous nuemrics. We establish closed-form conditions for bifurcations like saddle-node and Hopf in terms of the local parameters of the system and the order of the fractional operators. This tells us how external stimulation current and memory effects (implemented via varying the fractional order) drive different complex regimes of the model. We support our analysis with numerics performed using \texttt{Julia} on the back-end and \texttt{Python} for the front-end plotting. We do the same thing for the coupled system in section~\ref{sec:dML_4D} and notice different dynamical behaviors like tonic-spiking, typical bursts, and mixed-mode oscillations. Finally, concluding remarks are provided in section~\ref{sec:Conclusions}.

\section{Fractional differential operator}
\label{sec:Frac}
Generally speaking, there exist three non-equivalent definitions of fractional differential operators: Gr\"unwald-Letnikov type, Riemann-Liouville type, and Caputo type. Our focus is Caputo type differential operator, which we use to study both the single-cell and two-coupled dML neurons. The Caputo type fractional differential operator~\citep{Ca67} is given by
\begin{align}
\label{eq:Caputo}
    ^C \mathcal{D}_{t_0}^\beta X(t) &= I_{t_0}^{n - \beta} X^{(n)}(t) \\
    &= \frac{1}{\Gamma(n - \beta)} \int_{t_0}^t \frac{X^{(n)}(\tau) d\tau}{(t-\tau)^{1+\beta - n}}.
\end{align}
In our problem setting we have $t_0=0$ and $n=1$. Here $C$ represents {\em Caputo}, and $\mathcal{D}_{t_0}^\beta$ is the time-fractional Caputo derivative having an order $\beta>0$ at time $t$. the order $\beta$ is also sometimes referred to as an {\em memory index}~\citep{MoKa24}. Initial time is indicated by $t_0$ and $X(\cdot) \in \mathbb{R}^N$ is a differentiable function. Furthermore, $I_{t_0}^{n - \beta}$ denotes the {\em Riemann-Liouville fractional integral}. For an order $\alpha$ and an arbitrary but fixed initial point $t_0$, the integral is defined by
\begin{align}
    \label{eq:RiemannLiouville}
    I^\beta_{t_0} X(t) = \frac{1}{\Gamma(\beta)} \int_{t_0}^t X(\tau) (t - \tau)^{\beta - 1} d\tau,
\end{align}
where $\Gamma(\cdot)$ is the {\em gamma} function. In definition~\eqref{eq:Caputo}, the Riemann-Liouville fractional integral is of order $n-\beta$. Also, $X^{(n)}(t)$ is the $n^{\rm th}$ order derivative of the function $X(t)$.

For our problem statement, we restrict $\beta \in [0, 1]$ and the initial time is set as $t_0 = 0$. We also set $n = 1$. The Caputo fractional-order differential equation then looks like
\begin{align}
\label{eq:Caputo2}
    ^C \mathcal{D}_{0}^\beta X(t)= \frac{1}{\Gamma(1 - \beta)} \int_{0}^t \frac{X^{(1)}(\tau) d\tau}{(t-\tau)^{\beta}}.
\end{align}
Caputo derivatives assimilate {\em memory} effects into the otherwise classical system where $\beta$ is an integer, meaning the fractional order differential operator $\mathcal{D}$ exhibits non-local properties. One advantage of the Caputo operator over other fractional order operators like the Gr\"unwald-Letnikov type or the Riemann-Liouville type is that the Caputo operator is more applicable to real-world physical and engineering problems in the sense that one needs to specify the initial conditions in terms of the integer-order derivatives, making the modeling tasks more tractable~\citep{Ka17}. Furthermore, the Caputo derivative of a constant is zero~\citep[Property 2.4]{LiDe07}, which is not true for the other type fractional derivatives. Some important reads on Caputo fractional derivative are Kilbas {\em et al.}~\citep{Ki06}, Podlubny~\citep{Po98}, and Diethelm~\citep{Di19}, among others.

Note that when $\beta \to 1$, the Caputo operator $^C \mathcal{D}_{0}^\beta X(t)$ converges to the standard derivative $\frac{dX(t)}{dt}$ having order $1$. The classical linearization theorem by Hartman and Grobman for standard integer-order dynamical systems states that the qualitative dynamics of a system in the neighborhood of a hyperbolic equilibrium point (with no eigenvalue having a zero real part) is similar to the dynamics of its linearization near the equilibrium point, making the linearisation an indispensable tool for analyzing dynamical systems~\citep{ArPl92}. This was extended for the case of fractional-order dynamical systems by Li and Ma~\citep{LiMa13}. The linearization at an equilibrium point of $^C \mathcal{D}_{0}^\beta X(t)$ can be written as
\begin{align}
\label{eq:linearization}
    ^C \mathcal{D}_{0}^\beta X = J(X^*)X,
\end{align}
where $X \in \mathbb{R}^N$ and $X^*$ is an equilibrium point of the $N$-dimensional fractional-order system. Furthermore, $J$ is the Jacobian matrix of the system evaluated at $X^*$. Following Matignon~\citep{Ma96}, the system~\eqref{eq:linearization} is asymptotically stable if and only if $|\arg(\lambda)| >\frac{\beta \pi}{2}$, for all the complex eigenvalues $\lambda$ of the Jacobian matrix $J \in \mathbb{R}^{N \times N}$. The equilibrium of the system is a saddle if $|\arg(\lambda)| <\frac{\beta \pi}{2}$~\citep{AhEl06}. Kaslik ~\citep{Ka17} showed that for $N=2$, any eigenvalue $\lambda$ of the Jacobian matrix $J(X^*)$ evaluated at the equilibrium point is asymptotically stable if and only if
\begin{align}
\label{eq:StabCond}  
\delta(X^*) = \det(J(X^*))>0, \qquad {\rm and} \qquad \tau(X^*) = {\rm trace}(J(X^*))< 2\sqrt{\delta(X^*)}\cos\bigg(\frac{\beta \pi}{2} \bigg).
\end{align}

\section{Model Specification}
\label{sec:Model}
It should be noted that we only consider commensurate models of fractional order differential equations in this paper. That means, the fractional order $\beta$ is same for all the state variables governing the dynamics of the dML neurons. The two-dimensional commensurate fractional order single-cell model of a dML neuron is given by

\begin{equation}
\label{eq:DML_2D}
\begin{aligned}
    ^C \mathcal{D}_{0}^\beta x &= x^2(1-x) - y +I, \\
    ^C \mathcal{D}_{0}^\beta y &= A e^{\alpha x} - \gamma y.\\
\end{aligned}
\end{equation}

The biophysical meaning of the state variables $(x, y)$ and the parameters $A$, $\alpha$, $\gamma$, and $I$ remain the same as~\eqref{eq:dML}. The difference lies in the order of the differential operators which could take any value $\beta \in (0, 1]$, that is~\eqref{eq:dML} is a special case of~\eqref{eq:DML_2D} when $\beta = 1$. We then extend our single-cell model~\eqref{eq:DML_2D} and couple two of them through a coupling scheme, now producing a four-dimensional system of fractional order differential equations. Note that we consider the two neurons to be identical, in the sense that both their local parameters are kept the same to each other. Also, the coupling is bidirectional and is between the voltage variables $x$. 

We employ two coupling schemes. The first one is a simple coupling strategy that quantifies a linear flow between the two neurons, modeled by $\theta(x_j - x_i)$, where $\theta>0$ is the magnitude of the coupling strength and $x_j - x_i$ denotes a flow from neuron $j$ to $i$. The model is represented by

\begin{equation}
\label{eq:DML_4D}
\begin{aligned}
    ^C \mathcal{D}_{0}^\beta x_1 &= x_1^2(1-x_1) - y_1 + I +\theta(x_2 - x_1), \\
    ^C \mathcal{D}_{0}^\beta y_1 &= A e^{\alpha x_1} - \gamma y_1,\\
    ^C \mathcal{D}_{0}^\beta x_2 &= x_2^2(1-x_2) - y_2 + I +\theta(x_1 - x_2), \\
    ^C \mathcal{D}_{0}^\beta y_2 &= A e^{\alpha x_2} - \gamma y_2.\\
\end{aligned}
\end{equation}We employ the second coupling scheme following the work by Belykh {\em et al.}~\citep{BeDe05}. This coupling scheme is modeled by the sigmoidal function and looks like $\sigma \frac{v_s - x_i}{1+e^{-\lambda(x_j - q)}}$. Physically, this means the effect of neuron $j$ on $i$. Parameter $\sigma>0$ is the magnitude of the coupling strength, and $v_s$ is the reversal potential. For the synapse to be excitatory, the reversal potential $v_s$ should be greater than $x(t)$ for all $x$ and $t$. The sigmoidal function is given by $\zeta(x_j) = \frac{1}{1+e^{-\lambda(x_j - q)}}$ where the constant $\lambda>0$ is the slope and $q$ is the synaptic threshold. Somers and Kopell~\citep{SoKo93} called this coupling scheme a {\em fast threshold modulation}. The four-dimensional model coupled via the sigmoidal function looks like

\begin{equation}
\label{eq:dML_4D_chemical}
\begin{aligned}
    ^C \mathcal{D}_{0}^\beta x_1 &= x_1^2(1-x_1) - y_1 + I +\sigma\frac{v_s - x_1}{1+e^{-\lambda(x_2 - q)}}, \\
    ^C \mathcal{D}_{0}^\beta y_1 &= A e^{\alpha x_1} - \gamma y_1,\\
    ^C \mathcal{D}_{0}^\beta x_2 &= x_2^2(1-x_2) - y_2 + I +\sigma\frac{v_s - x_2}{1+e^{-\lambda(x_1 - q)}}, \\
    ^C \mathcal{D}_{0}^\beta y_2 &= A e^{\alpha x_2} - \gamma y_2.\\
\end{aligned}
\end{equation} 

In the next two sections, we unfold how the memory index $\beta$ induces important bifurcation patterns in the single-cell model~\eqref{eq:DML_2D} and the two-coupled models~\eqref{eq:DML_4D}, and~\eqref{eq:dML_4D_chemical}.

\section{Theoretical and numerical analysis of the two-dimensional single-cell dML model}
\label{sec:dML_2D}
We cover the technical details of~\eqref{eq:DML_2D} in this section by closely following the qualitative approach undertaken by Sharma {\em et al.}~\citep{ShMo23}. To start with, we first look into the equilibrium points of ~\eqref{eq:DML_2D} which can be computed from the transcendental equations
\begin{align}
\label{eq:fp_2d_1}
    x^2(1-x) - y + I &= 0, \\
\label{eq:fp_2d_2}
    Ae^{\alpha x} - \gamma y = 0,
\end{align}
by solving for $x$. We can write $I$ as a function of $x, y$ by shuffling~\eqref{eq:fp_2d_1} and writing as $I = K(x, y)$, where 
$$
K(x, y) = y - x^2(1-x).
$$
From~\eqref{eq:fp_2d_2} we have
\begin{align}
    y = y_\infty(x) = \frac{A e^{\alpha x}}{\gamma}. \nonumber
\end{align}
Using the above two equations we can write $I_\infty = K(x, y_\infty(x))$, leading us
\begin{align}
    I_\infty(x) = \frac{A}{\gamma} e^{\alpha x} - x^2(1 - x),
\end{align}
with
\begin{align}
\label{eq:dinftydx}
    \frac{dI_\infty(x)}{dx} &= \frac{\alpha A}{\gamma}e^{\alpha x} - x(2 - 3x), \\
    \frac{d^2I_\infty(x)}{dx^2} &= \frac{\alpha^2 A}{\gamma}e^{\alpha x} -2(1-3x), \\
    \frac{d^3I_\infty(x)}{dx^3} &= \frac{\alpha^3 A}{\gamma}e^{\alpha x} +6, \\
    &\vdots \\
    \frac{d^mI_\infty(x)}{dx^m} &= \frac{\alpha^m A}{\gamma}e^{\alpha x},\qquad m>3.
\end{align}
Throughout the paper, for numerics' purpose, we set 
\begin{align}
\label{eq:param}
    A = 0.0041, \qquad \alpha = 5.276, \qquad \rm{and} \qquad \gamma = 0.3.
\end{align}
Considering~\eqref{eq:param} we plot $I_\infty(x)$ as a function of $x$ in Fig.~\ref{fig:Iinfty}.

\begin{figure}[h]
    \centering
    \includegraphics[width=0.5\linewidth]{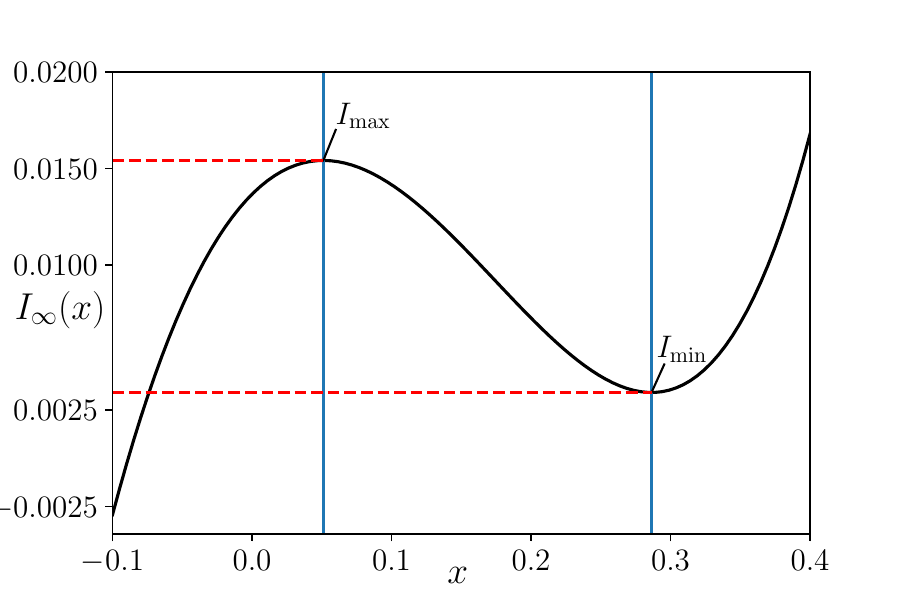}
    \caption{$I_\infty(x)$ as a function of $x$ with parameters~\eqref{eq:param}. The two extrema are denoted by $I_{\rm max}$ and $I_{\rm min}$.}
    \label{fig:Iinfty}
\end{figure}

Note that 
\begin{itemize}
    \item[i)] $I_\infty(x)$ is $C^k$ smooth,
    \item[ii)] $\lim_{x \to -\infty} I_\infty(x) = -\infty$, $\lim_{x \to \infty} I_\infty(x) = \infty$, and
    \item[iii)] $I_\infty (x)$ has two extrema, one maximum at $x_{\rm max}$ and one minimum at $x_{\rm min}$.
\end{itemize}
Let us consider $I_{\rm max} = I_{\infty}(x_{\rm max})$ and $I_{\rm min} = I_{\infty}(x_{\rm min})$. It is observed that $I_{\infty}$ is monotonically increasing in the intervals $(-\infty, x_{\rm max}]$ and $[x_{\rm min}, \infty)$, whereas is monotonically decreasing within the interval $(x_{\rm max}, x_{\rm min})$. For the parameters~\eqref{eq:param}, we report a table of $x$, $I_\infty(x)$ and $\frac{dI_\infty(x)}{dx}$ values in Table~\ref{tab:ITable}. Numerics agree with the analytical findings that $\frac{dI_\infty}{dx} <0$ for $I_\infty \in (I_{\rm max}, I_{\rm min})$, pointing out that we are on the right path. We have $I_\infty(x)$ monotonically decreasing in the interval $(x_{\rm max}, x_{\rm \min})$. Following~\eqref{eq:param}, we have $I_{\rm max} \approx 0.01542$ and $I_{\rm min} \approx 0.0034$. Values of $x_{\max} \approx 0.05114$ and $x_{\min} \approx 0.28639$ were numerically computed by solving $\frac{dI_\infty}{dx} = 0$ from~\eqref{eq:dinftydx}. This is done by utilizing \texttt{Python}'s \texttt{fsolve()} function from the \texttt{scipy.optimize} suite, which is an optimization and root-finding package provided by \texttt{SciPy}.  
\begin{table}[h]
        \centering
        \begin{tabular}{|c|c|c|}
        \hline
             $x$ & $I_{\infty}(x)$ & $\frac{dI_{\infty}(x)}{dx}$ \\ \hline \hline
             $x_{\rm max} = 0.051143193209885154$& $I_{\rm max} = 0.015417976156715866$  & 0 \\
             $0.05351939825528394$ & $0.015414622120595075$ & $-0.0028148833020992196$ \\
             $0.05589560330068273$ & $0.015404637416021267$ & $-0.005580857722584676$ \\
             $\vdots$ & $\vdots$ & $\vdots$ \\
             $0.2151013413424015$ & $0.006197891516995346$ & $-0.06709282841464406$ \\
             $0.21747754638780026$ & $0.00603983091028578$ & $-0.06593178042834522$ \\
             $\vdots$ & $\vdots$ & $\vdots$\\
             $0.2816350826135675$ & $0.0034130953188769644$ & $-0.006683882435937927$\\
             $0.2840112876589663$ & $0.003401116673691529$ & $-0.0033842365907062744$ \\
             $x_{\rm min} = 0.2863874927043651$& $I_{\rm min} = 0.003397079040195275$ & $0$ \\
             \hline
        \end{tabular}
        \caption{We have $\frac{dI_\infty}{dx}<0$ for $I_\infty \in (I_{\max}, I_{\min})$. Parameters are set according to~\eqref{eq:param}.}
        \label{tab:ITable}
    \end{table}
    
Thus, there exist three branches of the equilibrium points and 
\begin{itemize}
    \item[i)] if $I < I_{\rm min}$ or $I > I_{\rm max}$, \eqref{eq:DML_2D} will have a unique equilibrium point (See Fig.~\ref{fig:eqPoints}-(a), (b)),
    \item[ii)] if $I = I_{\rm min}$ or $I = I_{\rm max}$, \eqref{eq:DML_2D} will have two equilibrium points (See Fig.~\ref{fig:eqPoints}-(c), (d)),
    \item[iii)] if $I \in (I_{\rm min}, I_{\rm max})$, it has three equilibrium points (See Fig.~\ref{fig:eqPoints}-(e)). 
\end{itemize}
We now report some examples based on the three branches of the equilibrium points, in Fig.~\ref{fig:eqPoints}. An intersection of the nullclines represents an equilibrium point of~\eqref{eq:DML_2D}. The $x$-nullcline is colored in blue and the $y$-nullcline is colored in red. Again we use \texttt{Python}'s \texttt{fsolve()} to numerically compute the equilibrium points for the parameter set~\eqref{eq:param} and a fixed $I$ in each of the five panels (a)--(e) to first solve for $x^*$ from
\begin{align}
\label{eq:eqPointSoln}
{x^*}^2(1-x^*) - \frac{Ae^{\alpha x^*}}{\gamma} + I = 0
\end{align}
and then substituting $x^*$ in $y^* = y_\infty(x^*)= \frac{Ae^{\alpha x^*}}{\gamma}$. In panel (a), we consider $I = 0.0001<I_{\min}$. It has a unique equilibrium point at $(x^*, y^*) \approx (-0.08827, 0.00858)$. In panel (b), we consider $I = 0.019> I_{\max}$. It also has a unique equilibrium point at $(x^*, y^*) \approx (0.40772, 0.11746)$. In panel (c), we have $I = I_{\min}$. In this case, we will have two equilibrium points at $(x^*, y^*) \approx (-0.07386, 0.00926)$ and $(x^*, y^*) \approx (0.28639, 0.06193)$. In panel (d) we consider the case $I = I_{\max}$, again in which case we have two equilibrium points at $(x^*, y^*) \approx (0.05114, 0.0179)$ and at $(x^*, y^*) \approx (0.39491, 0.109785)$. Finally, in panel (e) we have $I = 0.011 \in (I_{\min}, I_{\max})$, in which case there exist three equilibrium points at $(x^*, y^*) \approx (-0.027865, 0.0118)$, at $(x^*, y^*) \approx (0.15041, 0.03022)$, and at $(x^*, y^*) \approx (0.37528, 0.09898)$.

\begin{figure}[h]
\begin{center}
\begin{tabular}{ccc}
  \includegraphics[scale=0.25]{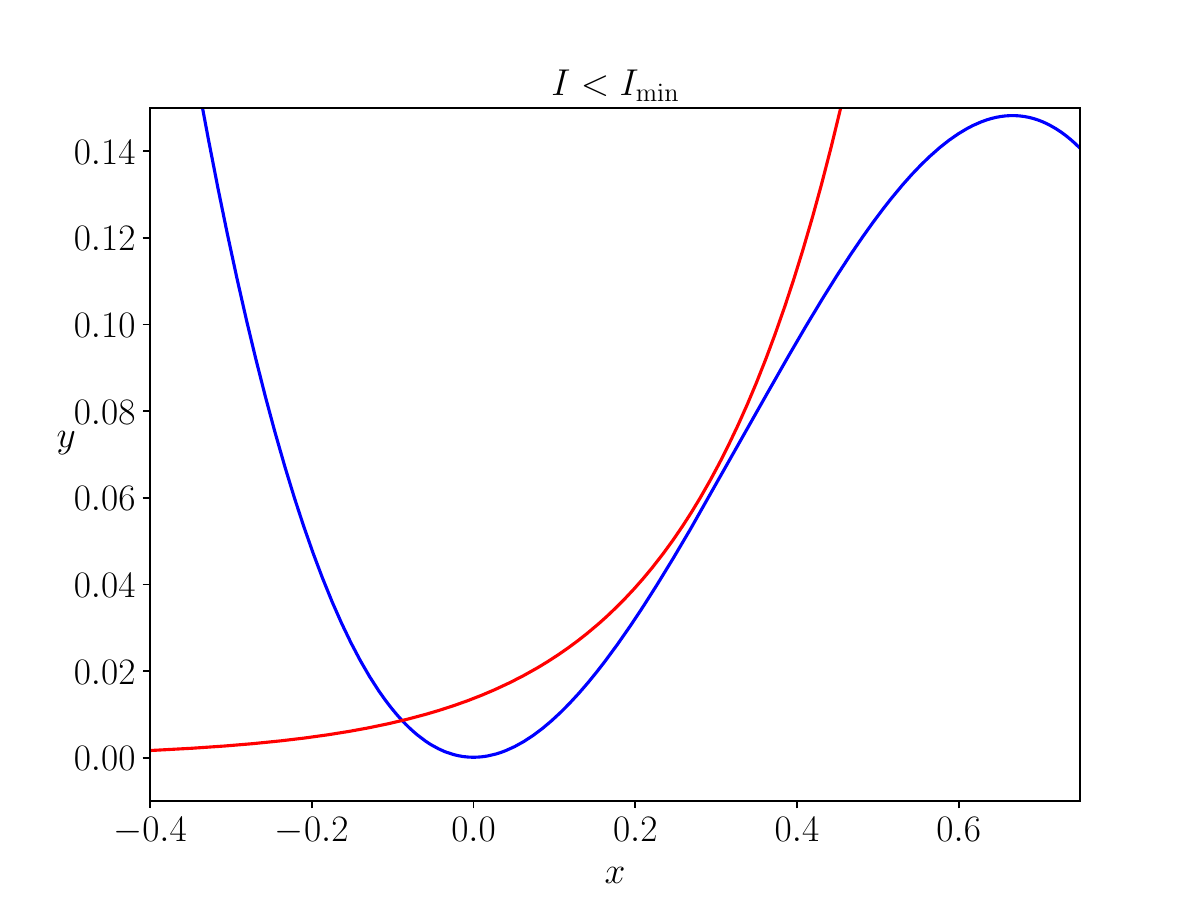} &  \includegraphics[scale=0.25]{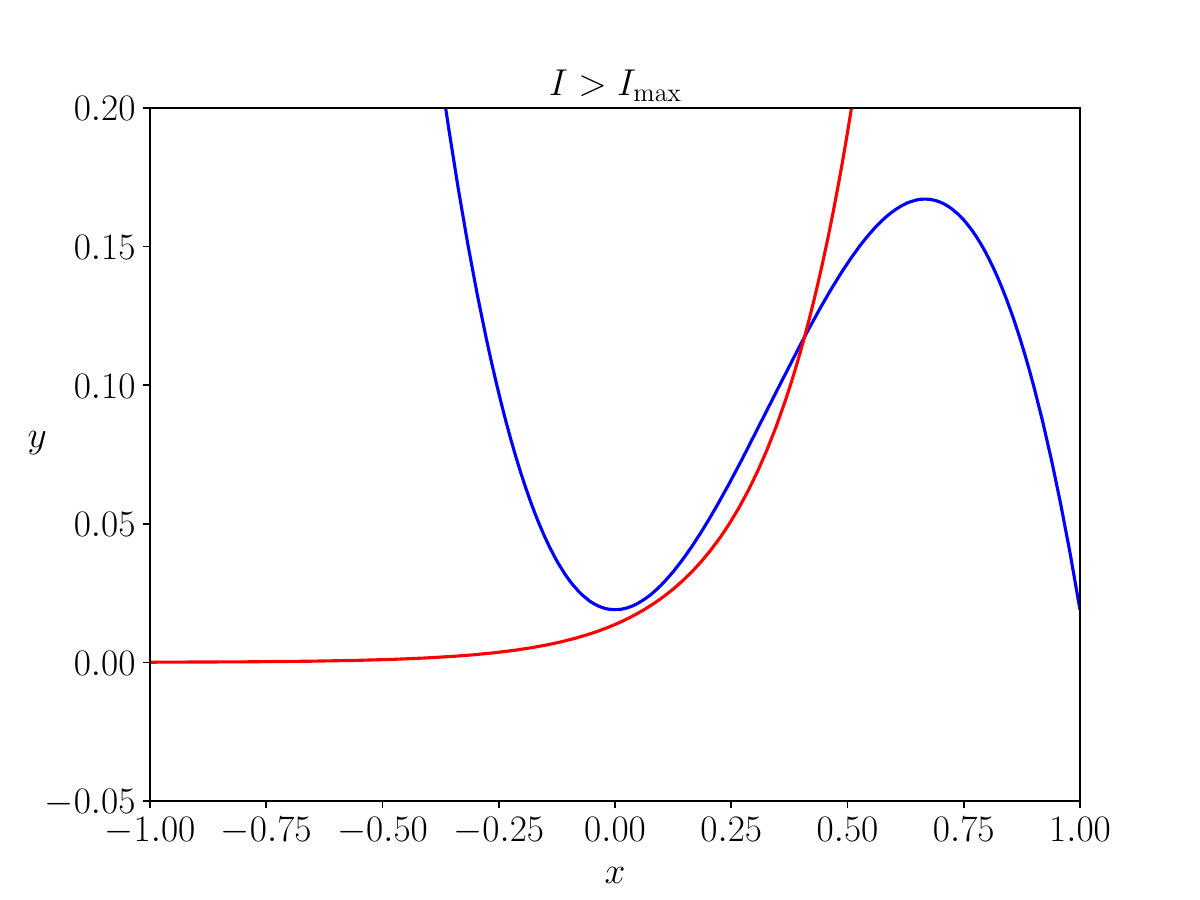} & \includegraphics[scale=0.25]{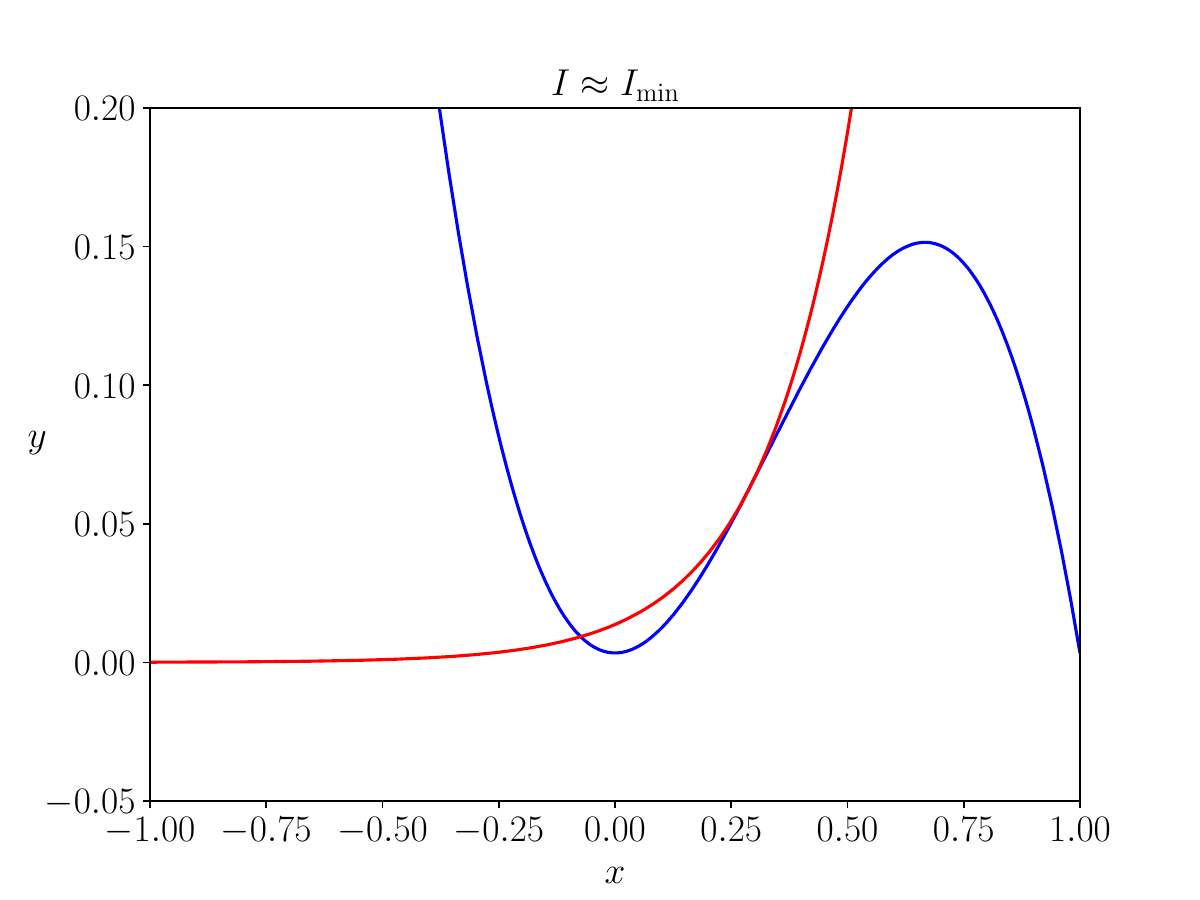} \\
(a)  & (b) &(c) \\[3pt]
\includegraphics[scale=0.25]{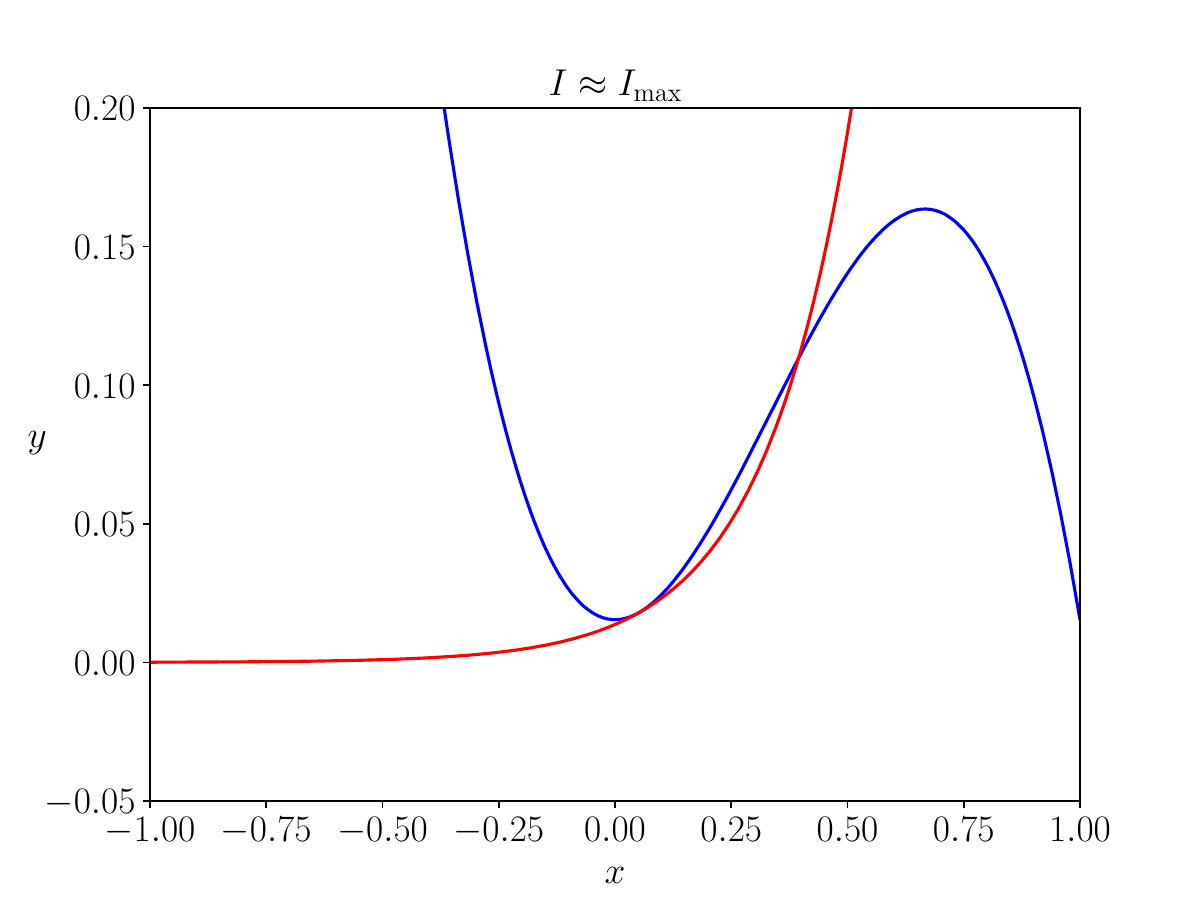} &  \includegraphics[scale=0.25]{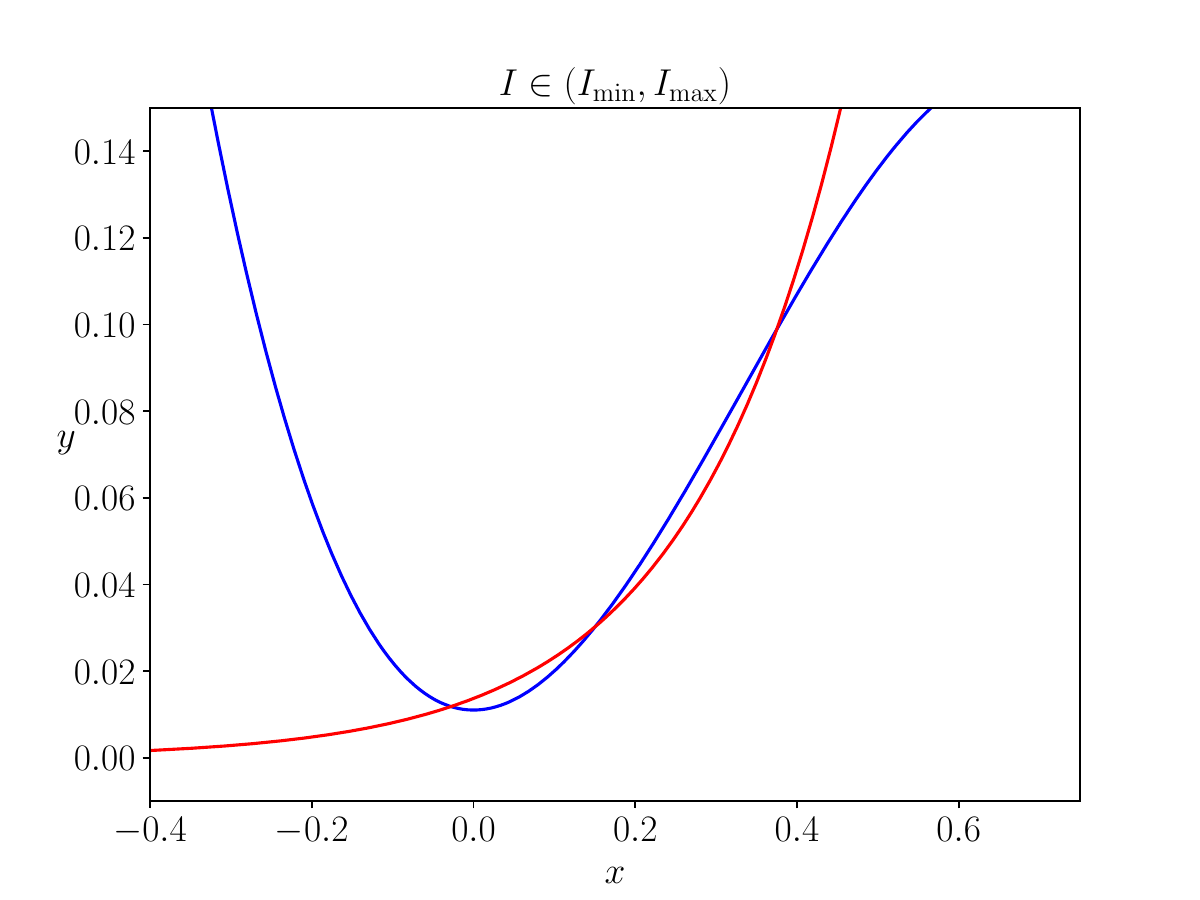} \\
(d) &  (e) \\[3pt]
\end{tabular}
\end{center}
\caption{Three different branches of the equilibrium points with $I$ set as (a) $0.0001< I_{\min}$, (b) $0.019 > I_{\max}$, (c) $I_{\min}$, (d) $I_{\max}$, (e) $0.11 \in (I_{\min}, I_{\max})$. Other parameters are set as~\eqref{eq:param}. Panels (a) and (b) have a unique equilibrium, panels (c) and (d) have two equilibriums, and panel (e) has three equilibriums.}
\label{fig:eqPoints}
\end{figure}

The Jacobian matrix for~\eqref{eq:DML_2D} at an equilibrium point $(x^*, y^*)$ is 
\begin{align}
    J = \begin{bmatrix}
        x^*(2-3x^*) & -1 \\ \alpha A e^{\alpha x^*} & -\gamma
    \end{bmatrix},
\end{align}
with the trace and determinant given by
\begin{align}
\label{eq:trace2D}
    \tau(x^*) &= x^*(2-3x^*)-\gamma , \\
\label{eq:deter2D}
    \delta(x^*) &= -\gamma x^*(2-3x^*) +\alpha A e^{\alpha x^*},
\end{align}
respectively. Next, we will perform a qualitative analysis of the equilibrium point to establish the necessary and sufficient conditions on $x^*$ that determine its stability.

\begin{theorem}
\label{thm:EqStab2D}
    Suppose 
    \begin{enumerate}
        \item[i)] $x^*(2-3x^*)-\gamma>0 $, and
        \item[ii)] $-\gamma x^*(2-3x^*) +\alpha A e^{\alpha x^*} < 2\sqrt{-\gamma - x^*(2-3x^*)} \cos(\frac{\beta \pi}{2})$.
    \end{enumerate}
    Then an equilibrium point $(x^*, y^*)$ of~\eqref{eq:DML_2D} is asymptotically stable.
\end{theorem}

\begin{proof}
    Substituting the values of the trace~\eqref{eq:trace2D} and the determinant~\eqref{eq:deter2D} in ~\eqref{eq:StabCond}, the necessary and sufficient conditions for $(x^*, y^*)$ to be asymptotically stable are easily established. The proof is a straightforward algebra problem.
\end{proof}

\begin{theorem}
\label{thm:EqUnstab2D}
    Suppose $I \in (I_{\rm min}, I_{\rm max})$. Then this branch of equilibrium points is completely unstable.
\end{theorem}

\begin{proof}
We have seen that $I_{\infty}(x)$ is monotonically decreasing in the range $(I_{\rm min}, I_{\rm max})$. This means $\frac{dI_{\rm \infty}(x^*)}{dx}<0$, i.e, 
    $$
    \frac{\alpha^2 A}{\gamma}e^{\alpha x^*} -2(1-3x^*) <0.
    $$
    This can be easily rearranged to realise that $\delta (x^*)<0$, meaning an equilibrium point $(x^*, y^*)$ of~\eqref{eq:DML_2D} is unstable in the branch $(I_{\rm min}, I_{\rm max})$.
\end{proof}

From Theorem~\ref{thm:EqUnstab2D} we can directly see that $\delta(x^*)<0$ implies one of the two eigenvalues is positive and the other negative, meaning the equilibrium point on this branch is a saddle, irrespective of the fractional order $\beta \in (0, 1]$.

\begin{theorem}
\label{thm:saddleNode}
 Suppose $I = I_{\rm min}$ or $I = I_{\rm max}$. Then~\eqref{eq:DML_2D} has a {\em saddle-node} bifurcation.    
\end{theorem}

\begin{proof}
    At the branch of the equilibrium point where $I = I_{\rm min}$ or $I = I_{\rm max}$, we have $\frac{dI_\infty(x^*)}{dx} = 0$, indicating that $\delta(x^*) = 0$. This means one of the eigenvalues is zero and two out of the three equilibrium points collide and annihilate into one, giving rise to a saddle-node bifurcation.
\end{proof}

The above theorem holds irrespective of the fractional order $\beta \in (0, 1]$. Saddle-node bifurcation occurs on the transition from the branch containing three equilibrium points to the branch containing two. This can happen in two different directions. For example when the value of $I$ is reduced from being in the range $(I_{\min}, I_{\max})$ to $I_{\min}$ (See Fig.~\ref{fig:eqPoints}-(e) transitioning to Fig.~\ref{fig:eqPoints}-(c)) or the value of $I$ is increased from being in the range $(I_{\min}, I_{\max})$ to $I_{\max}$ (See Fig.~\ref{fig:eqPoints}-(d) transitioning to Fig.~\ref{fig:eqPoints}-(c)). In either of the cases $\frac{dI}{dx}$ becomes $0$ from being negative, giving rise to a saddle-node bifurcation (see the top or the bottom end of Table~\ref{tab:ITable}). A typical codimension-one bifurcation diagram for varying $I$ showing both saddle-node bifurcations and Hopf bifurcations was shown by Fatoyinbo {\em et al.}~\citep{FaMu22} in their work from 2022.

\begin{theorem}
\label{thm:betaStar2D}
    Suppose $I<I_{\rm min}$ or $I > I_{\rm max}$. Then 
    \begin{itemize}
        \item[i)] the stability of an equilibrium point of~\eqref{eq:DML_2D} depends on the sign of $\tau(x^*)$,
        \item[ii)] for $\tau(x^*)\ge 0$ the equilibrium is asymptotically stable if and only if the order $\beta < \beta^* = \frac{2}{\pi}\cos^{-1}\bigg( \min \bigg( 1, \frac{-\gamma + x^*(2-3x^*)}{2\sqrt{\alpha Ae^{\alpha x^*} - \gamma x^* (2-3x^*)}} \bigg) \bigg)$.
    \end{itemize} 
\end{theorem}

\begin{proof}
    We have noticed that $I_{\infty}(x)$ is monotonically increasing for $I<I_{\rm min}$ and $I> I_{\rm max}.$ This means $\frac{dI_{\infty}(x^*)}{dx}>0$, leading to $\delta(x^*) >0$. In that case, the stability of the equilibrium point depends on the sign of $\tau(x^*)$~\citep{MoKa24}.
    \begin{itemize}
        \item[i)] If $\tau(x^*) \le 0$, the equilibrium is asymptotically stable for any $\beta \in (0, 1]$.
        \item[ii)] If however, $\tau(x^*)\ge 0$, the stability of the equilibrium depends on two subcases
        \begin{itemize}
            \item[a)] Suppose $\frac{\tau(x^*)}{2\sqrt{\delta(x^*)}}\ge 1$. From $\tau(x^*) < 2\sqrt{\delta(x^*)}\cos \big(\frac{\beta \pi}{2} \big)$ leads to the inequality $\cos \big(\frac{\beta \pi}{2}\big)>1$, which is nonsense. Thus, the equilibrium is unstable.
            \item[b)] Suppose $\frac{\tau(x^*)}{2\sqrt{\delta(x^*)}} < 1$, then $\tau(x^*) < 2\sqrt{\delta(x^*)}\cos \big(\frac{\beta \pi}{2} \big)$ leads to an inequality involving the critical value of the order $\beta <\frac{2}{\pi} \cos^{-1} \bigg(\frac{\tau(x^*)}{2\sqrt{\delta(x^*)}} \bigg)$. Substituting the values of $\tau(x^*)$ and $\delta(x^*)$ gives us
            \begin{align}
            \label{eq:betaStar}
            \beta^* = \frac{2}{\pi}\cos^{-1}\bigg( \min \bigg( 1, \frac{-\gamma + x^*(2-3x^*)}{2\sqrt{\alpha Ae^{\alpha x^*} - \gamma x^* (2-3x^*)}} \bigg) \bigg).
            \end{align}
        \end{itemize}
    \end{itemize}
\end{proof}
In the above case $\beta^*$ is a critical value at which~\eqref{eq:DML_2D} undergoes a Hopf bifurcation as $\beta$ is increased. The stable equilibrium loses its stability as $\beta$ is increased beyond $\beta^*$ and eventually a stable limit cycle appears. We will now run numerical simulations to validate our theoretical analysis of~\eqref{eq:DML_2D}. Our main tool is the \texttt{FDEsolver()} function from the \texttt{FdeSolver.jl} (v1.0.8) package, see Appendix~\ref{app:Software}. Function \texttt{F} is first set as the right hand side of~\eqref{eq:DML_2D}, time span \texttt{tSpan} as $[0, 6000]$, and the initial conditions as $[0.1, 0.1]$, i.e., $x(0) = 0.1$ and $y(0) = 0.1$. Time step \texttt{h} is fixed to $0.01$. The parameters $A$, $\alpha$, $\gamma$ follow~\eqref{eq:param}. The main bifurcation parameters are $I$ and $\beta$ which we vary according to our research question. All the parameters are collected in the vector \texttt{par}.

Fig.~\ref{fig:pp} is a collection of phase portraits and their corresponding time series plots for the voltage $x(t)$ and the recovery-like variable $y(t)$, for a $I>I_{\max}$. This means there is a unique equilibrium point. Particularly, we have set $I = 0.019$. In each panel, the left hand side shows the phase portrait with the $x$- and $y$-nullclines, with the equilibrium point denoted by their unique intersection. For the selected \texttt{tSpan} and \texttt{h} values, we will have $6 \times 10^5$ data points. We discard the first $10^5$ data points in the phase portrait to remove the transient dynamics. On the right hand side, we plot the time series of both $x(t)$ and $y(t)$ and plot all the data points for the time series. We portray how the dynamics of~\eqref{eq:DML_2D} varies with changing $\beta$ from $0.9$ to $1$, transitioning through $\beta^*$. We have previously computed $x^* \approx 0.40772$ for $I = 0.019$ and the other parameters set as~\eqref{eq:param}. Now substituting these in~\eqref{eq:trace2D} and~\eqref{eq:deter2D} we get $\tau(x^*) \approx 0.01673$ and $\delta(x^*) \approx 0.0909$, meaning $\beta^* \approx 0.98233$ (using~\eqref{eq:betaStar}). Fig.~\ref{fig:pp}-(a) to (e) indicate $\beta <\beta^*$ where we observe that the dynamics always converges to the stable equilibrium. Note that for a higher value of $\beta$ when $\beta <\beta^*$, the time required to converge to the equilibrium increases, i.e, in panel (c) for $\beta = 0.94$ the system converges almost immediately, however for $\beta =0.98$ (approaching $\beta^*)$, the system converges roughly near $t=3500$. At $\beta=\beta^*$ (panel (f)), the equilibrium loses its stability, and both the voltage and recovery-like variables exhibit tonic spiking with small amplitudes. As $\beta$ is further increased (panels (g),(h)), we observe that the equilibrium point becomes unstable and a stable limit cycle appears indicating Hopf bifurcation. We also notice that the dynamical variables exhibit tonic spiking with higher amplitudes as $\beta$ is further increased to $1$ from $\beta^*$. Thus, the numerics satisfy the theory put forward in this paper. A similar behavior is reported in a fractional-order commensurate system of Wilson-Cowan neurons in the work by Mondal {\em et al.}~\citep{MoKa24}.

\begin{figure*}
\begin{center}
\begin{tabular}{cc}
  \includegraphics[scale=0.25]{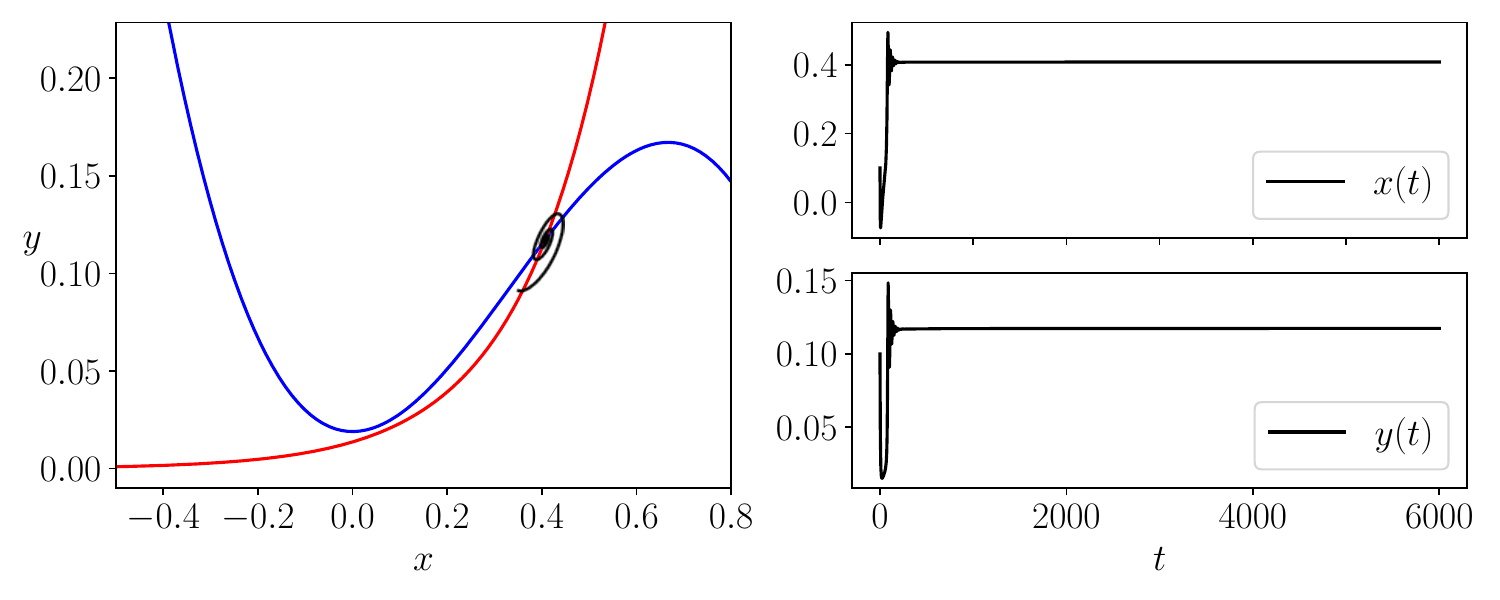} &  \includegraphics[scale=0.25]{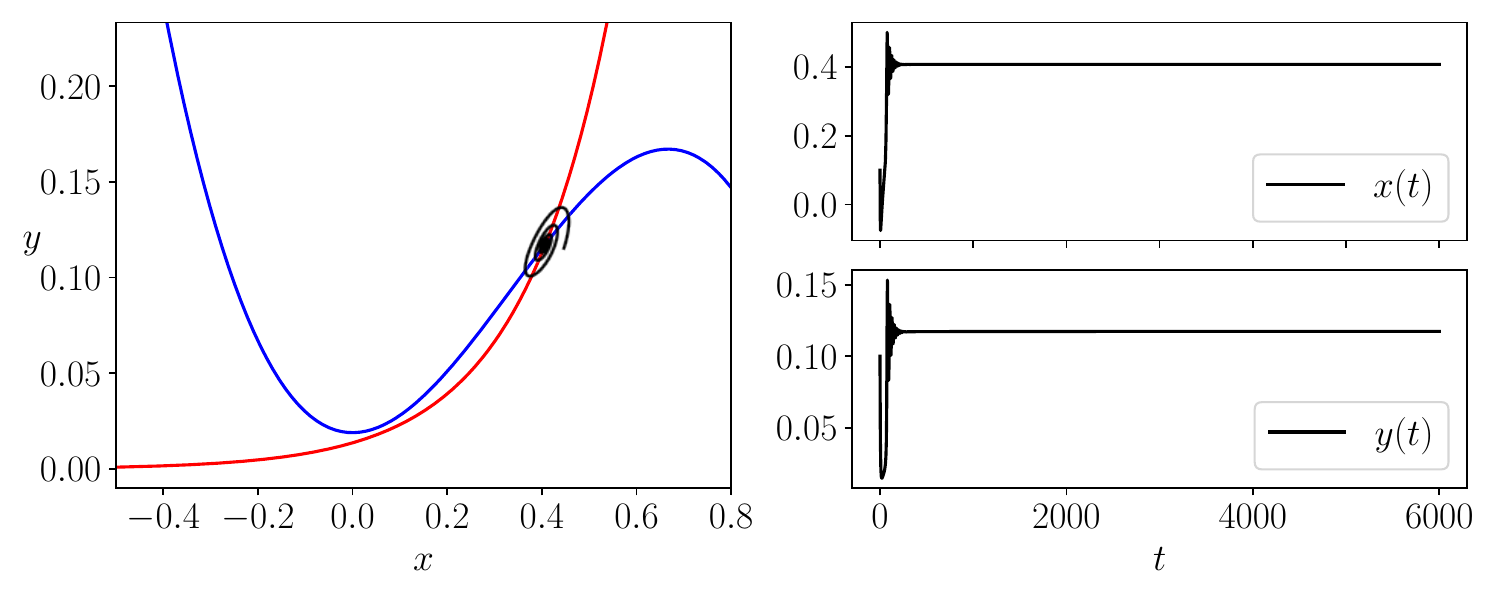} \\
(a) $\beta = 0.9$ & (b) $\beta = 0.92$ \\[3pt]
\includegraphics[scale=0.25]{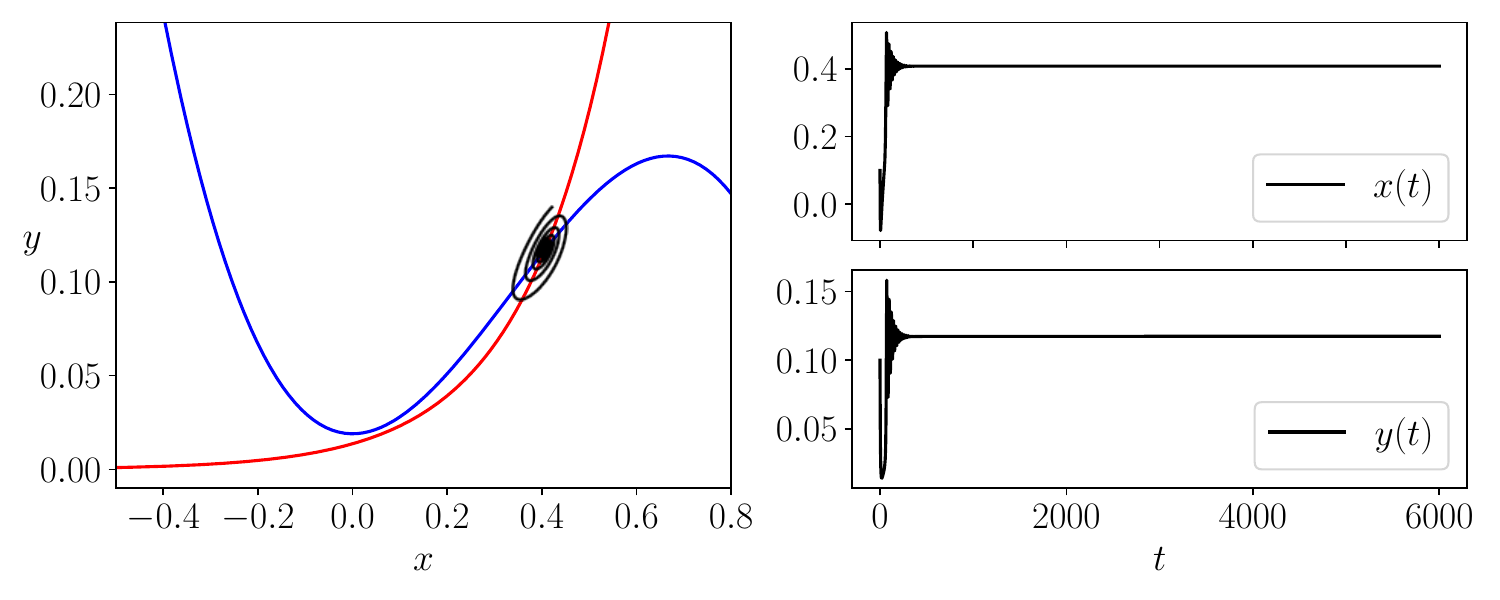} &  \includegraphics[scale=0.25]{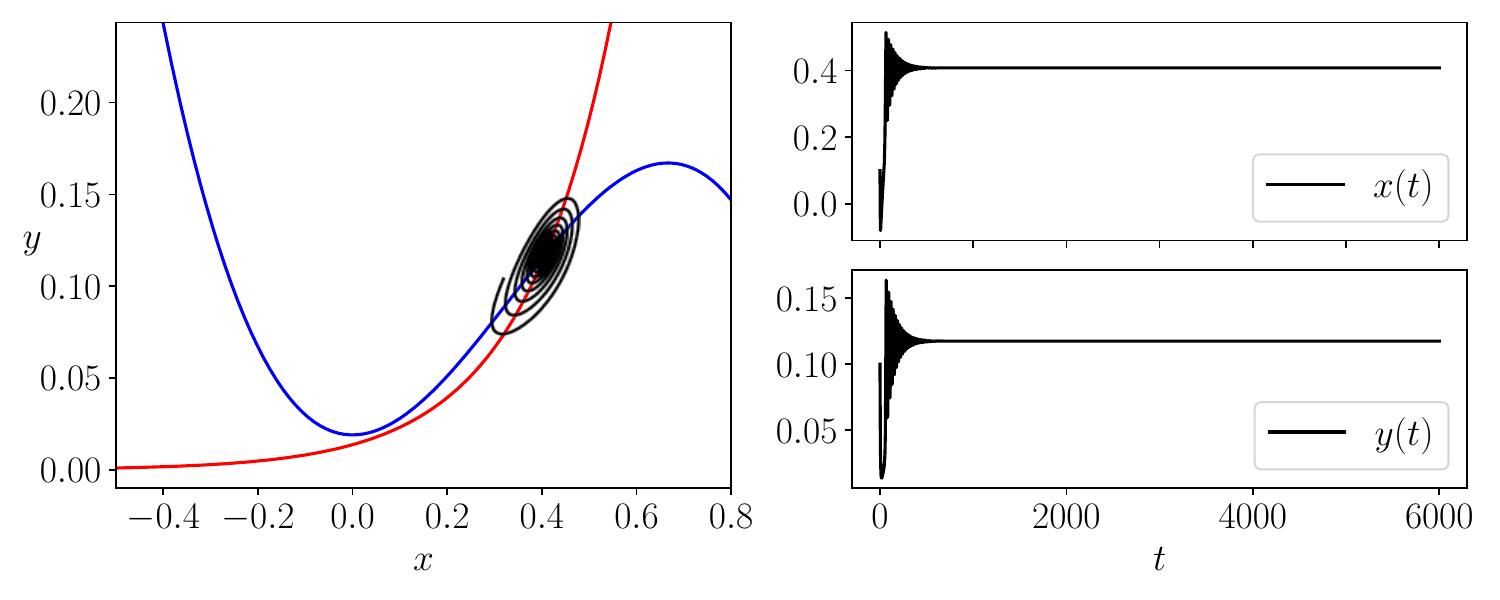} \\
(c) $\beta = 0.94$ &  (d) $\beta = 0.96$\\[3pt]
\includegraphics[scale=0.25]{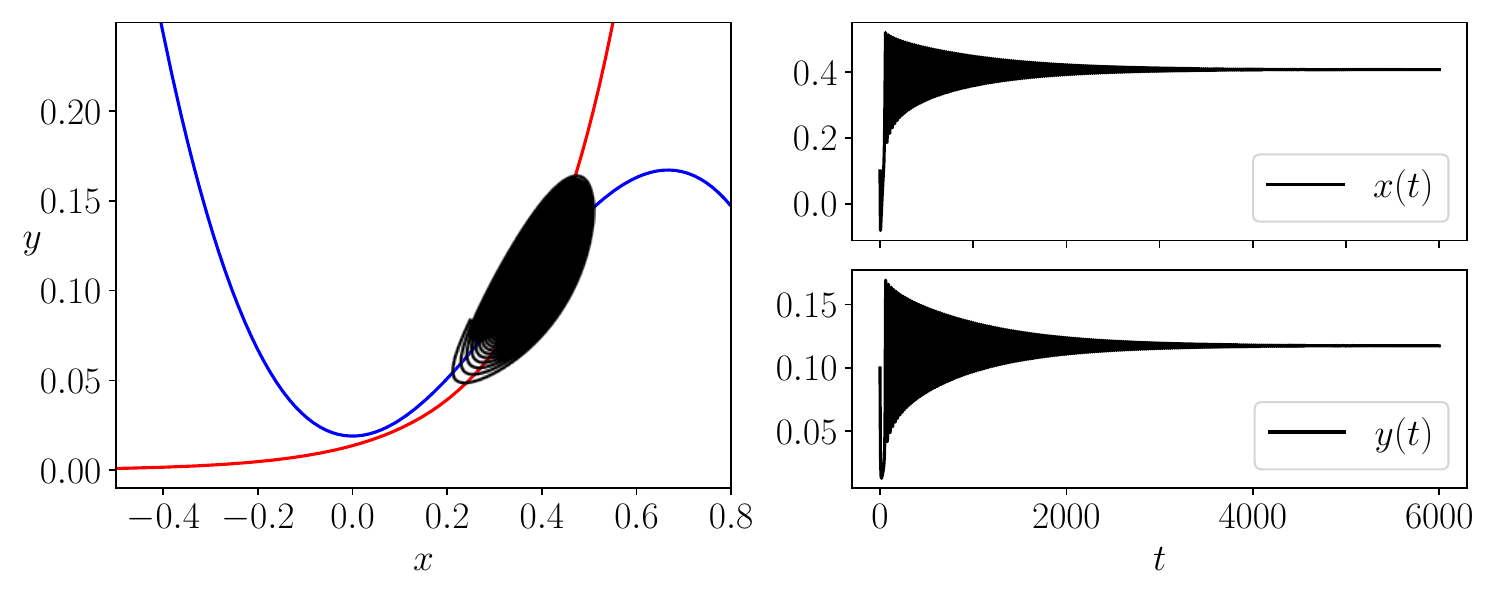} &  \includegraphics[scale=0.25]{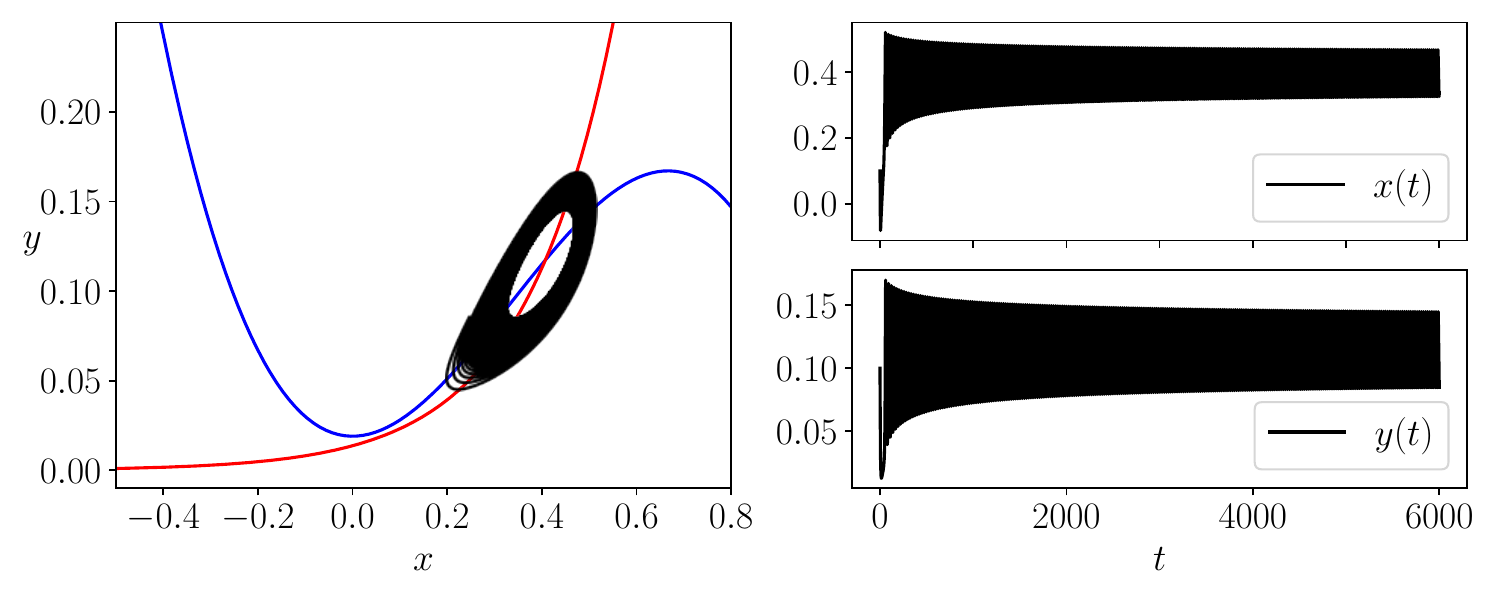} \\
(e) $\beta = 0.98$ & (f) $\beta = \beta^*$ \\[3pt]
\includegraphics[scale=0.25]{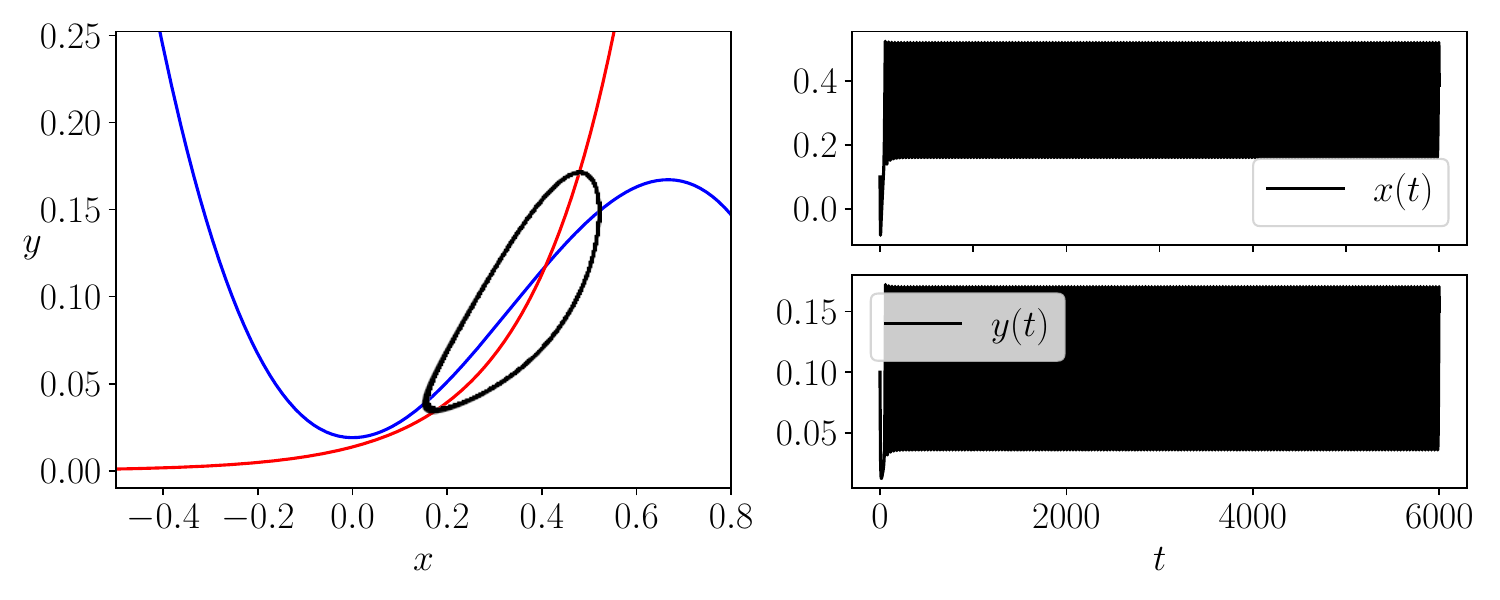} &  \includegraphics[scale=0.25]{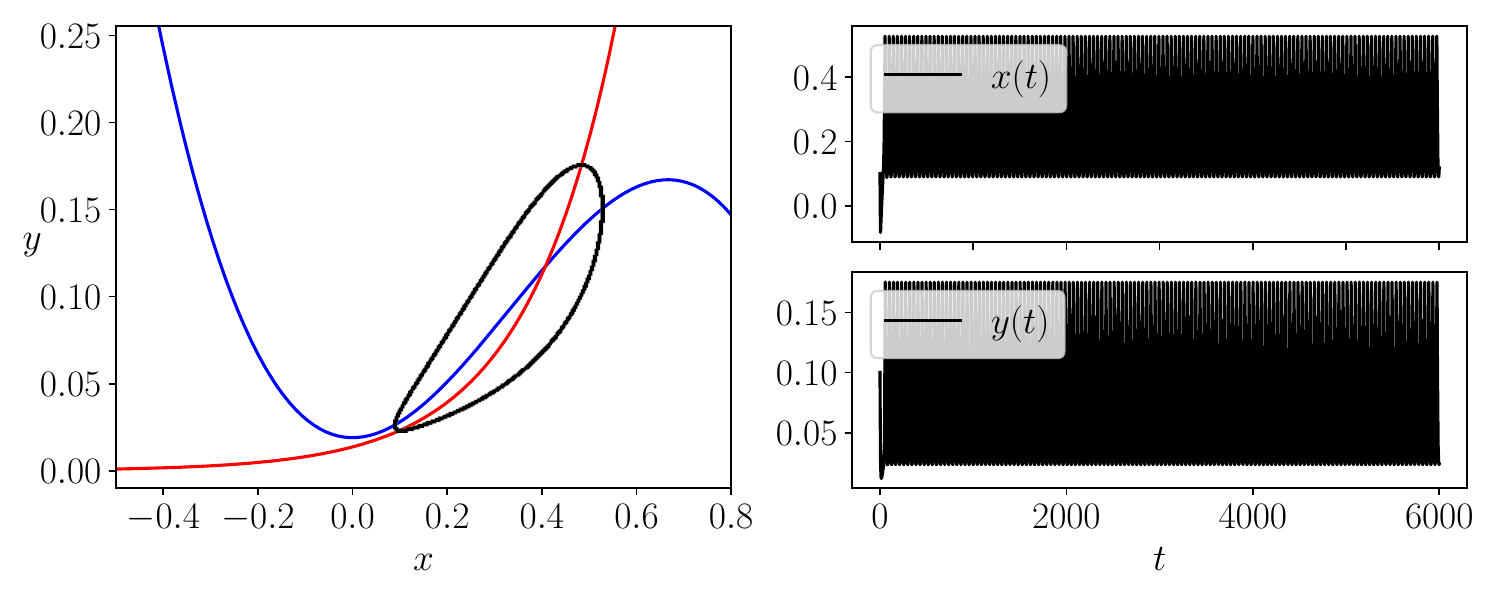}\\
(g) $\beta = 0.99$ & (h) $\beta = 1$\\[3pt]
\end{tabular}
\end{center}
\caption{Phase portraits and time series with varying $\beta$ for~\eqref{eq:DML_2D}. We have set $I = 0.019>I_{\max}$. The other parameters are set as~\eqref{eq:param}. This gives $\beta^* \approx 0.98233$. We notice that~\eqref{eq:DML_2D} converges to a stable equilibrium for $\beta <\beta^*$. As $\beta \ge \beta^*$, the equilibrium loses stability through a Hopf bifurcation and a stable limit cycle appears (showing tonic spiking in the dynamical variables).}
\label{fig:pp}
\end{figure*}

Following this, it is imperative to display a bifurcation diagram of the voltage $x$ with $\beta$ as the bifurcation parameter at a fixed $I$ with the other parameters set as~\eqref{eq:param}. We do this in Fig.~\ref{fig:BIF_I_2D} for two different values of $I$: (a) $I = 0.019 > I_{\max}$, (b) $I = 0.022 > I_{\max}$ with $\beta$ in the range $[0.9, 1]$. We will additionally have $\beta^* \approx 0.98772$ for $I = 0.022$. The bifurcation diagrams are computed by a continuation procedure. We start with $\beta = 1$ with initial conditions set to $x(0) = 0.1$ and $y(0) = 0.1$. Then we run the simulation using \texttt{FDEsolver()} for \texttt{tSpan}$=[0, 6000]$. After that we decrease the value of $\beta$ from $1$ and use the $x(6000)$ and $y(6000)$ values generated from the last simulation as the initial conditions for the current simulation. At each $\beta$, we plot the last $500$ data points of $x$. We do this for the whole range of $\beta$ values in both panels. Vertical blue lines in both panels indicate $\beta^*$. We see that for $\beta <\beta^*$ the dynamics converges to the stable equilibrium point. Then $\beta$ crosses $\beta^*$ through a Hopf bifurcation giving rise to a stable limit cycle and the equilibrium point becomes unstable. In Fig.~\ref{fig:Hopf2D} we show the stable and the unstable regions of the equilibrium point $(x^*, y^*)$ on the $(I, \beta)$-plane (with $I$ varied in the range $[0.016, 0.03]$), separated by the Hopf bifurcation curve $\beta^*$ denoted in blue. The broken vertical lines indicate the values $I = 0.019, 0.022$. Below the Hopf curve, the equilibrium is asymptotically stable, and above the curve the equilibrium is unstable for the selected range of $I$ values. Finally, we plot a series of bifurcation diagrams from the perspective of fixing $\beta$ to three different values and varying $I$ in the range $[-0.015, 0.05]$ in Fig.~\ref{fig:BIF_beta_2D}. The simulation technique is kept the same as in Fig.~\ref{fig:BIF_I_2D} and three different cases of $\beta$ are displayed: (a) $\beta = 0.99$, (b) $\beta \approx 0.98233$, and (c) $\beta = 0.97$. We see that as $\beta$ is decreased, the system~\eqref{eq:DML_2D} stabilizes. The vertical broken line indicates the value of $I = 0.019$. Panel (b) is displayed for a $\beta$ value which is close to $\beta^*$ computed from $I = 0.019$.  

\begin{figure}[h]
\begin{center}
\begin{tabular}{c}
  \includegraphics[scale=0.4]{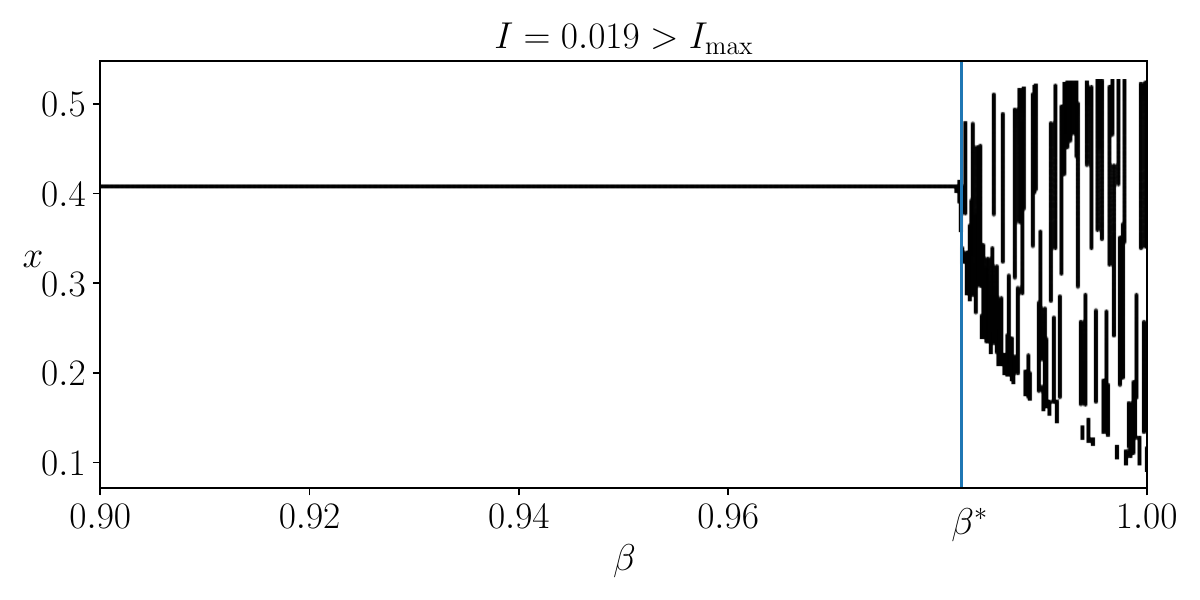} \\
(a) $\beta^* \approx 0.98233$ \\[3pt]
\includegraphics[scale=0.4]{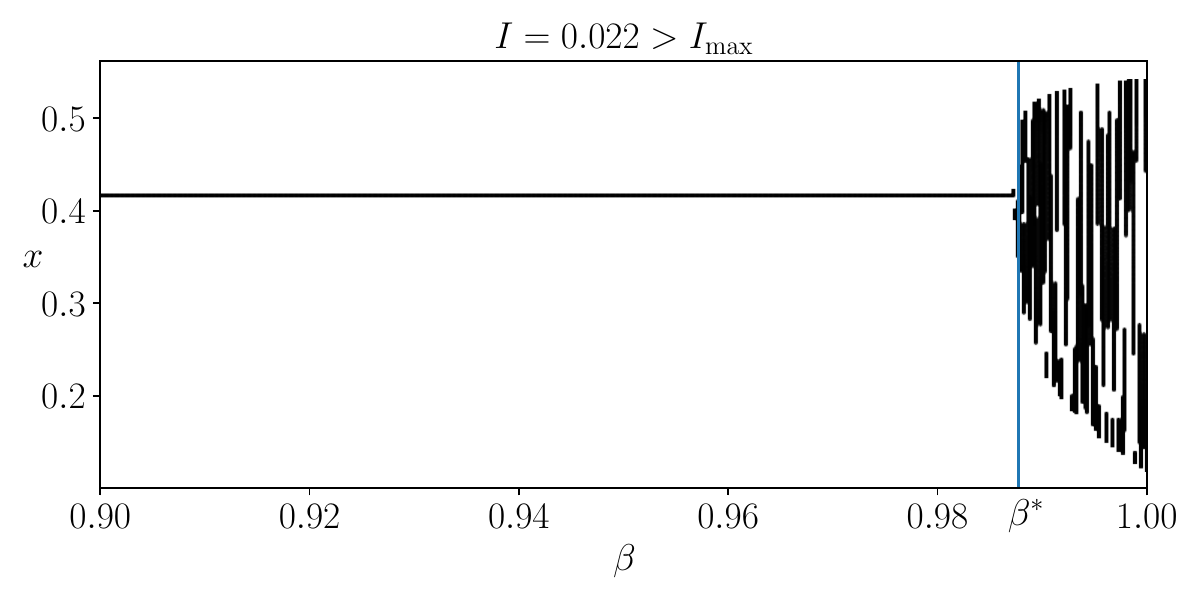} \\
(b) $\beta^* \approx 0.98772$ \\[3pt]
\end{tabular}
\end{center}
\caption{Bifurcation diagrams with varying $\beta$ for~\eqref{eq:DML_2D} at $I > I_{\rm max}$. Specifically we have (a) $I=0.019$ whose corresponding phase portraits are displayed in Fig.~\ref{fig:pp} and (b) $I = 0.022$. Other parameters are set following~\eqref{eq:param}. The vertical blue lines indicate the $\beta^*$ for each of the cases where a Hopf bifurcation occurs. The $\beta^*$ values for these cases are approximately (a) $0.98233$ and (b) $0.98772$.}
\label{fig:BIF_I_2D}
\end{figure}

\begin{figure}
    \centering
    \includegraphics[width=0.5\linewidth]{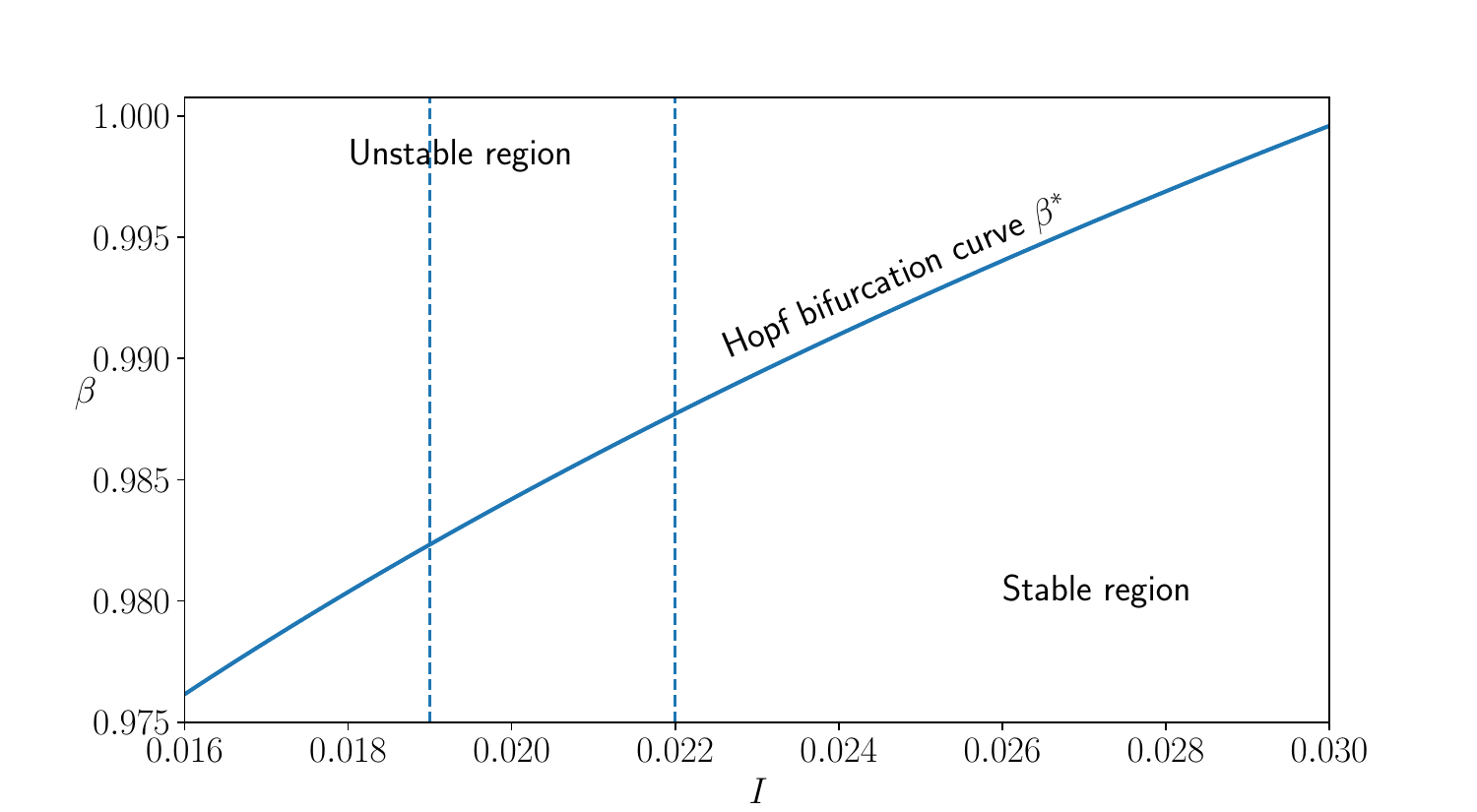}
    \caption{Hopf bifurcation curve $\beta^*$ for~\eqref{eq:DML_2D} separating the stable and the unstable regions of the equilibrium point $(x^*, y^*)$ on the $(I, \beta)$-plane. The vertical broken lines correspond to $I = 0.019$ and $I=0.022$. The system undergoes a Hopf bifurcation when $\beta$ crosses the curve from the bottom for a particular $I$ value. Other parameter values are set as~\eqref{eq:param}.}
    \label{fig:Hopf2D}
\end{figure}

\begin{figure}[h]
\begin{center}
\begin{tabular}{cc}
  \includegraphics[scale=0.3]{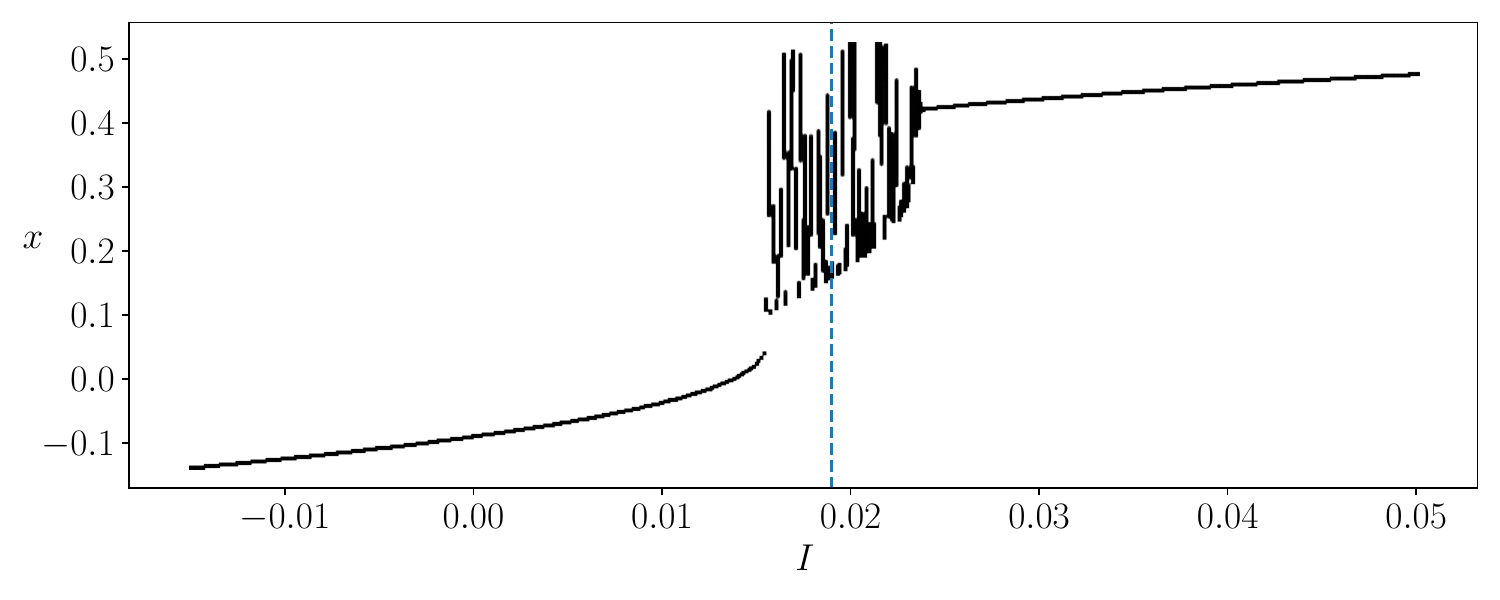} & \includegraphics[scale=0.3]{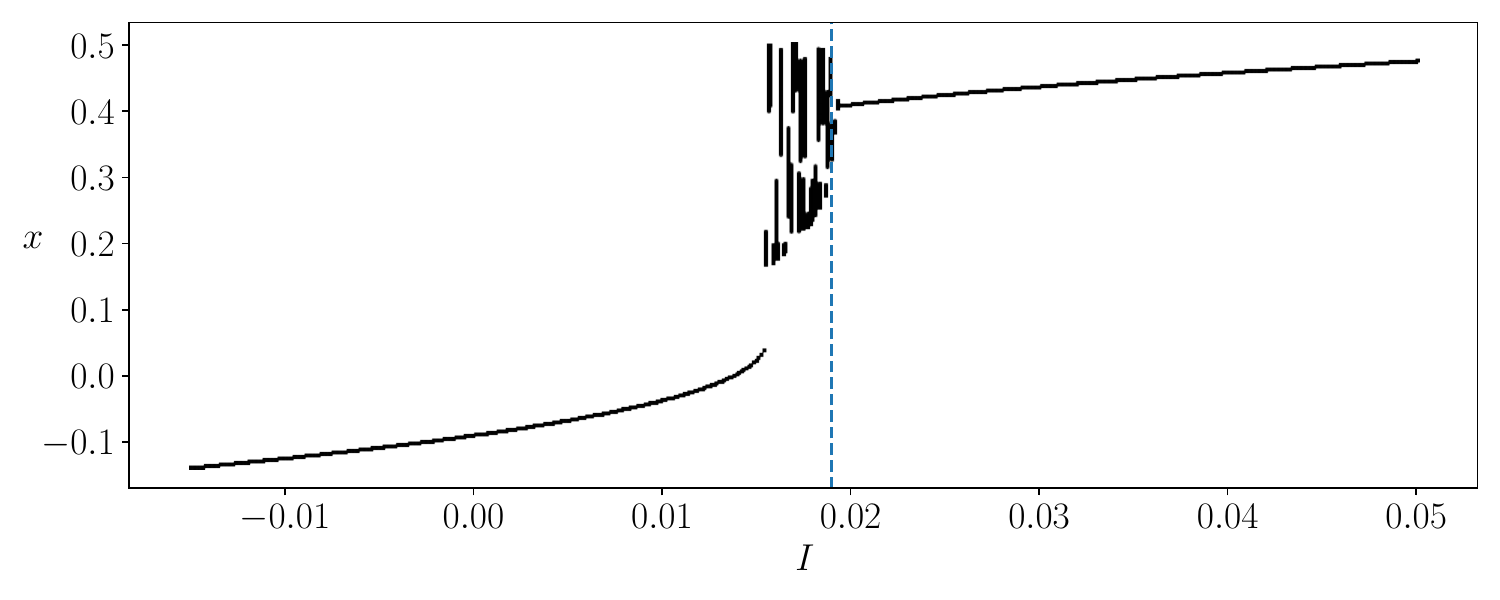}\\
(a) $\beta = 0.99$ & (b) $\beta \approx 0.98233$ \\[3pt]
\includegraphics[scale=0.3]{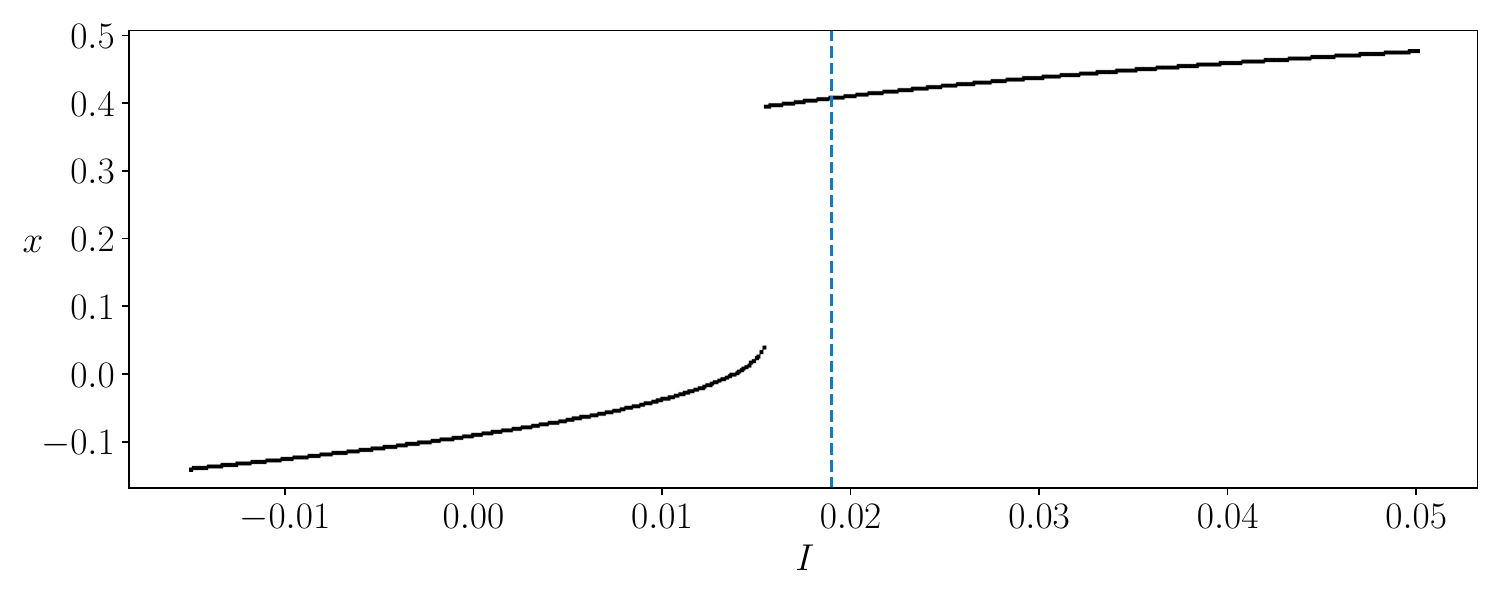} \\
(c) $\beta = 0.97$ \\[3pt]
\end{tabular}
\end{center}
\caption{Bifurcation diagrams with varying $I$ for~\eqref{eq:DML_2D} for different fixed values of $\beta$: (a) $\beta = 0.99$, (b) $\beta \approx 0.98233$, and (c) $\beta = 0.97$. We have set the parameters according to~\eqref{eq:param}. The blue vertical broken line in each panel indicate the $I$ value $I = 0.019$. We know the Hopf bifurcation occurs at $\beta^* \approx 0.98233$ which was the motivation for plotting panel (b). We observe that decreasing $\beta$ stabilizes the system.}
\label{fig:BIF_beta_2D}
\end{figure}

\section{Theoretical and numerical analysis of the coupled systems of dML neurons}
\label{sec:dML_4D}

The systems of four-dimensional fractional differential equations representing a bidirectionally coupled dimer of dML neurons through two different coupling mechanisms are introduced in Sec.~\ref{sec:intro}, see~\eqref{eq:DML_4D} and~\eqref{eq:dML_4D_chemical}. First we look into~\eqref{eq:DML_4D}. We realise that the system is symmetric in terms of two dML neurons coupled to each other via a bidirectional coupling, with a coupling strength $\theta >0$. We look into the stability analysis of a symmetric equilibrium point of~\eqref{eq:DML_4D}, which we are going to denote as $(x^*, y^*, x^*, y^*)$. We closely follow the analysis done for coupled systems of fractional order neuron models in~\citep{MoKa24, Ka17}. The equilibrium can be computed by solving for $x_1, x_2$ from the equations
\begin{align}
\label{eq:eqPoints4D}
    f(x_1, x_2) = 0,\qquad f(x_2, x_1) = 0,
\end{align}
where 
$$
f(x_i, x_j) = x_i^2(1-x_i) - \frac{Ae^{\alpha x_i}}{\gamma} +I + \theta(x_j - x_i). 
$$
For a symmetric equilibrium point, the value of $y^*$ is then given by
$$
y^* = \frac{Ae^{\alpha x}}{\gamma}.
$$
The Jacobian matrix of~\eqref{eq:DML_4D} at the symmetric equilibrium point $(x^*, y^*, x^*, y^*)$ is a block matrix given by 
$$
J = \begin{bmatrix}
    J_1 & \Theta \\ \Theta & J_1
\end{bmatrix},
$$
where
$
J_1 = \begin{bmatrix}
    x^*(2-3x^*) - \theta & -1 \\
    \alpha Ae^{\alpha x^*} & -\gamma
\end{bmatrix},
$
and $\Theta$ is the coupling matrix given by
$
    \Theta = \begin{bmatrix}
        \theta &0 \\ 0 & 0
    \end{bmatrix}
$. The characteristic equation of $J$ is $\det(J_1 + \Theta - \lambda I)\det(J_1 - \Theta - \lambda I) = 0$ with the eigenvalues of $J$ as the union of the eignevalues of $J \pm \Theta$. We have
\begin{align}
    J_1 + \Theta = \begin{bmatrix}
        x^*(2-3x^*) & - 1 \\
        \alpha A e^{\alpha x^*} & - \gamma
    \end{bmatrix}, \qquad 
    J_1 - \Theta = \begin{bmatrix}
        x^*(2-3x^*)-2\theta & - 1 \\
        \alpha A e^{\alpha x^*} & - \gamma
    \end{bmatrix}, \nonumber
\end{align}
with their traces and determinants represented by $\delta^{\pm}(x^*) =\det(J_1 \pm \Theta)$ and $\tau^{\pm}(x^*) = {\rm trace}(J_1 \pm \Theta)$. We have
\begin{equation}
\label{eq:traceDet4D}
\begin{aligned}
    \tau^+(x^*)&= x^*(2-3x^*) - \gamma, \\
    \delta^+(x^*) &= -\gamma x^*(2-3x^*) + \alpha Ae^{\alpha x^*}, \\
    \tau^-(x^*)&= x^*(2-3x^*) - \gamma - 2\theta, \\
    \delta^-(x^*) &= -\gamma \big[ x^*(2-3x^*) -2\theta] + \alpha Ae^{\alpha x^*}.
\end{aligned}
\end{equation}

\begin{theorem}
    Suppose
    \begin{itemize}
        \item[i)] $-\gamma x^*(2-3x^*) > - \alpha Ae^{\alpha x^*}$,
        \item[ii)] $-\gamma x^*(2-3x^*) < 2\sqrt{-\gamma x^*(2-3x^*) + \alpha Ae^{\alpha x^*}} \cos\big(\frac{\beta \pi}{2} \big)$, and
        \item[iii)] $-\gamma x^*(2-3x^*) < 2\big[\theta + \sqrt{-\gamma x^*(2-3x^*) + \alpha Ae^{\alpha x^*} + 2\theta\gamma} \cos\big(\frac{\beta \pi}{2} \big)\big]$.
    \end{itemize}
    Then a symmetric equilibrium point $(x^*, y^*, x^*, y^*)$ of~\eqref{eq:DML_4D} is asymptotically stable.
\end{theorem}

\begin{proof}
  Using the similar approach as Theorem~\ref{thm:EqStab2D}, it is required that the equilibrium point $(x^*, y^*, x^*, y^*)$ of~\eqref{eq:DML_4D} is asymptotically stable if and only if 
  $$
  \delta^\pm(x^*) >0 \qquad {\rm and} \qquad \tau^\pm(x^*) <2\sqrt{\delta^\pm(x^*)}\cos\bigg(\frac{\beta \pi}{2}\bigg).
  $$
  Now assumption (i) directly means $\delta^+(x^*) >0$. This means $\delta^-(x^*) = \delta^+(x^*) + 2\theta \gamma >0$, because $\theta, \gamma >0$. Rearranging assumption (ii) and (iii) a little bit gives us back $\tau^\pm(x^*) <2\sqrt{\delta^\pm(x^*)}\cos\bigg(\frac{\beta \pi}{2}\bigg)$. 
\end{proof}

\begin{theorem}
    Suppose $-\gamma x^*(2-3x^*) + \alpha Ae^{\alpha x^*} < - 2\theta \gamma$. Then the equilibrium point is a saddle point.
\end{theorem}

\begin{proof}
The equilibrium is a saddle point if
$$
  \delta^+(x^*) <0 \qquad {\rm or} \qquad \delta^-(x^*) <0.
  $$
    The inequality $-\gamma x^*(2-3x^*) + \alpha Ae^{\alpha x^*} < - 2\theta \gamma$, directly implies $-\gamma x^*(2-3x^*) + \alpha Ae^{\alpha x^*} < 0$, because $\theta, \gamma>0$, meaning $\delta^+(x^*)<0$. Also a simple rearrangement of $-\gamma x^*(2-3x^*) + \alpha Ae^{\alpha x^*} < - 2\theta \gamma$ directly implies $\delta^-(x^*) <0$.
\end{proof}

A saddle-node bifurcation occurs when 
$$
-\gamma x^*(2-3x^*) + \alpha Ae^{\alpha x^*} = 0 \qquad {\rm and }\qquad \theta = 0.
$$The above cases denote that for saddle-node bifurcation to occur, the coupling should be vanished and each of the systems should have Theorem~\ref{thm:saddleNode} (condition for saddle-node bifurcation for a single-cell system) satisfied.

\begin{theorem}
    Suppose
    \begin{itemize}
        \item[i)] $-\gamma x^*(2-3x^*) + \alpha Ae^{\alpha x^*}>0$, and
        \item[ii)] $x^*(2-3x^*) - \gamma <0$.
    \end{itemize}
    Then the equilibrium point $(x^*, y^*, x^*, y^*)$ is asymptotically stable irrespective of $\beta \in (0, 1]$.
\end{theorem}

\begin{proof}
    The equilibrium point is asymptotically stable irrespective of the order $\beta \in (0, 1]$ if and only if
    $$
  \delta^\pm(x^*) >0 \qquad {\rm and} \qquad \tau^\pm(x^*) <0.
  $$
    The first assumption directly means $\delta^+(x^*)>0$. Also because $\theta>0, \gamma >0$, we can directly as $\delta^-(x^*) = \delta^+(x^*) + 2\theta \gamma >0$. Also the second assumption directly meanse $\tau^+(x^*) <0$. And because $\tau^-(x^*) = \tau^+(x^*) - 2\theta$, we have $\tau^-(x^*)<0$. 
\end{proof}

\begin{theorem}
\label{thm:betaStar4D}
    Suppose 
    \begin{itemize}
        \item[i)] $-\gamma x^*(2-3x^*) + \alpha Ae^{\alpha x^*}>0$, and
        \item[ii)] $x^*(2-3x^*)-\gamma > 2\theta$.
         \end{itemize}
         Then the equilibrium point $(x^*, y^*, x^*, y^*)$ is asymptotically stable if and only if
   \begin{align}
   \label{eq:betaStar4D}
       \beta < \beta^* &= \min \bigg[\frac{2}{\pi} \cos^{-1}\bigg(\min\bigg(1, \frac{x^*(2-3x^*) - 2\theta - \gamma}{2\sqrt{-\gamma x^*(2-3x^*) + \alpha A e^{\alpha x^*} + 2\theta \gamma}} \bigg) \bigg), \nonumber \\
       &\frac{2}{\pi} \cos^{-1}\bigg(\min\bigg(1, \frac{x^*(2-3x^*)  - \gamma}{2\sqrt{-\gamma x^*(2-3x^*) + \alpha A e^{\alpha x^*}}} \bigg) \bigg) \bigg]. 
   \end{align}
\end{theorem}

\begin{proof}
    The first assumption means $\delta^\pm(x^*) >0$. The second assumption implies that $\tau^-(x^*)>0$, which also directly implies $\tau^+(x^*) = \tau^-(x^*)+2\theta>0$ because $\theta>0$. Following similar technical details as for~\eqref{eq:DML_2D} in Theorem~\ref{thm:betaStar2D}, it can be deduced that the symmetric equilibrium point for~\eqref{eq:DML_4D} is asymptotically stable if and only if 
    \begin{align}
       \beta < \beta^* &= \min \bigg[\frac{2}{\pi} \cos^{-1}\bigg(\min\bigg(1, \frac{\tau^-(x^*)}{2\sqrt{\delta^-(x^*)}} \bigg) \bigg),  \frac{2}{\pi} \cos^{-1}\bigg(\min\bigg(1, \frac{\tau^+(x^*)}{2\sqrt{\delta^+(x^*)}} \bigg) \bigg) \bigg]. \nonumber 
   \end{align}
Substituting the values of $\tau^\pm(x^*)$ and $\delta^{\pm}(x^*)$ in the above formula gives us~\eqref{eq:betaStar4D}.
\end{proof}

From~\eqref{eq:traceDet4D}, we directly see that $\delta^-(x^*)>\delta^+(x^*)$ and $\tau^-(x^*) < \tau^+(x^*)$ because $\theta >0$. This automatically means for $\tau^\pm(x^*), \delta^\pm(x^*)>0$, we will have $\frac{\tau^-(x^*)}{2\sqrt{\delta^-(x^*)}} < \frac{\tau^+(x^*)}{2\sqrt{\delta^+(x^*)}}$, implying $\cos^{-1}\bigg(\frac{\tau^-(x^*)}{2\sqrt{\delta^-(x^*)}} \bigg) > \cos^{-1}\bigg(\frac{\tau^+(x^*)}{2\sqrt{\delta^+(x^*)}} \bigg)$. This essentially means the formula of $\beta^*$~\eqref{eq:betaStar4D} reduces to
\begin{align}
    \beta^* &= \frac{2}{\pi} \cos^{-1}\bigg(\min\bigg(1, \frac{\tau^+(x^*)}{2\sqrt{\delta^+(x^*)}} \bigg) \bigg) \nonumber \\
    &= \frac{2}{\pi} \cos^{-1}\bigg(\min\bigg(1, \frac{x^*(2-3x^*)  - \gamma}{2\sqrt{-\gamma x^*(2-3x^*) + \alpha A e^{\alpha x^*}}} \bigg) \bigg) \nonumber,
\end{align}
which is exactly what we have for the formula of $\beta^*$ for the single neuron model~\eqref{eq:DML_2D}, see~\eqref{eq:betaStar}.

Thus for the case of the dimer~\eqref{eq:DML_4D}, a Hopf bifurcation also occurs at~\eqref{eq:betaStar4D}. The symmetric equilibrium $(x^*, y^*, x^*, y^*)$ loses its stability as $\beta$ is increased beyond $\beta^*$ and eventually stable limit cycles appear for both neurons. We initiate the numerical analysis of the coupled system by first looking into their phase portraits and time series. The initial conditions for the first neuron are set as $(x_1(0), y_1(0)) =(0.1, 0.1)$ and for the second neuron are set as $(x_2(0), y_2(0)) = (-0.2, 0.1)$. Local parameters of both neurons are~\eqref{eq:param}, indicating that the neurons are identical. Function \texttt{F} for using \texttt{FDEsolver()} is set as the right hand side of~\eqref{eq:DML_4D}. Time span is \texttt{tSpan}$=[0, 6000]$ and time step $h = 0.01$. External current $I$ is set to $I = 0.19>I_{\max}$. We treat $\beta$ and $\theta$ as the primary bifurcation parameters and vary them accordingly to study various firing patterns. All parameters are collected in the vector \texttt{par} for running the simulation.

For $\theta = 0.008$, the phase portraits and time series of both neurons from~\eqref{eq:DML_4D} for $I = 0.019>I_{\max}$ are displayed in Fig.~\ref{fig:pp4D}. We study a symmetric equilibrium point $(x^*, y^*, x^*, y^*)$ which is unique for this choice of $I$. In each panel the left hand side displays the $x$- and $y$-nullclines of~\eqref{eq:DML_4D}. Note that the nullclines will be equivalent for both the neurons for a symmetric equilibrium point. Again for plotting the phase portraits we discard the first $10^5$ data points out of the $6 \times 10^5$ data points. Phase portrait for the first neuron is colored in green and the second neuron in black. On the right hand side we plot the time series for both $x$ and $y$ corresponding to both neurons and we do not discard any data points. The $x$-variable time series is colored in green for the first neuron and in black for the second neuron, along with the time series of the $y$-variable colored in orange for both. We study the dynamics by again varying $\beta$ in the range $[0.9, 1]$ transitioning through $\beta^*$. Solving for a symmetric equilibrium point for~\eqref{eq:DML_4D} from~\eqref{eq:eqPoints4D} ultimately reduces to solving for $x^*$ from
\begin{align}
\label{eq:eqPointSoln4D}
{x^*}^2(1-x^*) - \frac{Ae^{\alpha x^*}}{\gamma} + I = 0,
\end{align}
which is interestingly what we have obtained for the single cell model~\eqref{eq:DML_2D}, see~\eqref{eq:eqPointSoln}. Thus, for $I = 0.019$, we will again have $x^* \approx 0.40772$ with the local parameters for both neurons set to~\eqref{eq:param}. Similar to~\eqref{eq:betaStar} for the single-cell case, we obtain $\beta^* \approx 0.98233$ for the dimer as well. Fig.~\ref{fig:pp4D}-(a) to (c) indicate $\beta < \beta^*$ where the dynamics of~\eqref{eq:DML_4D} always converges asymptotically to the symmetric equilibrium. We observe a Hopf bifurcation at $\beta = \beta^*$ (panel (d)), and the time series displays tonic spiking. As $\beta$ is further increased from $\beta^*$ (panels (e), (f)), we see the appearance of stable limit cycles in the phase portraits of both neurons and the symmetric equilibrium loses its stability (note that the limit cycles of the neurons are not equivalent to each other pertaining to different initial conditions for the neurons). In panel (e) for $\beta = 0.99$ we observe a bursting behavior where the spike counts in every burst deviate in terms of having various small and large amplitudes. We also notice some kind of periodicity. As $\beta$ is increased to $\beta = 1$ (panel (f)), making the system a standard integer order coupled system, we observe periodic bursting. Note that across the two neurons a mixed mode oscillatory behavior also arise~\citep{Tr21, BrKa08} A similar behavior was also reported in a coupled system of fractional-order Wilson Cowan neurons by Mondal {\em et al.}~\citep{MoKa24}. Next, we decrease the coupling strength to $\theta = 0.001$ and again study the corresponding phase portraits and time series for a $\beta$ in the range $[0.9, 1]$ in Fig.~\ref{fig:pp4DB}. We observe that for $I = 0.019$ the $\beta^*$ value does not change for changing the coupling strength due to the fact that the equation for solving $x^*$~\eqref{eq:eqPointSoln4D} is independent of $\theta$. Thus, we have $\beta^* \approx 0.98233$. Note that, this time, we observe tonic spiking in the time series after the system undergoes a Hopf bifurcation at $\beta^*$. When $\beta > \beta^*$, the amplitude of the tonic spikes increases with increasing $\beta$. 

\begin{figure*}[h]
\begin{center}
\begin{tabular}{cc}
  \includegraphics[scale=0.25]{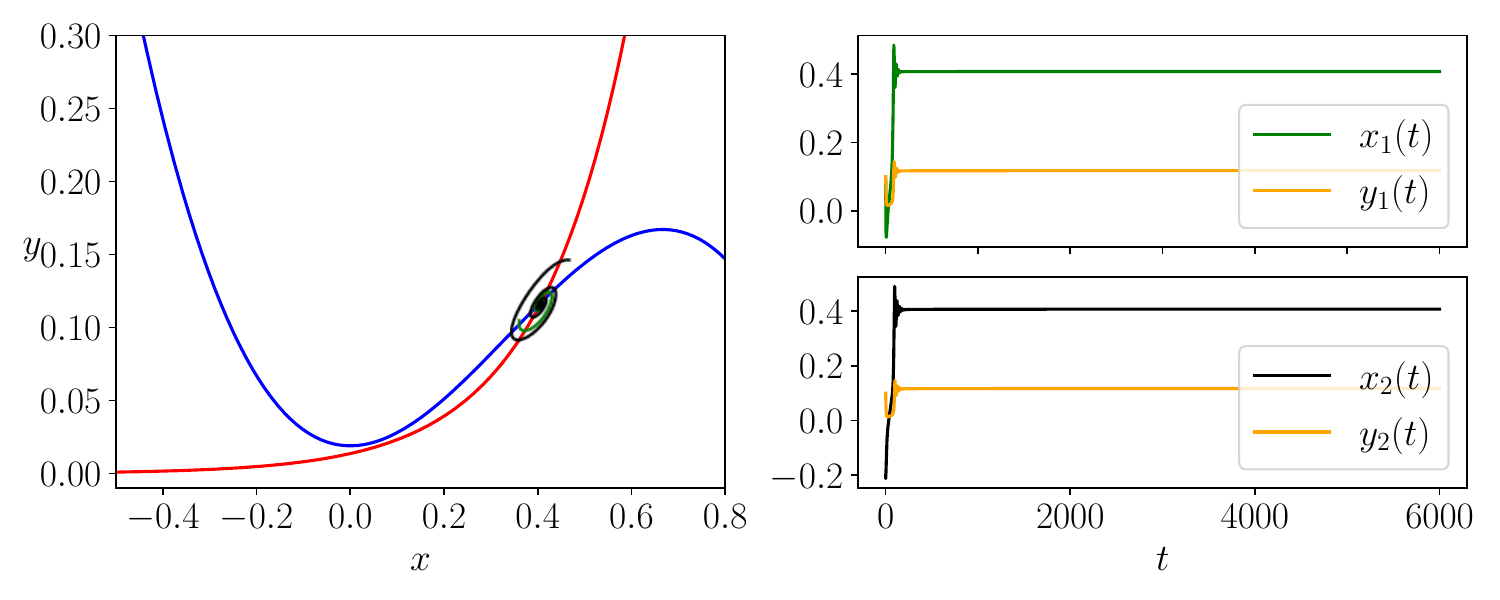} &  \includegraphics[scale=0.25]{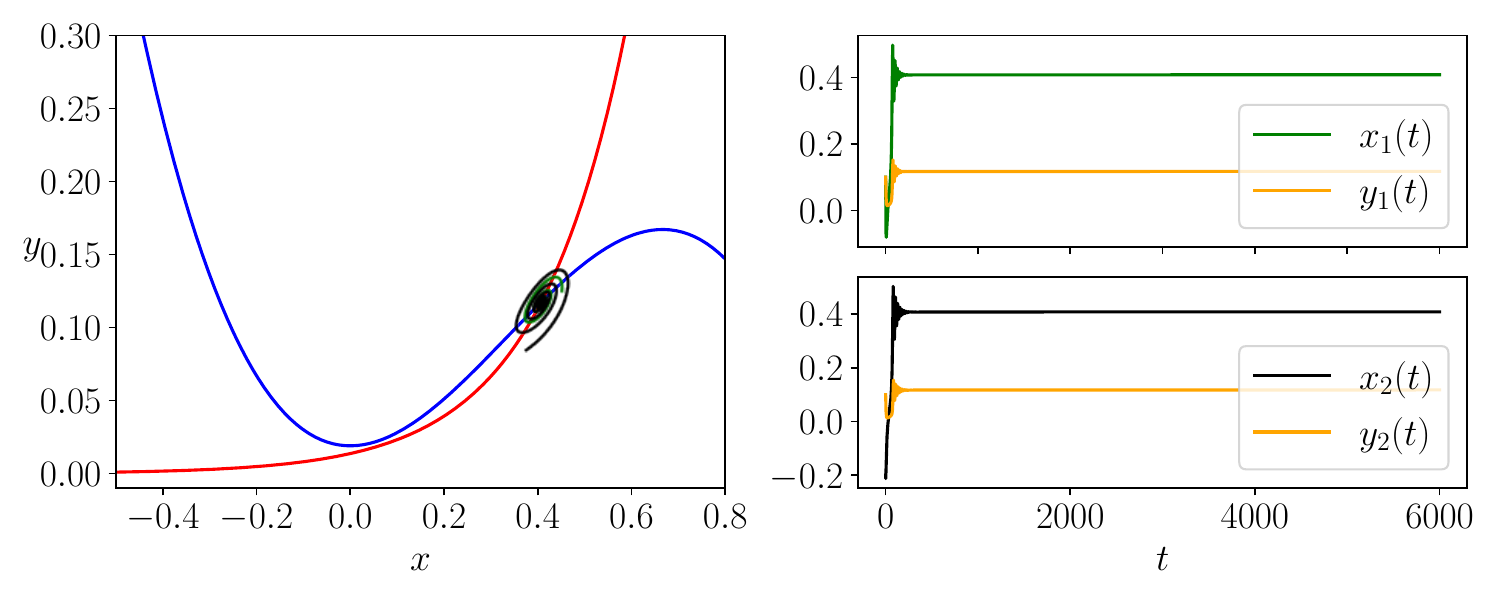} \\
(a) $\beta = 0.9$ & (b) $\beta = 0.93$ \\[3pt]
\includegraphics[scale=0.25]{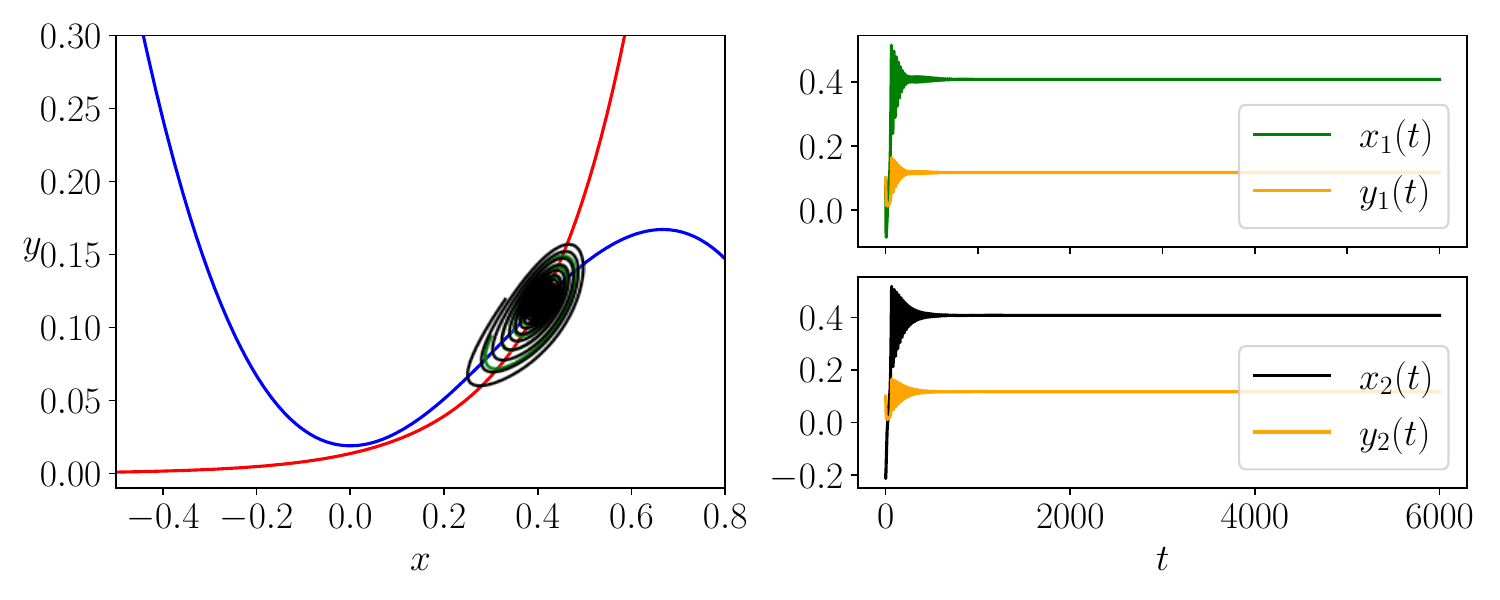} &  \includegraphics[scale=0.25]{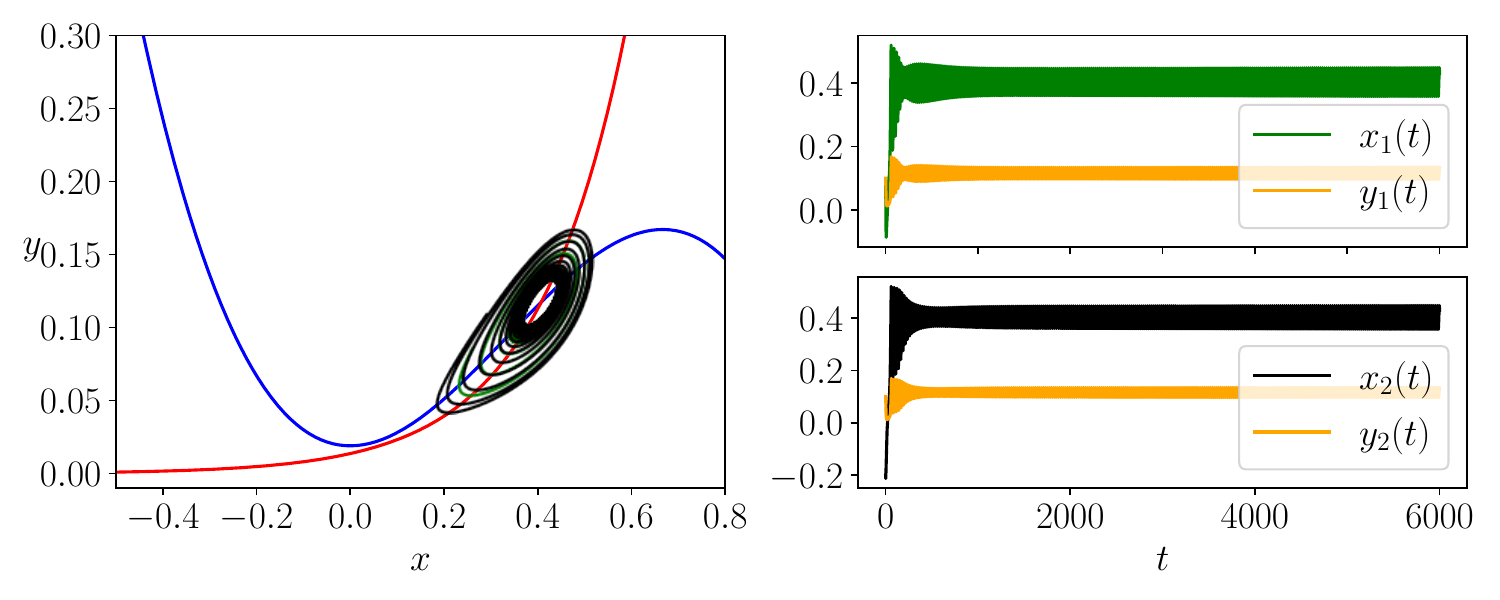} \\
(c) $\beta = 0.97$ &  (d) $\beta = \beta^*$\\[3pt]
\includegraphics[scale=0.25]{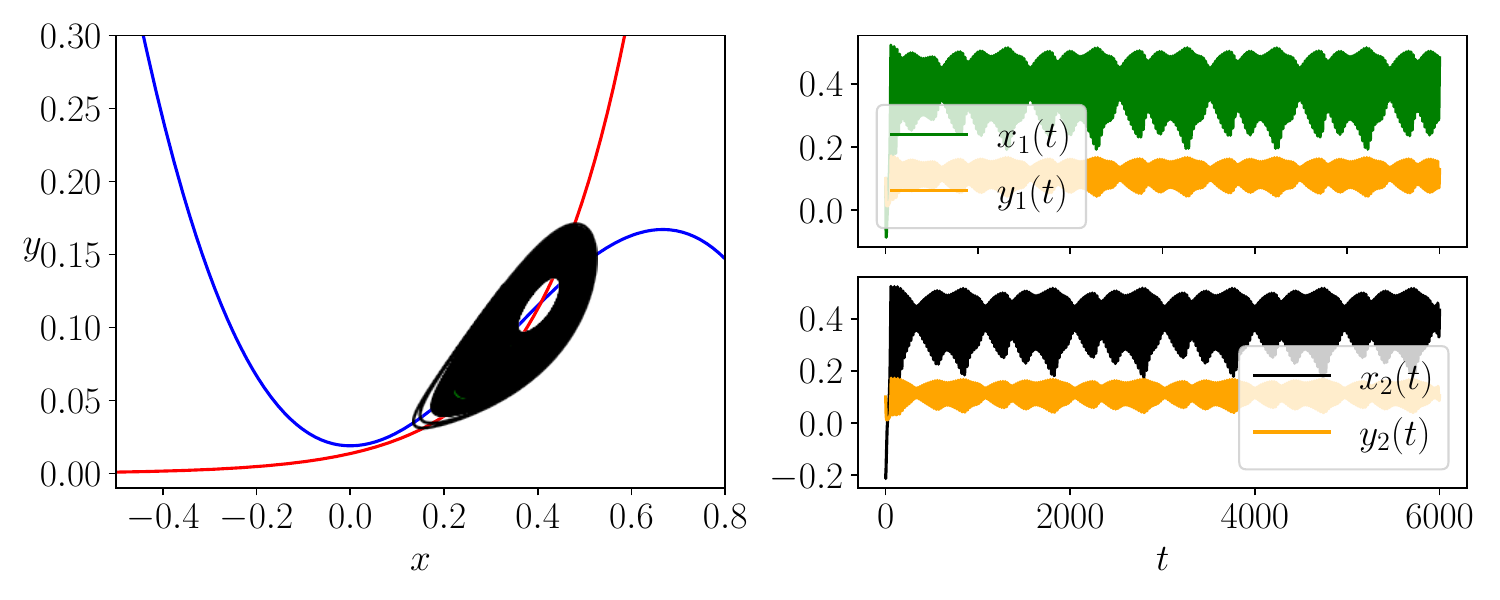} &  \includegraphics[scale=0.25]{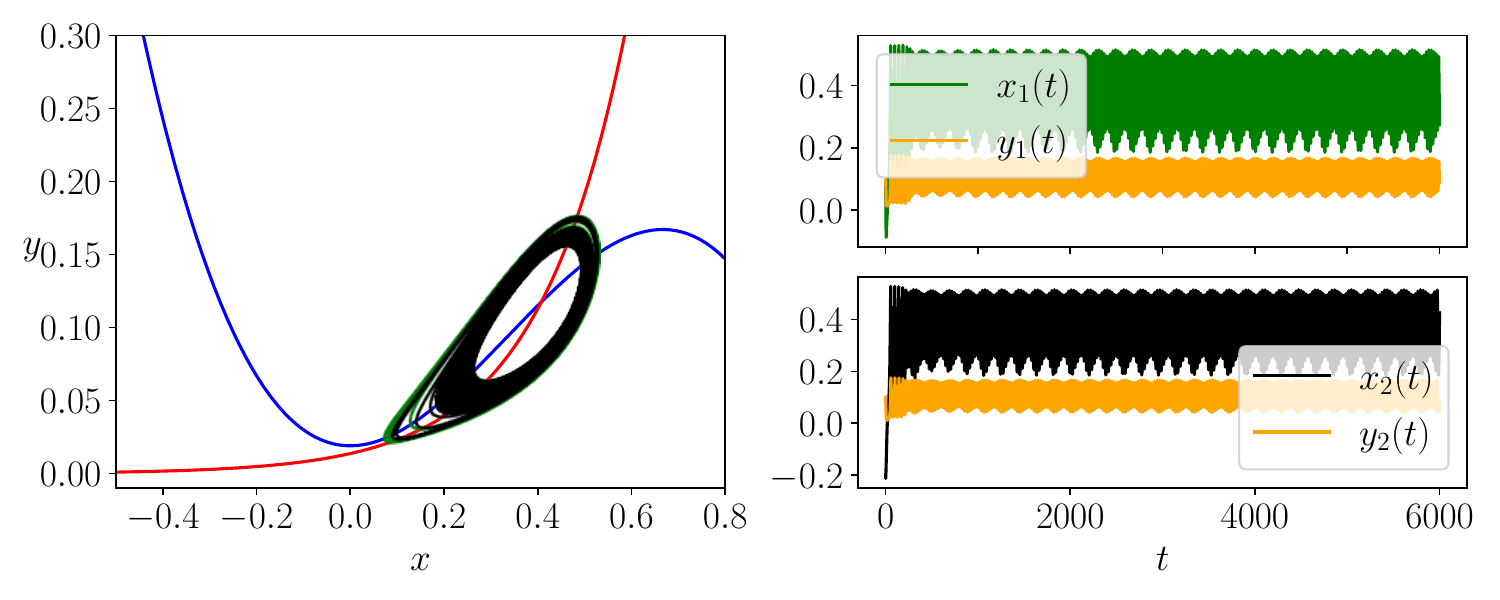} \\
(e) $\beta = 0.99$ & (f) $\beta = 1$ \\[3pt]
\end{tabular}
\end{center}
\caption{Phase portraits and time series with varying $\beta$ for~\eqref{eq:DML_4D} with a coupling strength $\theta = 0.008$. We have set $I = 0.019>I_{\max}$ similar to what we have done for the two-dimensional system. The other local parameters for both the neurons are set as~\eqref{eq:param}. We will have $\beta^* \approx 0.98233$ similar to what has been reported for~\eqref{eq:DML_2D}. The coupled system~\eqref{eq:DML_4D} converges to a stable symmetric equilibrium $(x^*, y^*, x^*, y^*)$ when $\beta < \beta ^*$. When $\beta \ge \beta^*$, the equilibrium loses stability through a Hopf bifurcation and a stable limit cycle appears for both neurons.}
\label{fig:pp4D}
\end{figure*}

\begin{figure*}[h]
\begin{center}
\begin{tabular}{cc}
  \includegraphics[scale=0.25]{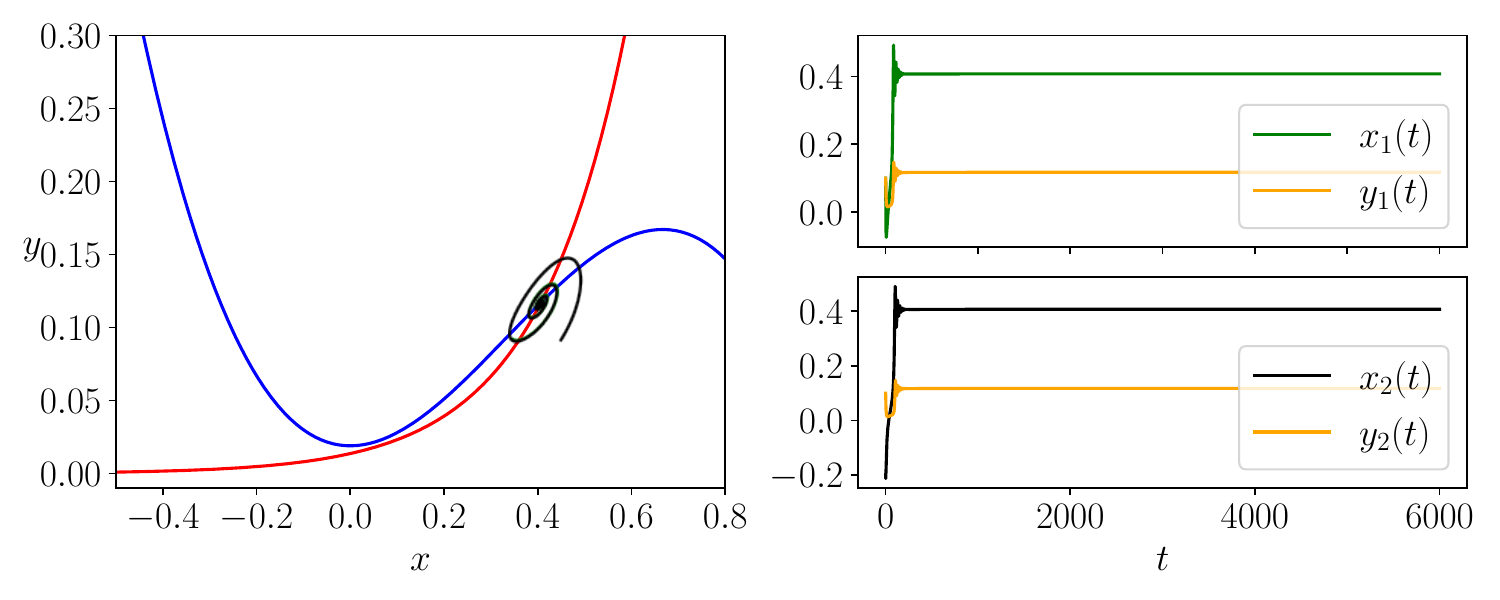} &  \includegraphics[scale=0.25]{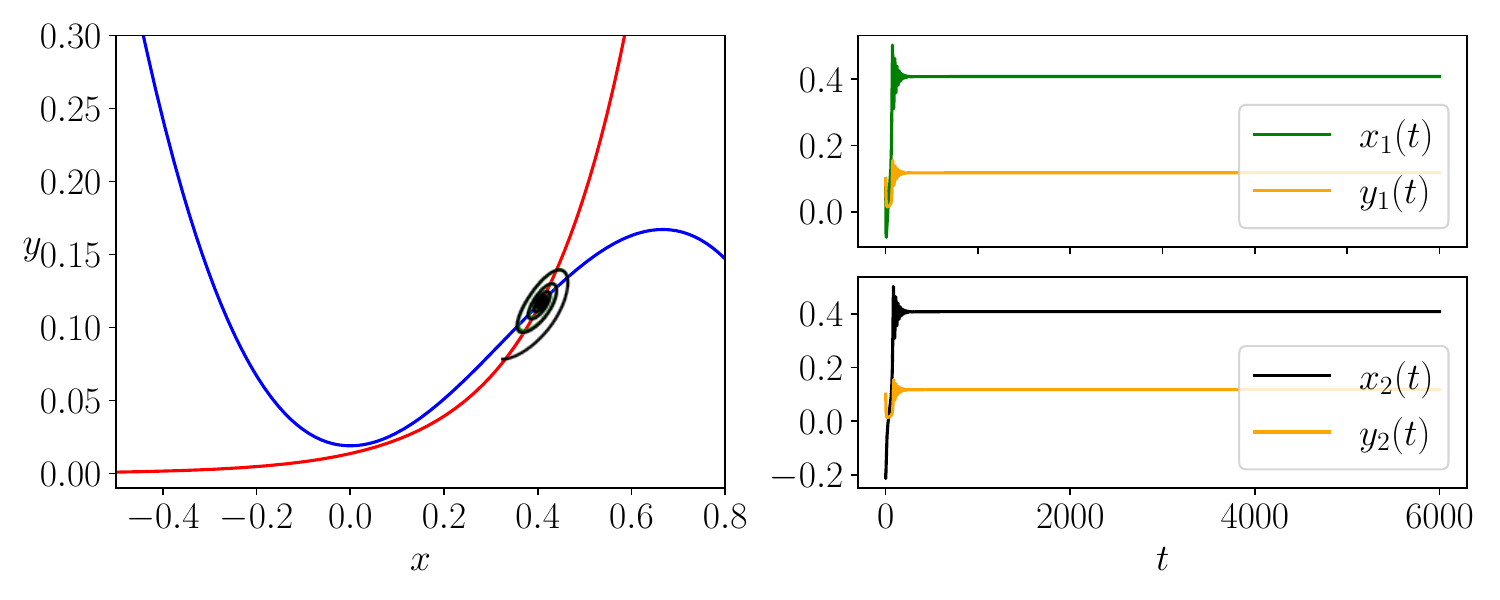} \\
(a) $\beta = 0.9$ & (b) $\beta = 0.93$ \\[3pt]
\includegraphics[scale=0.25]{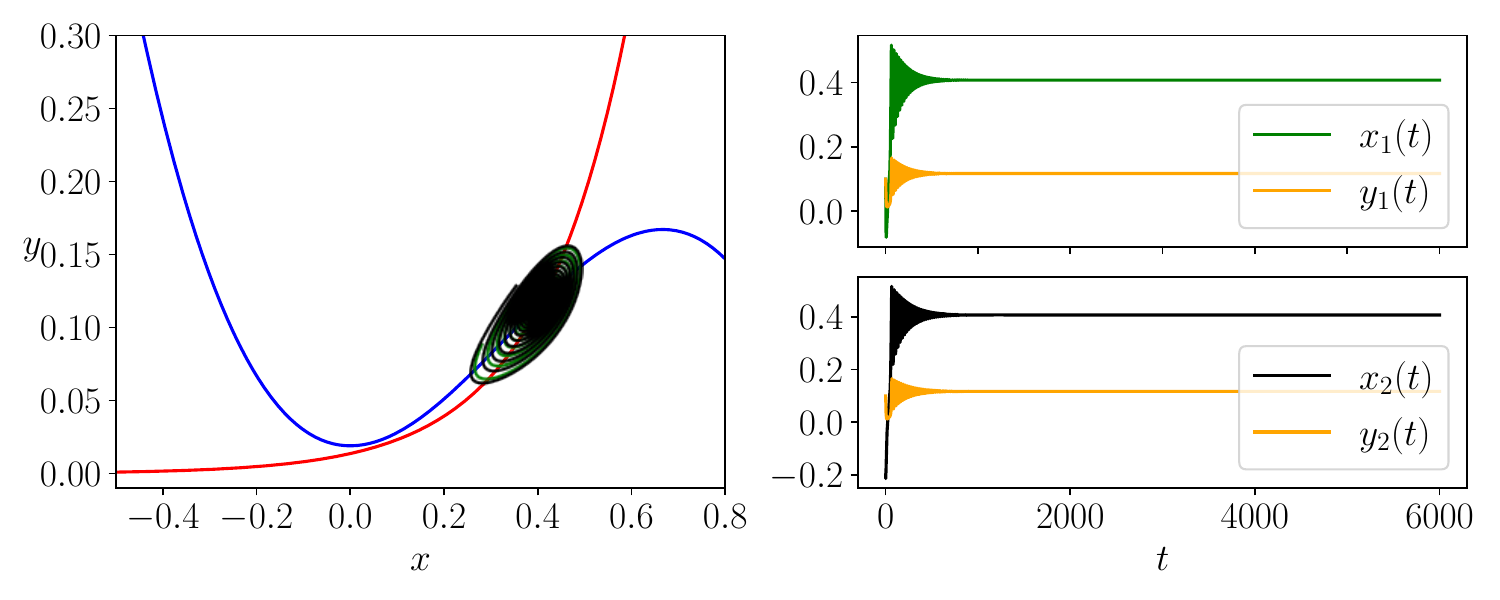} &  \includegraphics[scale=0.25]{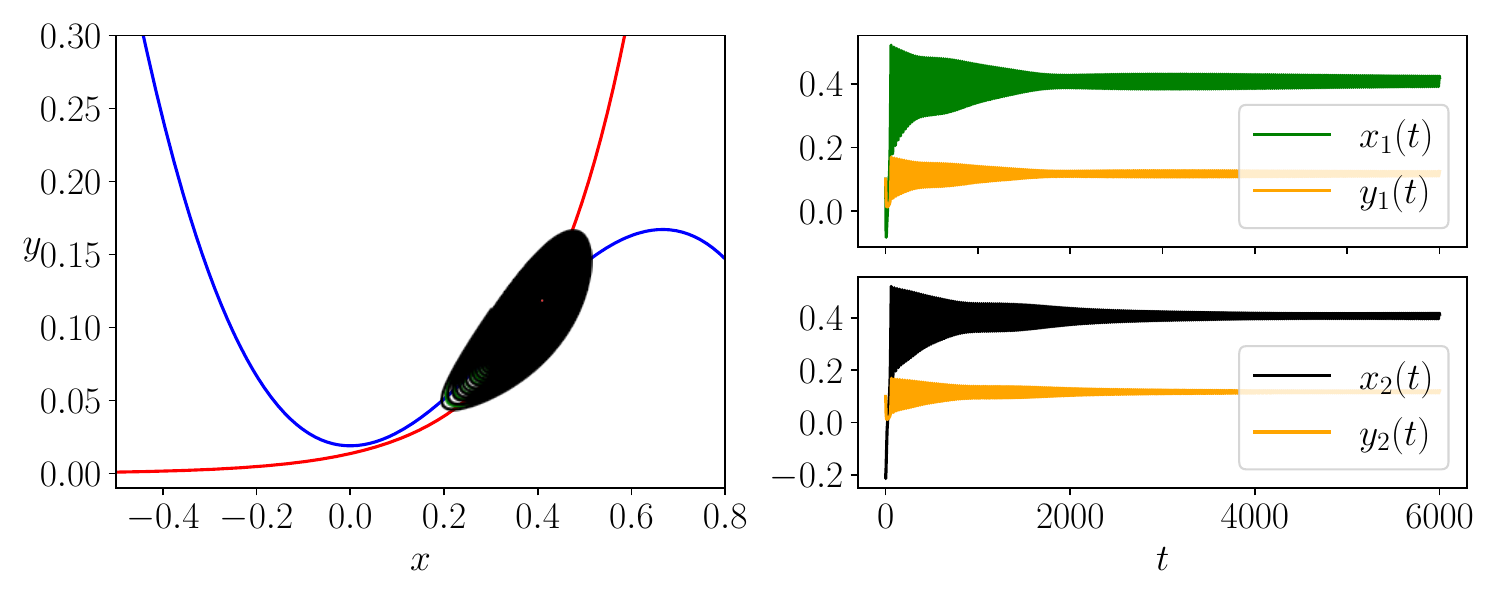} \\
(c) $\beta = 0.97$ &  (d) $\beta = \beta^*$\\[3pt]
\includegraphics[scale=0.25]{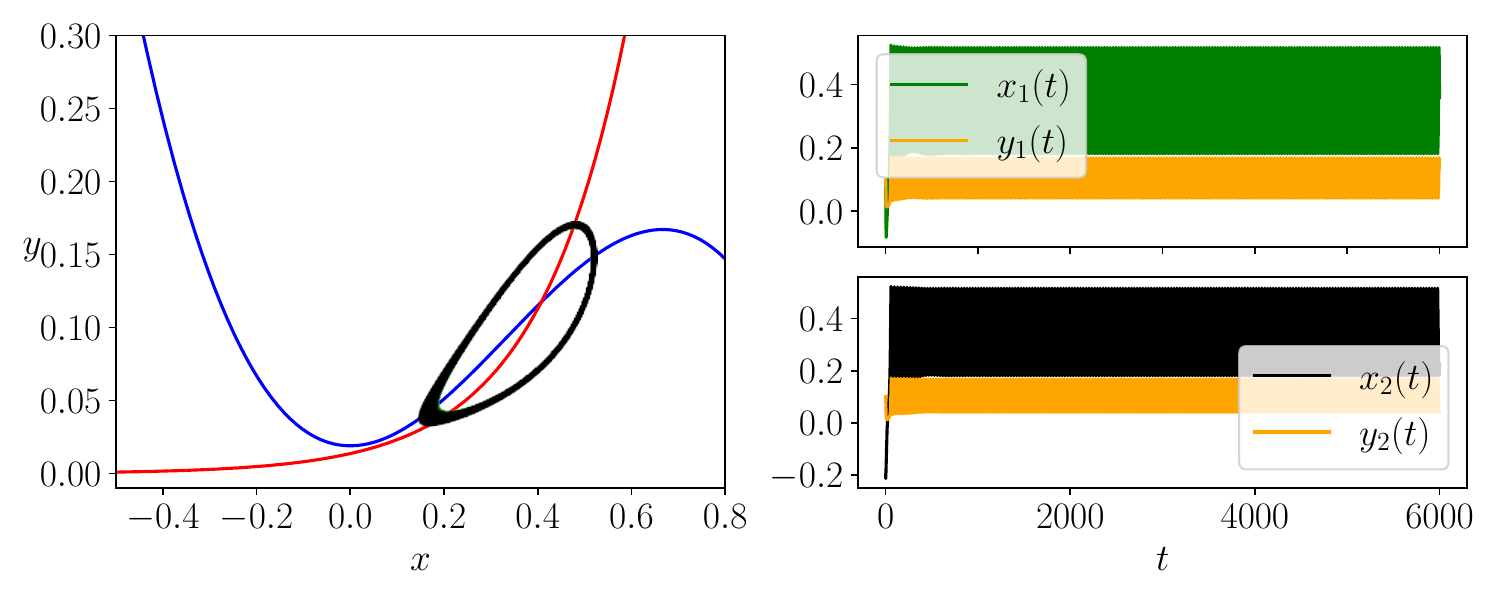} &  \includegraphics[scale=0.25]{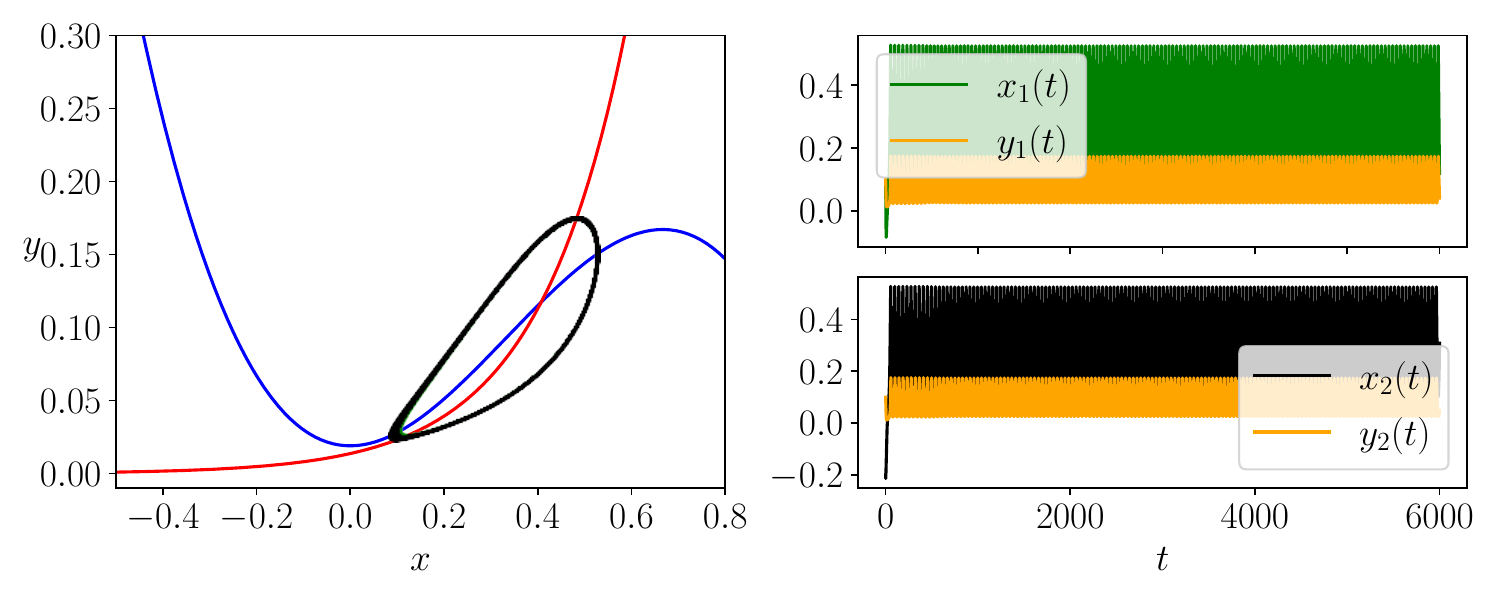} \\
(e) $\beta = 0.99$ & (f) $\beta = 1$ \\[3pt]
\end{tabular}
\end{center}
\caption{Phase portraits and time series with varying $\beta$ for~\eqref{eq:DML_4D} with a coupling strength $\theta = 0.001$. We have set $I = 0.019>I_{\max}$ similar to what we have done for the two-dimensional system. The other local parameters for both the neurons are set as~\eqref{eq:param}. We will have $\beta^* \approx 0.98233$ similar to what has been reported for~\eqref{eq:DML_2D}. The coupled system~\eqref{eq:DML_4D} converges to a stable symmetric equilibrium $(x^*, y^*, x^*, y^*)$ when $\beta < \beta ^*$. When $\beta \ge \beta^*$, the equilibrium loses stability through a Hopf bifurcation and a stable limit cycle appears for both neurons.}
\label{fig:pp4DB}
\end{figure*}

Next, we plot the bifurcation diagrams of both the voltage variables $x_1(t), x_2(t)$ with varying $\beta$ for two different cases of $\theta$: (a) $\theta = 0.008$, (b) $\theta = 0.001$ in Fig.~\ref{fig:BIF4D}. As proved earlier, the value of $\beta^*$ is independent of $\theta$ as the equation for solving the symmetric equilibrium point of~\eqref{eq:DML_4D} is independent of $\theta$. This essentially gives us identical results to what we have obtained for the single-cell case~\eqref{eq:DML_2D}. For both the cases we have Hopf bifurcation occurring at $\beta^* \approx 0.98233$. Then in Fig.~\ref{fig:Hopf4D} we show the stable and unstable regions of the symmetric equilibrium point $(x^*, y^*, x^*, y^*)$ on the $(I, \beta)$-plane with $I$ varied in the range $[0.016, 0.03]$. As expected, this diagram is identical to Fig.~\ref{fig:Hopf2D} pertaining to the fact that the equation to solve for $x^*$ for~\eqref{eq:DML_4D} reduces to that of~\eqref{eq:DML_2D} (independent of $\theta$), with the fact that the formulas of $\beta^*$ are exactly the same.

\begin{figure}[h]
\begin{center}
\begin{tabular}{c}
  \includegraphics[scale=0.4]{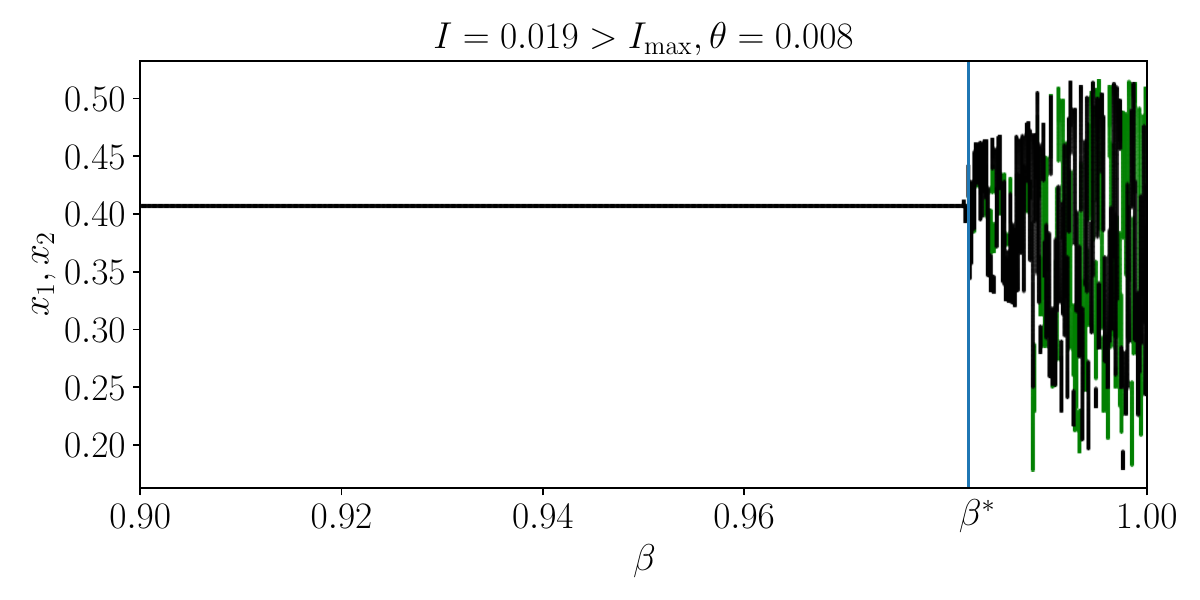} \\
(a) $\beta^* \approx 0.98233$ \\[3pt]
\includegraphics[scale=0.4]{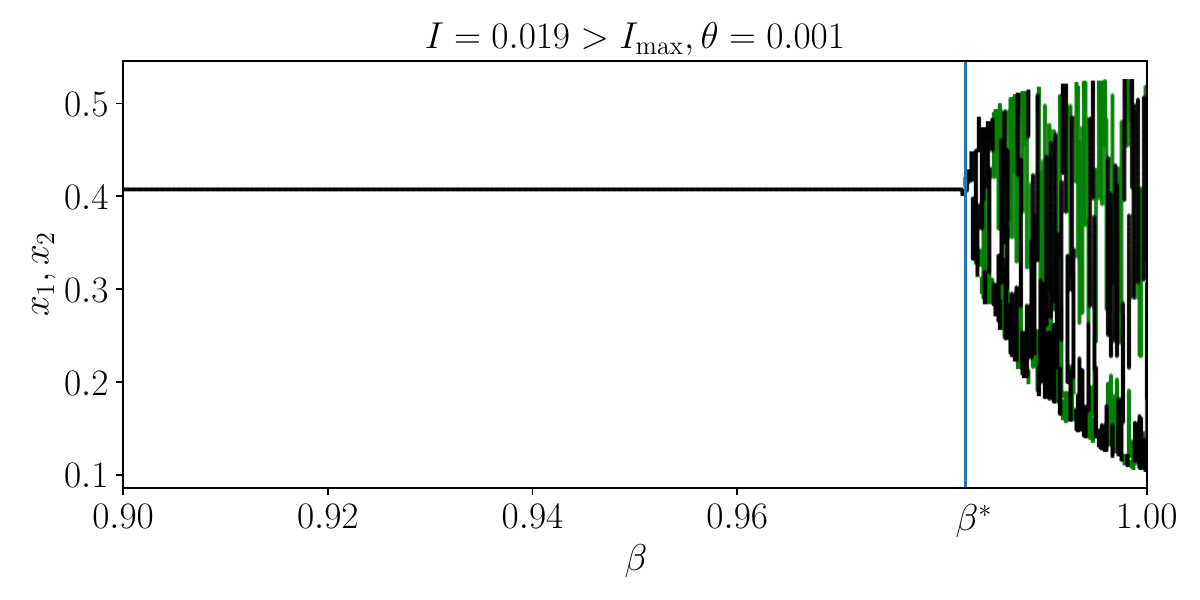} \\
(b) $\beta^* \approx 0.98233$ \\[3pt]
\end{tabular}
\end{center}
\caption{Bifurcation diagrams of $x_1(t), x_2(t)$ with varying $\beta$ for~\eqref{eq:DML_4D} at $I = 0.019 > I_{\rm max}$ with fixed $\theta$. Specifically, we have (a) $\theta=0.008$ whose corresponding phase portraits are displayed in Fig.~\ref{fig:pp4D} and (b) $\theta = 0.001$ whose corresponding phase portraits are displayed in Fig.~\ref{fig:pp4DB}. Local parameters for both neurons are set following~\eqref{eq:param}. Variable $x_1(t)$ is plotted in green and $x_2(t)$ in black. The vertical blue lines indicate $\beta\approx 0.98233$ for each of the cases where a Hopf bifurcation occurs.}
\label{fig:BIF4D}
\end{figure}

\begin{figure}[h]
    \centering
    \includegraphics[width=0.5\linewidth]{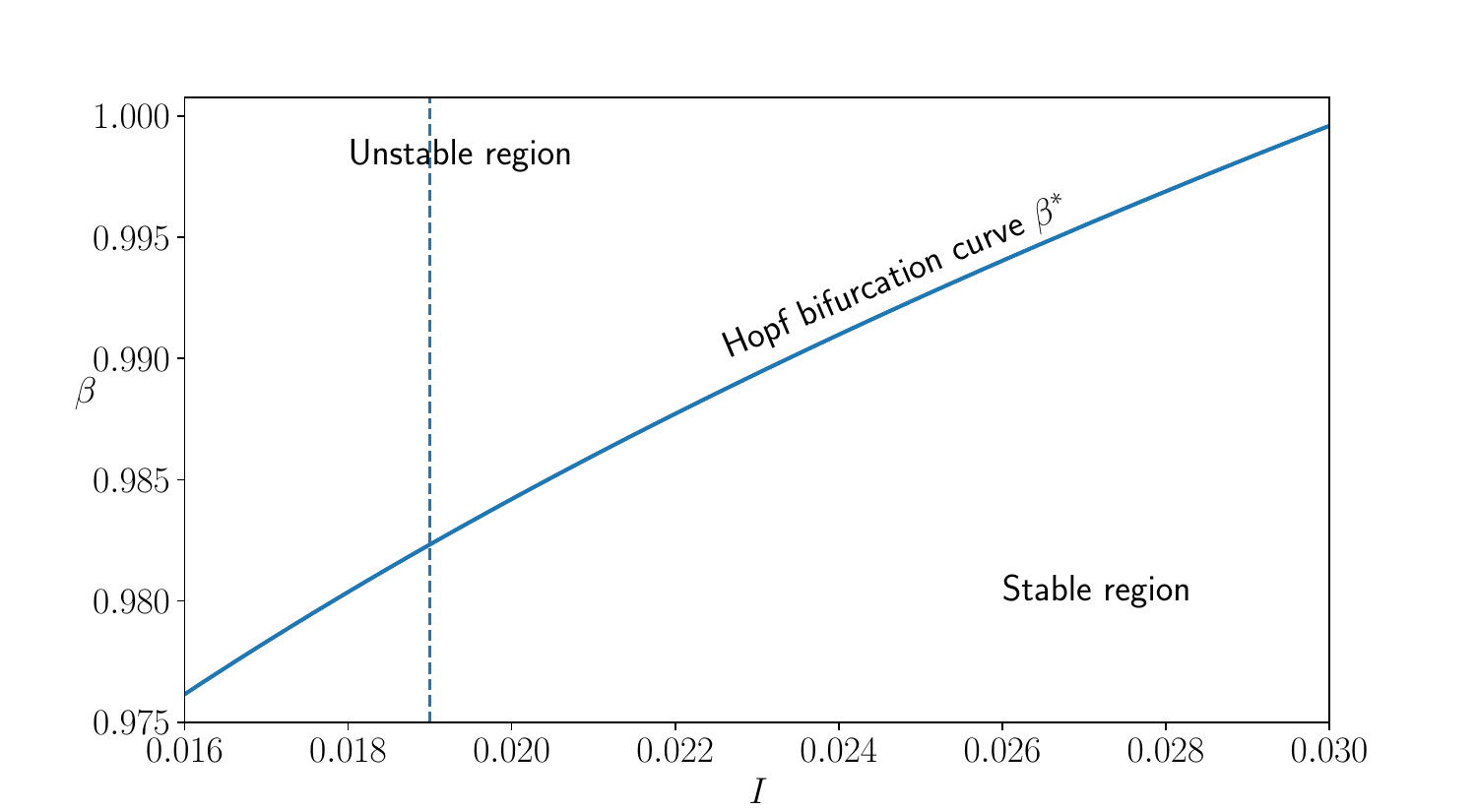}
    \caption{Hopf bifurcation curve $\beta^*$ for~\eqref{eq:DML_4D} separating the stable and the unstable regions of the symmetric equilibrium point $(x^*, y^*, x^*, y^*)$ on the $(I, \beta)$-plane. The vertical broken line indicates $I = 0.019$. The system undergoes a Hopf bifurcation when $\beta$ crosses the curve from the bottom for a particular $I$ value. Other parameter values are set as~\eqref{eq:param}. Note that this curve is independent of the coupling strength $\theta$ and is identical to Fig.~\ref{fig:Hopf2D}.}
    \label{fig:Hopf4D}
\end{figure}

We next move on to employing the coupling scheme modeled by the sigmoidal function that acts as a bidirecational synaptic connection between the two neurons, see~\eqref{eq:dML_4D_chemical}. Here we briefly discuss the qualitative properties of~\eqref{eq:dML_4D_chemical} that will closely follow what has been done for~\eqref{eq:DML_4D}. We again look into the stability analysis of a symmetric equilibrium point of~\eqref{eq:dML_4D_chemical}. An equilibrium point of the system~\eqref{eq:dML_4D_chemical} can be computed by solving for $x_1, x_2$ from the equations
\begin{align}
\label{eq:eqPoints4DChem}
g(x_1, x_2) = 0, \qquad g(x_2, x_1) = 0,
\end{align}
where
$$
g(x_i, x_j) = x_i^2(1-x_i) - \frac{Ae^{\alpha x_i}}{\gamma} + I + \sigma \frac{v_s - x_i}{1+e^{-\lambda(x_j - q)}}.
$$
Once $x^*$ is computed, then we will have $y^* = \frac{Ae^{\alpha x^*}}{\gamma}$. The Jacobian matrix of~\eqref{eq:dML_4D_chemical} at $(x^*, y^*, x^*, y^*)$ is again a block matrix given by
$$
J = \begin{bmatrix}
    J_2 & \Sigma \\ \Sigma & J_2
\end{bmatrix}, 
$$
where $J_2 = \begin{bmatrix}
    x^*(2-3x^*) - \frac{\sigma}{1+e^{-\lambda(x^* - q)}} & -1 \\
    \alpha Ae^{\alpha x^*} & -\gamma
\end{bmatrix}$, and $\Sigma$ is the coupling matrix given by $\Sigma = \begin{bmatrix}
    \sigma(v_s - x^*)\times R(x^*) & 0 \\ 0 & 0
\end{bmatrix}$. Note that 
$$
R(x) = \frac{d}{dx} \bigg[\frac{1}{(1+e^{-\lambda(x-q)})} \bigg] = \frac{\lambda e^{-\lambda(x-q)}}{\big(1+e^{-\lambda(x-q)}\big)^2}.
$$
We will have
$$
J_2 \pm \Sigma = \begin{bmatrix}
  S^{\pm}(x^*)  & -1 \\
    \alpha Ae^{\alpha x^*} & -\gamma  
\end{bmatrix},
$$ with their traces and determinants given by
\begin{equation}
    \label{eq:traceDet4DChemical}
    \begin{aligned}
    \tau^\pm(x^*) &= S^{\pm}(x^*) - \gamma, \\
    \delta^\pm(x^*) &= -\gamma S^\pm(x^*) + \alpha Ae^{\alpha x^*},
\end{aligned}
\end{equation}
with
$$
S^{\pm}(x^*) = x^*(2-3x^*) - \frac{\sigma}{1+e^{-\lambda(x^* - q)}} \pm \sigma(v_s - x^*)\times R(x^*).
$$
We utilise a similar approach to what was done for the dimer coupled using the first coupling scheme, for the stability analysis of~\eqref{eq:dML_4D_chemical}. Note that for the special case where the coupling strength $\sigma = 0$, the dynamics of each neuron is governed by~\eqref{eq:DML_2D}. Also the formulas of traces and determinants~\eqref{eq:traceDet4DChemical} reduce to~\eqref{eq:trace2D}. An equilibrium point $(x^*, y^*, x^*, y^*)$ of~\eqref{eq:dML_4D_chemical} is asymptotically stable if and only if
\begin{itemize}
    \item[i)] $S^{\pm}(x^*) < \frac{\alpha A e^{\alpha x^*}}{\gamma}$, and
    \item[ii)] $S^{\pm}(x^*) < 2\sqrt{\alpha Ae^{\alpha x^*} - \gamma S^{\pm}(x^*)} \cos(\frac{\beta \pi}{2})$.
\end{itemize}
If $S^+(x^*)>\frac{\alpha Ae^{\alpha x^*}}{\gamma}$ or $S^-(x^*)> \frac{\alpha Ae^{\alpha x^*}}{\gamma}$, then the equilibrium point is a saddle. A saddle-node bifurcation occurs when $S^{\pm}(x^*) = \frac{\alpha Ae^{\alpha x^*}}{\gamma}$. However, when $S^{\pm}(x^*)<\frac{\alpha Ae^{\alpha x^*}}{\gamma}$, the stability of the symmetric equilibrium depends on two different cases:
\begin{itemize}
    \item[i)] If $S^{\pm}(x^*)<\gamma$, then the equilibrium is asymptotically stable irrespective of $\beta \in (0, 1]$.
    \item[ii)] Otherwise, the equilibrium is asymptotically stable if and only if $\beta <\beta^*$ where
    \begin{align}
    \label{eq:betaStar4DChem}
        \beta^* = \min\bigg[\frac{2}{\pi}\cos^{-1}\bigg(\min\bigg(1, \frac{S^-(x^*)-\gamma}{2\sqrt{-\gamma S^-(x^*) + \alpha A e^{\alpha x^*}}} \bigg) \bigg), \nonumber\\  \frac{2}{\pi}\cos^{-1}\bigg(\min\bigg(1, \frac{S^+(x^*)-\gamma}{2\sqrt{-\gamma S^+(x^*) + \alpha A e^{\alpha x^*}}} \bigg) \bigg)\bigg].
    \end{align}
\end{itemize}
Again $\beta^*$ is a Hopf bifurcation for~\eqref{eq:dML_4D_chemical} and the formula~\eqref{eq:betaStar4DChem} reduces to~\eqref{eq:betaStar} if and only if $\sigma = 0$. For the numerical simulations, we set the initial conditions again as $(x_1(0), y_1(0)) = (0.1, 0.1)$, and $(x_2(0), y_2(0)) =(-0.2, 0.1)$. The local parameters that govern the dynamics of each neuron are set according to~\eqref{eq:param}, again indicating identical neurons. Function \texttt{F} for \texttt{FDEsolver()} is the right hand side of~\eqref{eq:dML_4D_chemical}, \texttt{tSpan}$=[0, 6000]$, and $h = 0.01$. We again set $I = 0.019$. Similar to the first coupling scheme we treat $\beta$ and $\sigma$ as the primary bifurcation parameters to study the firing patterns. The reversal potential is set as $v_s = 2$ ensuring the synapse is excitatory at all time. The parameters for the sigmoidal function $\zeta(x) = \frac{1}{1+e^{-\lambda(x-q)}}$ are set as $\lambda=10$ (slope) and $q = -0.25$ (synaptic threshold), following the parameters set in Belykh {\em et al.}~\citep{BeDe05}.

The phase portraits and time series are portrayed in Fig.~\ref{fig:pp4D} and~\ref{fig:pp4DB} for two different $\sigma = 0.001, 0.0001$. In both simulations, we notice that the synapse is always excitatory because $v_s=2>x_j(t)$ for all $t$ and $j = 1, 2$ (both neurons). The simulation techniques are kept exactly the same as for the first coupling scheme. Solving for a symmetric equilibrium point for~\eqref{eq:dML_4D_chemical} from~\eqref{eq:eqPoints4DChem} reduces to solving for $x^*$ from
\begin{align}
\label{eq:eqPointSoln4DChem}
{x^*}^2(1-x^*) - \frac{Ae^{\alpha x^*}}{\gamma} + I + \sigma \frac{v_s - x^*}{1+e^{-\lambda(x^* - q)}} = 0,
\end{align}
which is dependent on the coupling strength $\sigma$. For $I = 0.019$, we will have $x^* \approx 0.41279$ and $x^* \approx 40824$ for $\sigma = 0.001$ and $0.0001$ respectively. The corresponding $\beta^*$ values will then be $\beta^*\approx 0.98628$ and $\beta^* \approx 0.98274$ respectively. In both simulations for the different $\sigma$ values, we notice that the dynamics of~\eqref{eq:dML_4D_chemical} always converges asymptotically to a symmetric stable fixed point for $\beta < \beta^*$. As expected a Hopf bifurcation occurs at $\beta = \beta^*$, and the time series displays tonic spiking with increasing amplitudes as $\beta$ is further increased beyond $\beta^*$. In that case, we see the appearance of stable limit cycles in the phase portraits for both neurons and the symmetric equilibrium loses stability.

\begin{figure*}
\begin{center}
\begin{tabular}{cc}
  \includegraphics[scale=0.25]{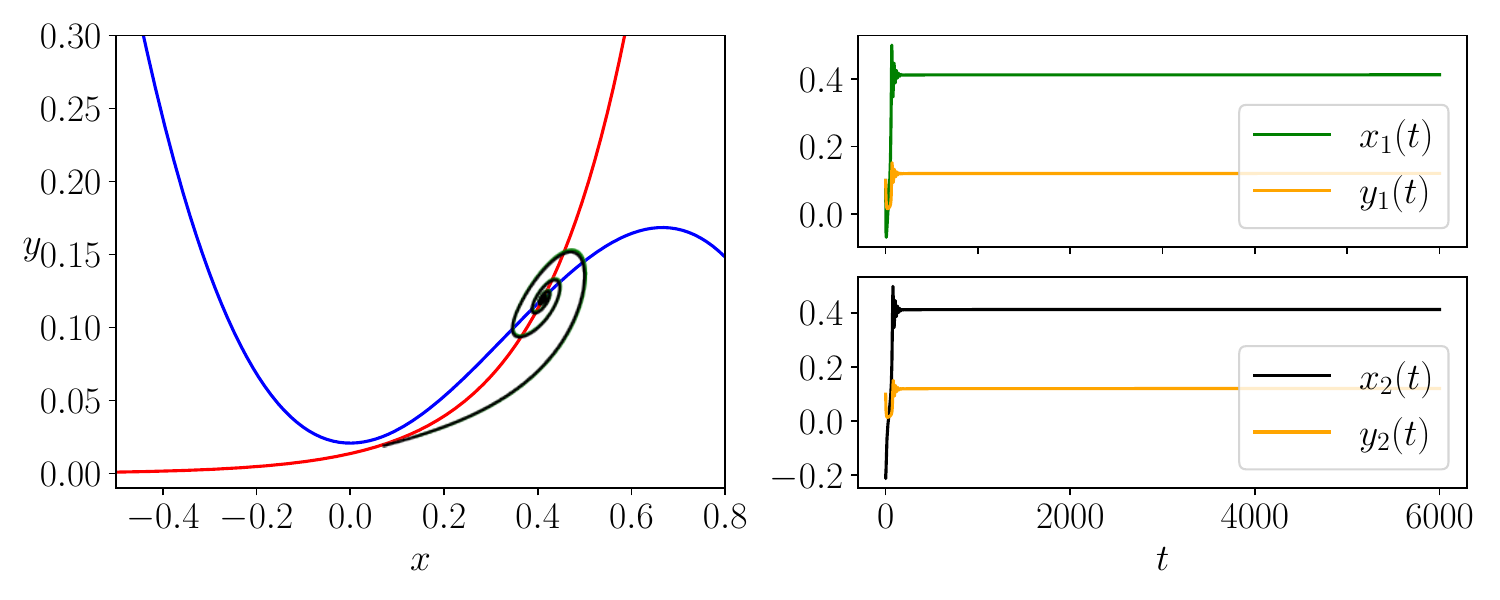} &  \includegraphics[scale=0.25]{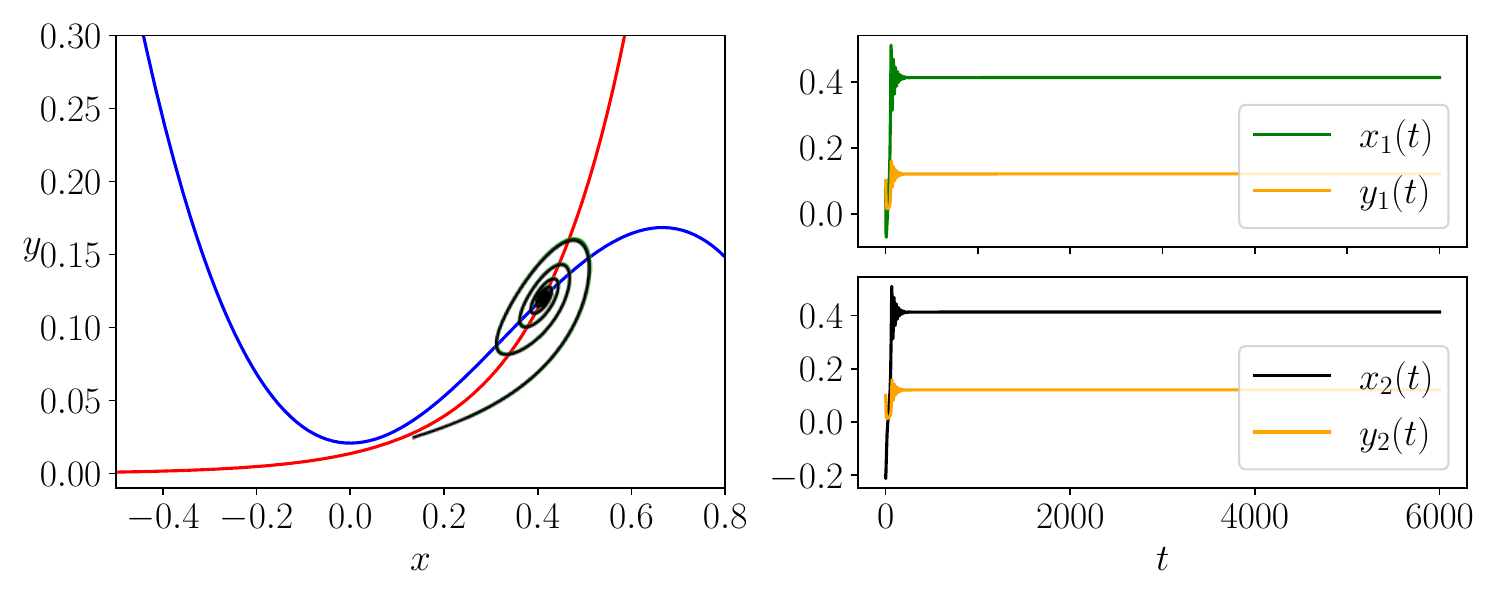} \\
(a) $\beta = 0.9$ & (b) $\beta = 0.93$ \\[3pt]
\includegraphics[scale=0.25]{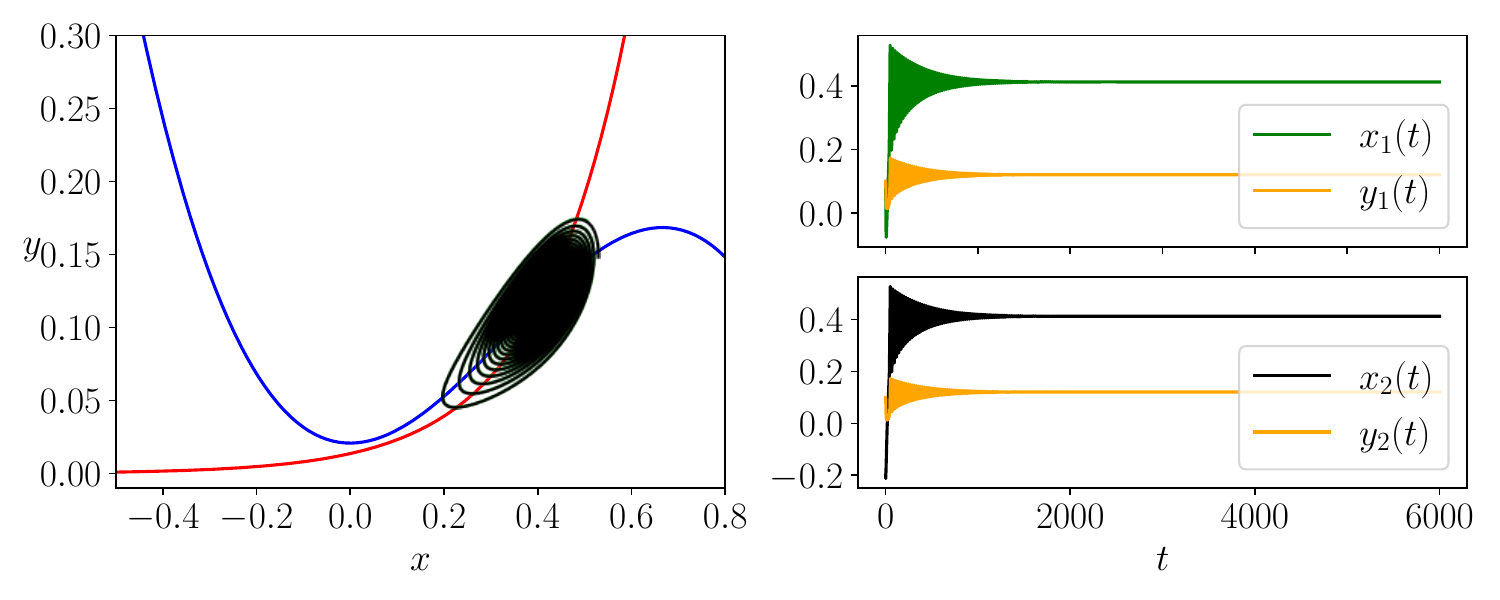} &  \includegraphics[scale=0.25]{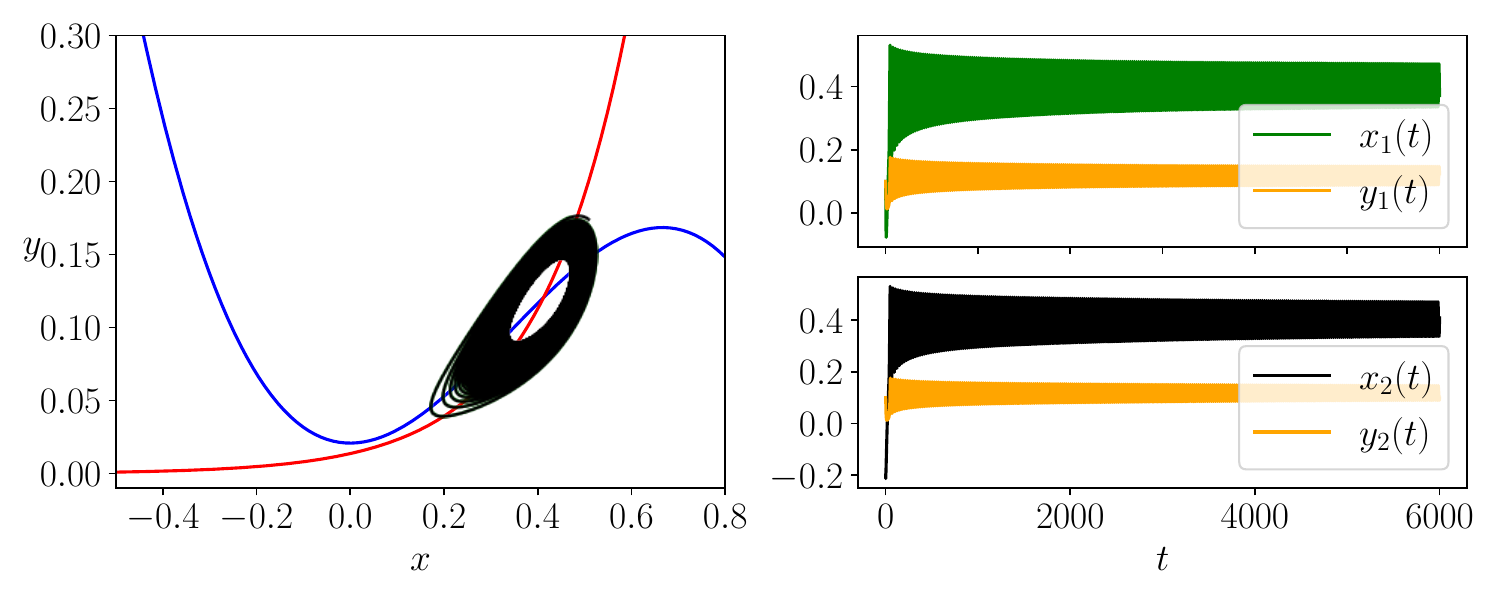} \\
(c) $\beta = 0.98$ &  (d) $\beta = \beta^*$\\[3pt]
\includegraphics[scale=0.25]{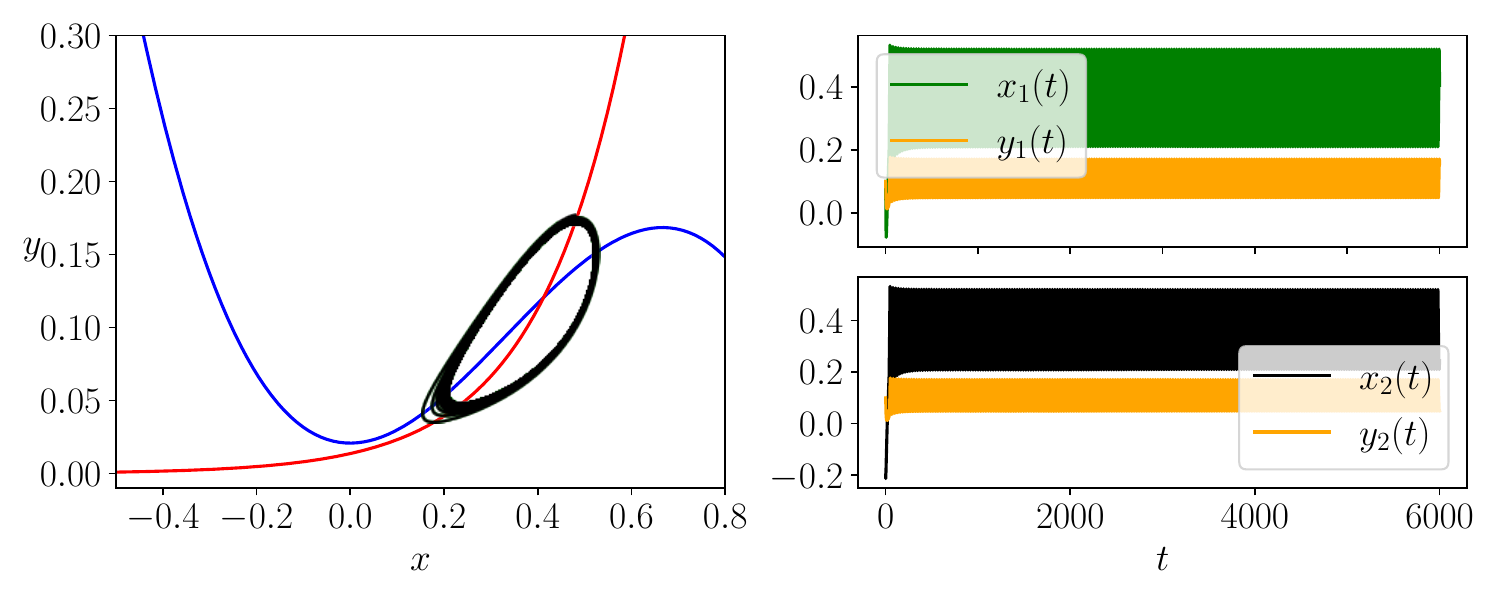} &  \includegraphics[scale=0.25]{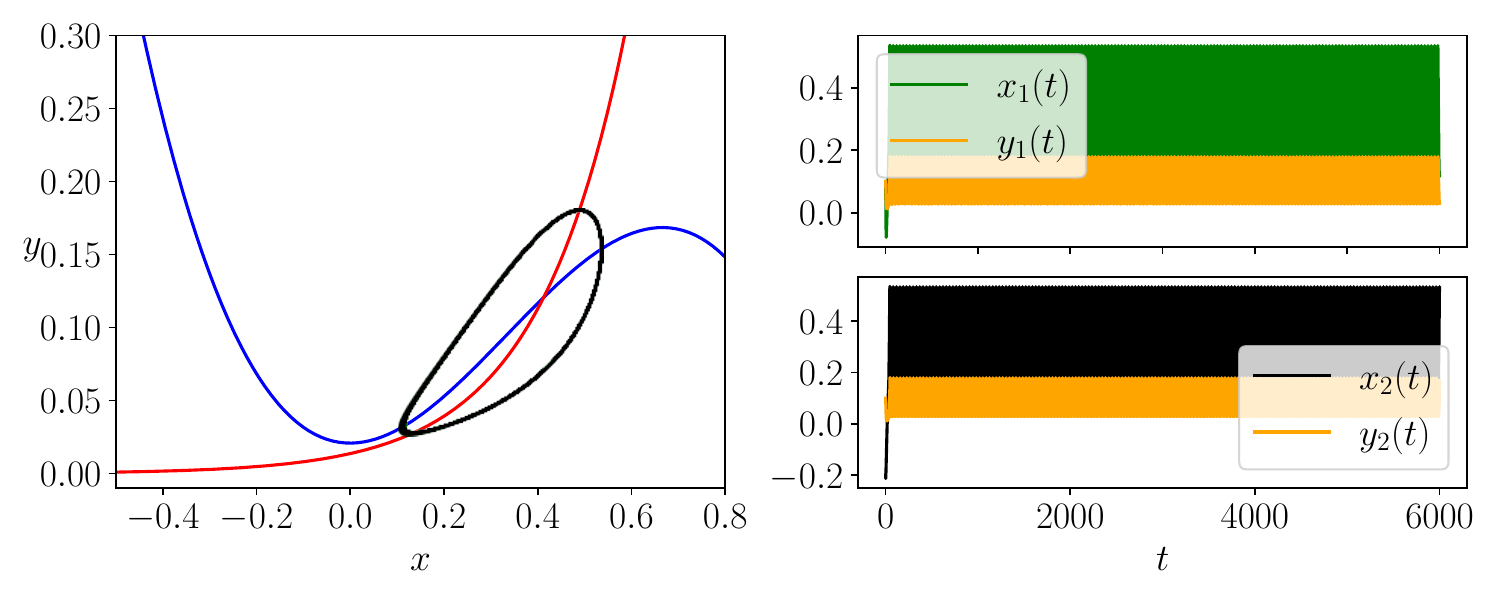} \\
(e) $\beta = 0.99$ & (f) $\beta = 1$ \\[3pt]
\end{tabular}
\end{center}
\caption{Phase portraits and time series with varying $\beta$ for~\eqref{eq:dML_4D_chemical} with a coupling strength $\sigma = 0.001$. We have set $I = 0.019>I_{\max}$ similar to what we have done for the two-dimensional system. The other local parameters for both the neurons are set as~\eqref{eq:param} Furthermore, we have $v_s=2$, $\lambda = 10$, and $q = -0.25$. We will have $\beta^* \approx 0.98628$. The coupled system~\eqref{eq:dML_4D_chemical} converges to a stable symmetric equilibrium $(x^*, y^*, x^*, y^*)$ when $\beta < \beta ^*$. When $\beta \ge \beta^*$, the equilibrium loses stability through a Hopf bifurcation and a stable limit cycle appears for both neurons.}
\label{fig:pp4DChem}
\end{figure*}

\begin{figure*}
\begin{center}
\begin{tabular}{cc}
  \includegraphics[scale=0.25]{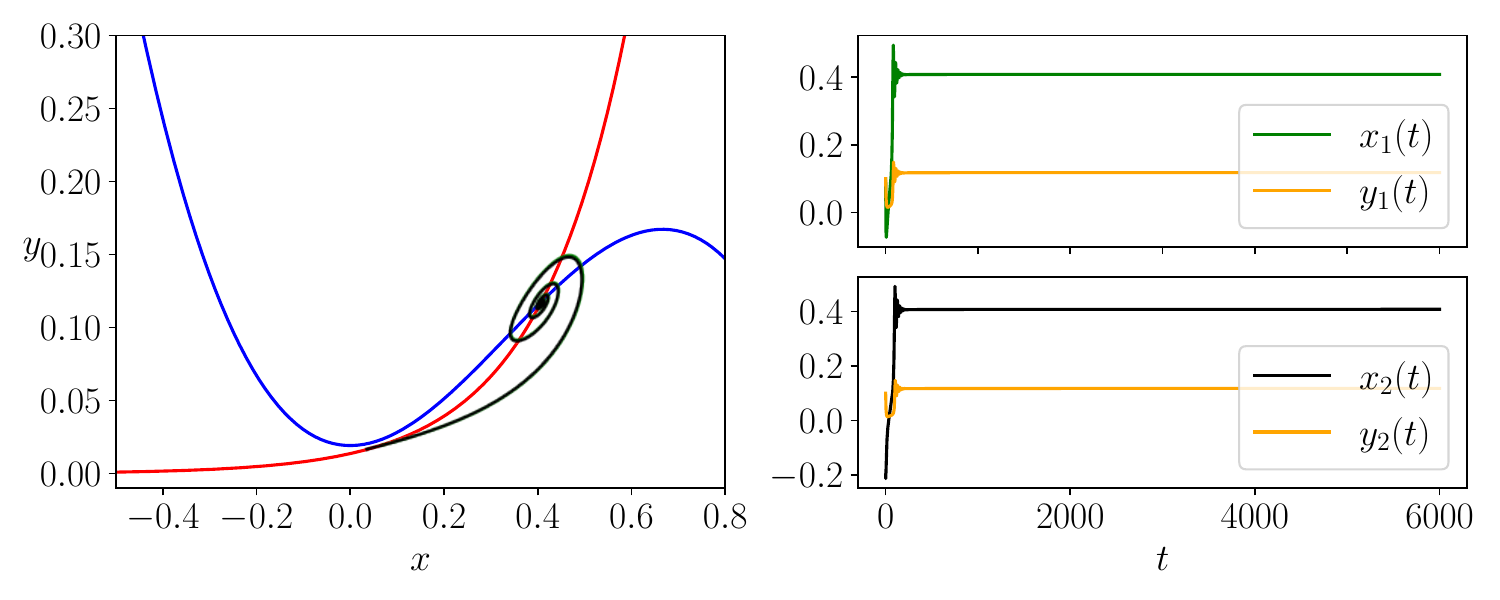} &  \includegraphics[scale=0.25]{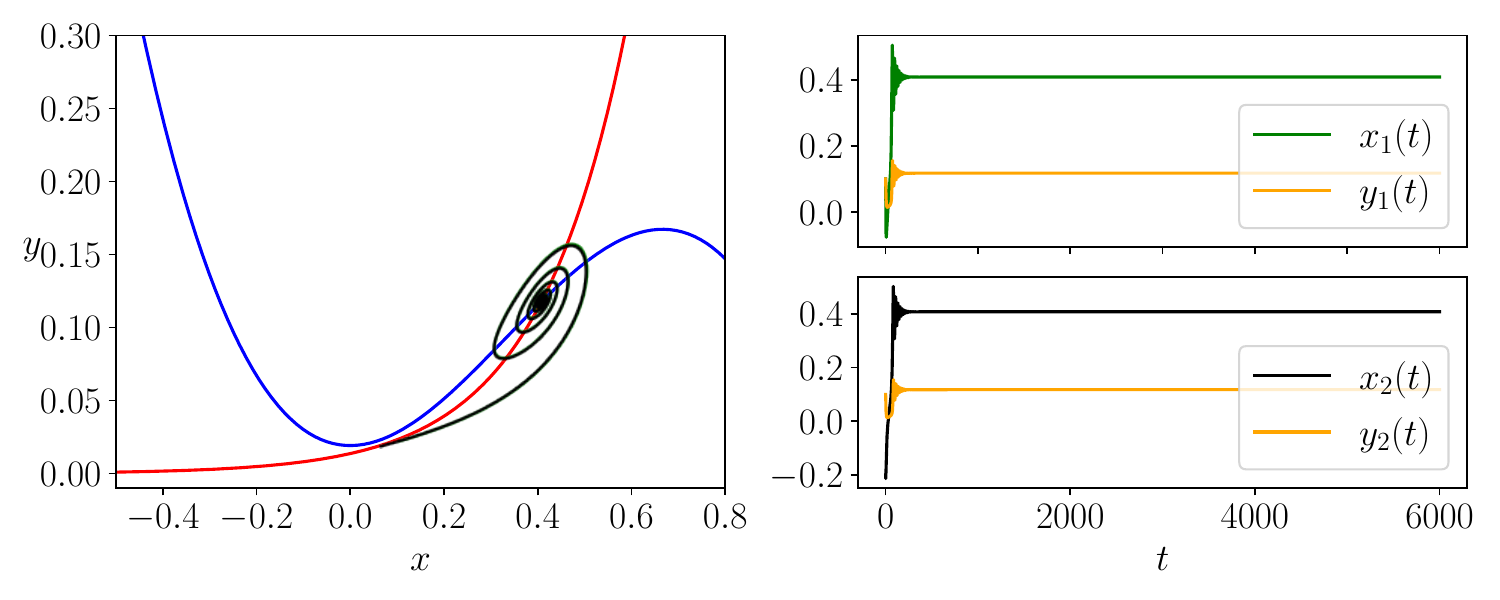} \\
(a) $\beta = 0.9$ & (b) $\beta = 0.93$ \\[3pt]
\includegraphics[scale=0.25]{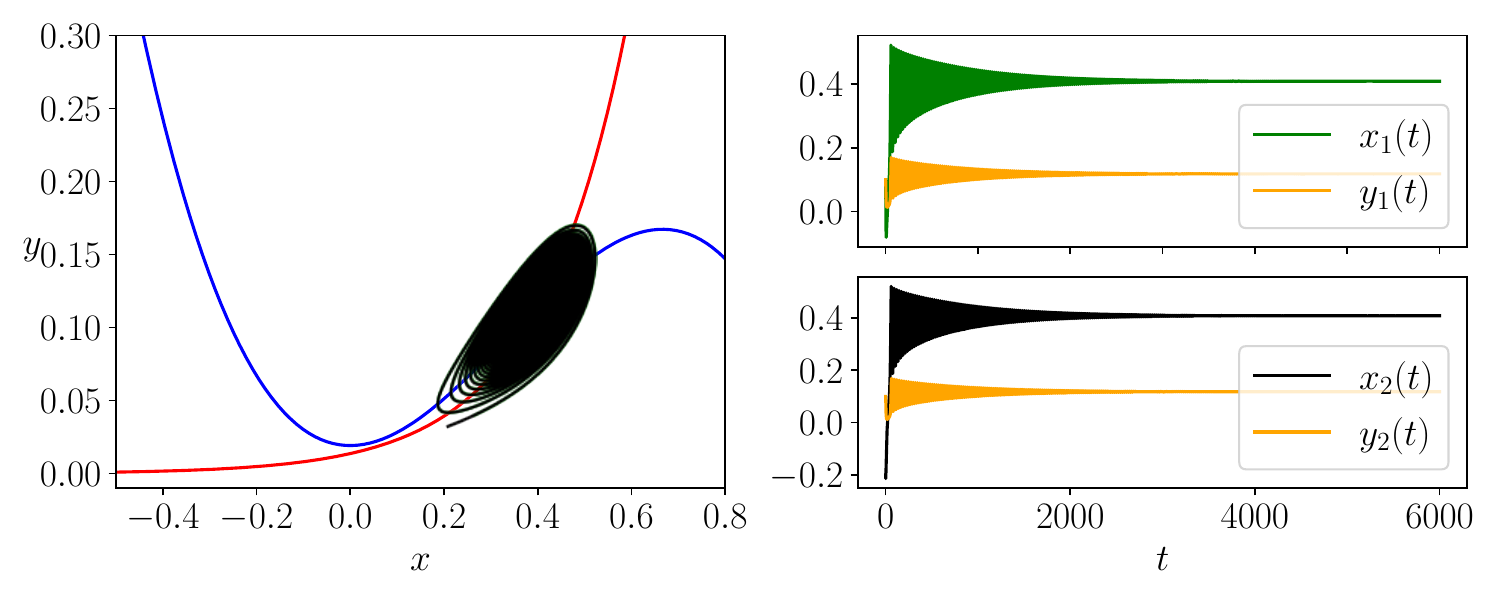} &  \includegraphics[scale=0.25]{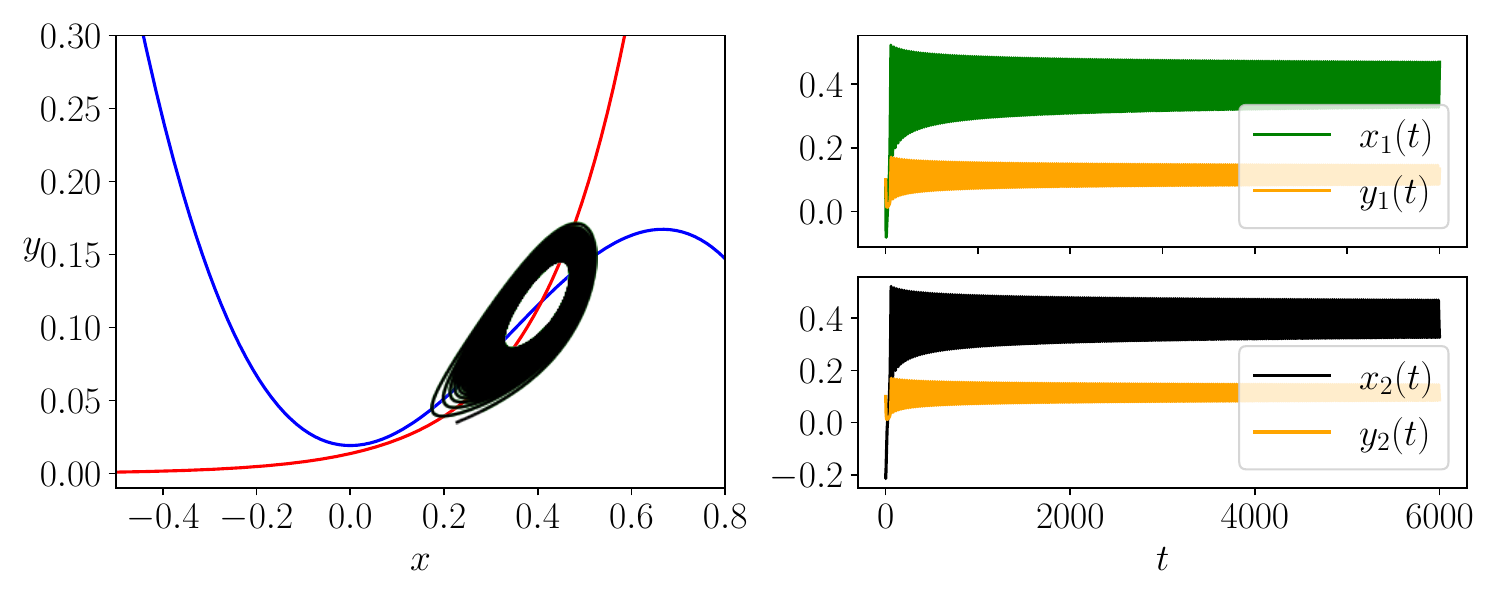} \\
(c) $\beta = 0.98$ &  (d) $\beta = \beta^*$\\[3pt]
\includegraphics[scale=0.25]{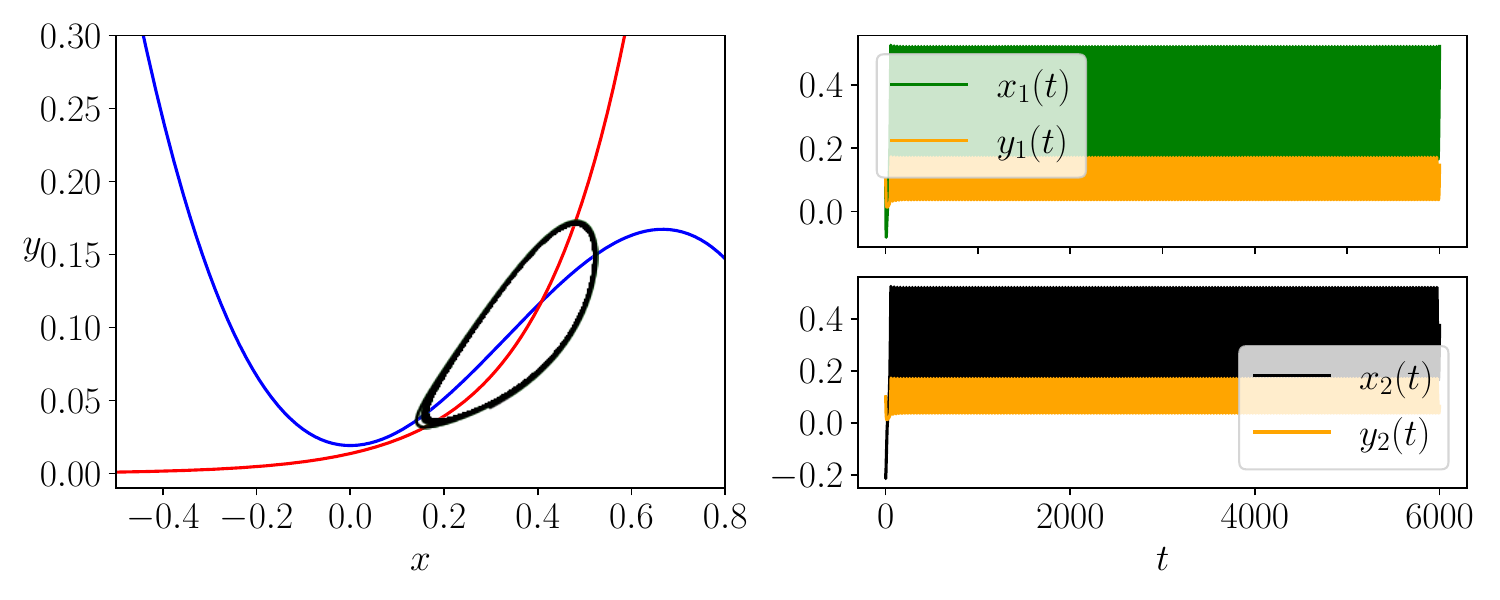} &  \includegraphics[scale=0.25]{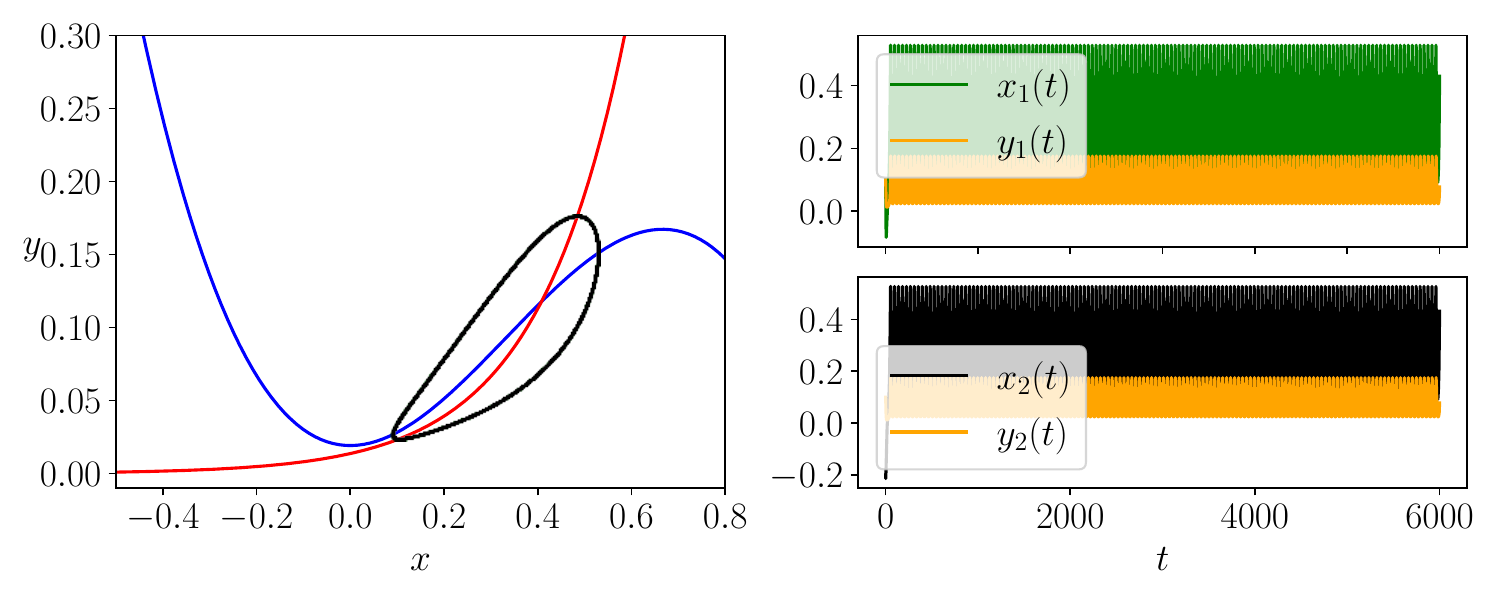} \\
(e) $\beta = 0.99$ & (f) $\beta = 1$ \\[3pt]
\end{tabular}
\end{center}
\caption{Phase portraits and time series with varying $\beta$ for~\eqref{eq:dML_4D_chemical} with a coupling strength $\sigma = 0.0001$. We have set $I = 0.019>I_{\max}$ similar to what we have done for the two-dimensional system. The other local parameters for both the neurons are set as~\eqref{eq:param} Furthermore, we have $v_s=2$, $\lambda = 10$, and $q = -0.25$. We will have $\beta^* \approx 0.98274$. The coupled system~\eqref{eq:dML_4D_chemical} converges to a stable symmetric equilibrium $(x^*, y^*, x^*, y^*)$ when $\beta < \beta ^*$. When $\beta \ge \beta^*$, the equilibrium loses stability through a Hopf bifurcation and a stable limit cycle appears for both neurons.}
\label{fig:pp4DChemB}
\end{figure*}

The bifurcation diagrams for $x_1(t), x_2(t)$ with varying $\beta$ are displayed in Fig.~\ref{fig:BIF4DChem} for two different cases of $\sigma$: (a) $\sigma = 0.001$ and $\sigma = 0.0001$. In each case $\beta^*$ is a function of $\sigma$ because computation of $x^*$ depends on $\sigma$. Thus we have Hopf bifurcation occurring at two different values of $\beta$ for two different $\sigma$'s for the same $I = 0.019$. Lastly, we show the Hopf bifurcation curves on the $(I, \beta)$-plane with $I$ varied in the range $[0.016, 0.0235]$ in Fig.~\ref{fig:Hopf4DChem}. Five different Hopf curves are displayed for values $\sigma = 0, 0.0001, 0.0005, 0.001, 0.003$. Not that for the special case when $\sigma = 0$, the Hopf curve is identical to Fig.~\ref{fig:Hopf2D}. As $\sigma$ keeps increasing, the Hopf curve keeps shifting up with more subregion of the $(I, \beta)$-plane opening up where the symmetric equilibrium point is stable.

\begin{figure}[h]
\begin{center}
\begin{tabular}{c}
  \includegraphics[scale=0.4]{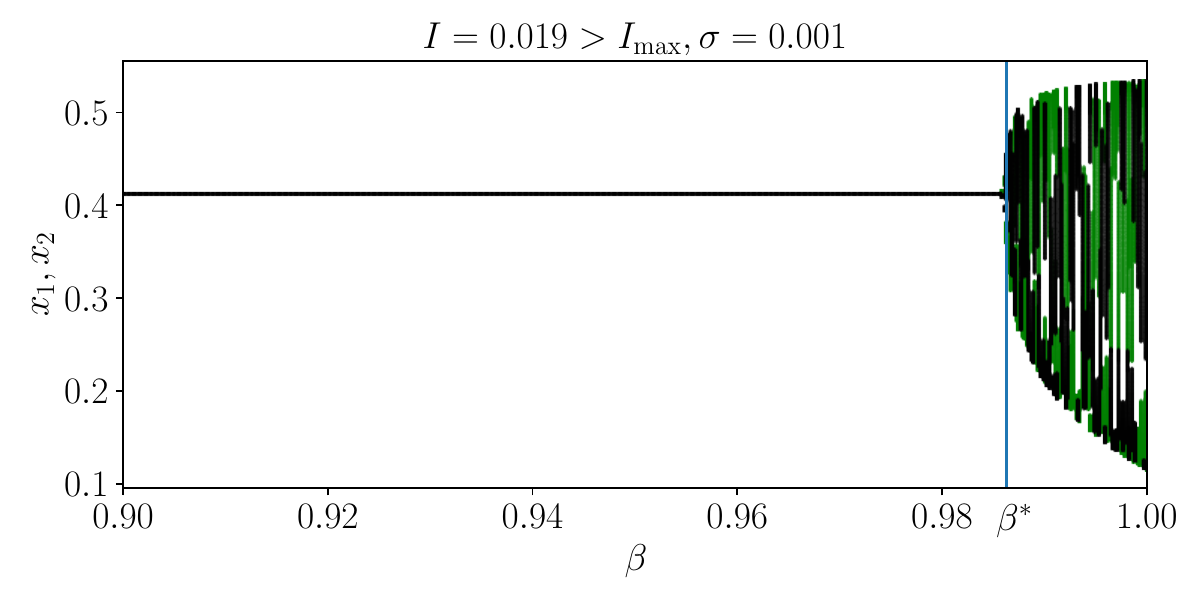} \\
(a) $\beta^* \approx 0.98628$ \\[3pt]
\includegraphics[scale=0.4]{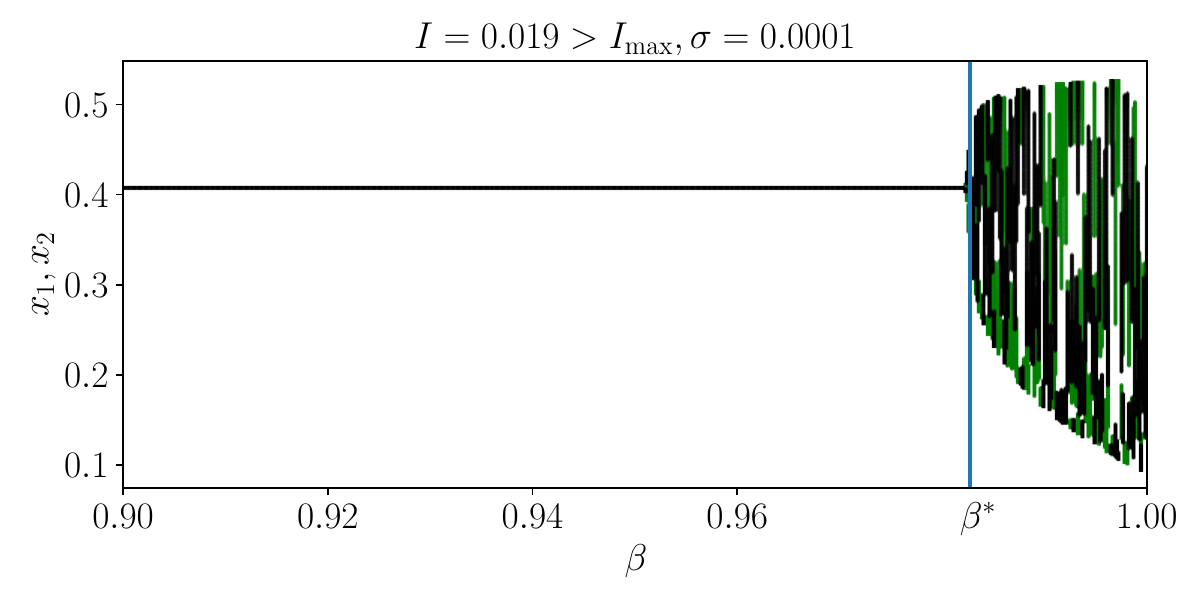} \\
(b) $\beta^* \approx 0.98274$ \\[3pt]
\end{tabular}
\end{center}
\caption{Bifurcation diagrams of $x_1(t), x_2(t)$ with varying $\beta$ for~\eqref{eq:DML_4D} at $I = 0.019 > I_{\rm max}$ with fixed $\sigma$. Specifically, we have (a) $\sigma=0.001$ whose corresponding phase portraits are displayed in Fig.~\ref{fig:pp4DChem} and (b) $\sigma = 0.0001$ whose corresponding phase portraits are displayed in Fig.~\ref{fig:pp4DChemB}. Local parameters for both neurons are set following~\eqref{eq:param} along with $v_S =2$, $\lambda = 10$, and $q=-0.25$. Variable $x_1(t)$ is plotted in green and $x_2(t)$ in black. The vertical blue lines indicate (a) $\beta^* \approx 0.98628$ for $\sigma = 0.001$ and (b) $\beta^* \approx 0.98274$ for $\sigma = 0.0001$ indicating Hopf bifurcations.}
\label{fig:BIF4DChem}
\end{figure}


\begin{figure}
    \centering
    \includegraphics[width=0.5\linewidth]{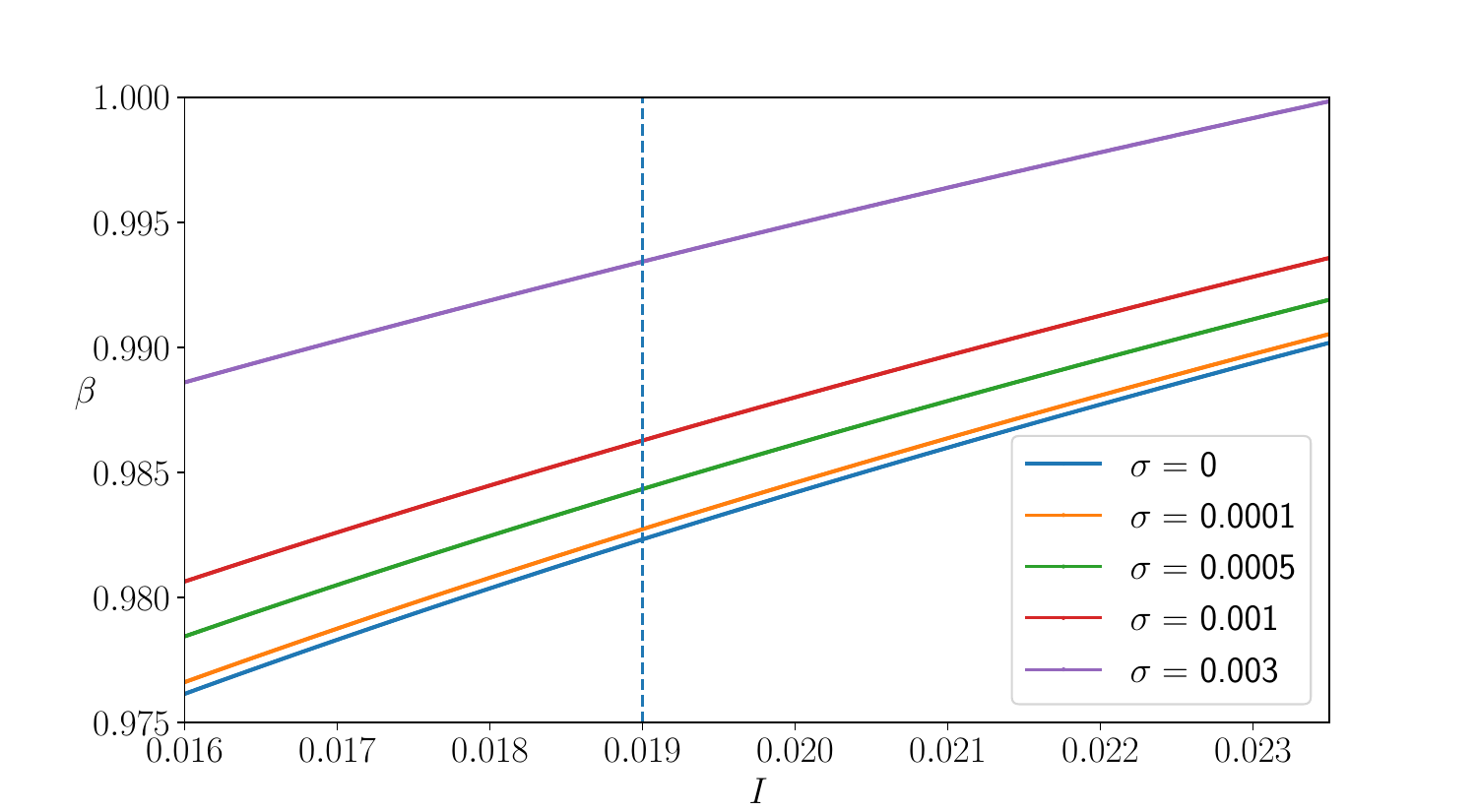}
    \caption{Hopf bifurcation curve $\beta^*$ for~\eqref{eq:dML_4D_chemical} separating the stable and the unstable regions of the symmetric equilibrium point $(x^*, y^*, x^*, y^*)$ on the $(I, \beta)$-plane. The vertical broken line indicates $I = 0.019$. Five different curves for different $\sigma$ values are displayed. The system undergoes a Hopf bifurcation when $\beta$ crosses the curve from the bottom for a particular $I$ value. Other parameter values are set as~\eqref{eq:param} along with $v_s=2$, $\lambda = 10$, and $q = -0.25$. With an increasing $\sigma$, the Hopf curve $\beta^*$ keeps shifting upwards with more stability region opening up for the symmetric equilibirum.}
    \label{fig:Hopf4DChem}
\end{figure}

\section{Conclusions}
\label{sec:Conclusions}
Reduced order models like the dML model~\eqref{eq:dML} provide a computationally efficient scenario to study the dynamics of excitable cells like neurons. The dML model is a simple two-dimensional neuron dynamics model that fully delivers the dynamics portrayed by the original conductance-based Morris-Lecar neurons. Moreover~\eqref{eq:DML_2D} demonstrates a topological similarity to a FitzHugh-Nagumo type neuron mode with the difference lying in the $y$-nullcilne ($y$-nullcline in dML neuron is exponential). After Fatoyinbo {\em et al.}~\citep{FaMu22} extensively explored~\eqref{eq:dML} in terms of codimension-one and -two bifurcation analysis, Ghosh {\em et al.}~\citep{GhFa24} recently performed an extensive analysis of a slow-fast three-dimensional dML model, also first introduced by Schaeffer and Cain in their book~\citep{ScCa18}. An important research question following this was ``How can one modify this simple reduced-order model to make it biophysically more meaningful?". One promising research approach is using fractional-order models of excitable cells that leads to a satisfactory answer. Fractional-order systems are biophysically more realistic than their standard integer-order counterparts, providing an efficient modeling of many real-world phenomena~\citep{AlBa16}. Kaslik~\citep{Ka17} says ``fractional-order derivatives are more precise in the description of dielectric processes and memory properties of membranes."

Motivated, we have explored the dynamics of single-cell and two-coupled fractional order dML neurons in this study. We have taken both a theoretical and a numerical approach to analyze the single-cell and the two-coupled Caputo-type fractional-order system of dML neurons. We considered two different coupling schemes: one just a linear flow and the other via a sigmoidal function. Qualitatively, we have established the necessary and sufficient conditions for an equilibrium point to be stable and for the appearance of a saddle-node bifurcation in terms of the system parameters. We also established the conditions for a Hopf bifurcation: when a stable equilibrium loses its stability alongside the appearance of a stable limit cycle as the order of the fractional derivative is varied. We usually notice tonic spiking after the dynamics transitions through a Hopf bifurcation. For the two-coupled system where the coupling is modeled by a linear flow~\eqref{eq:betaStar4D}, we also observe a typical firing behavior as $\beta$ approaches $1$ after the Hopf bifurcation. There also exist mixed-mode oscillations across the two neurons as $\beta$ is increased. We also make an important observation that for~\eqref{eq:betaStar4D} a Hopf bifurcation intrinsically depends on the magnitude of the coupling strength $\sigma$. For a higher coupling strength, the stability region for a symmetric equilibrium increases on a $(I, \beta)$-plane. The coupling strength $\sigma$ and the memory index various stability regimes. 

The Caputo derivative provides a lot of advantages in terms of modeling memory effects in models and also the fact that the Caputo derivative of a constant is zero. However, researchers should be careful while implementing Caputo derivatives in modeling biophysical systems. They should take into consideration that memory effects can make the computational task more inefficient. Also, Caputo derivatives are only defined for differentiable functions~\citep{AtSe13} which might make mathematical modeling of a biophysical system consisting of discontinuities with Caputo derivatives absurd. 

All in all, we have performed an extensive analysis of a commensurate system of fractional-order differential equations representing the dynamics of single-cell and two-coupled dML neurons. The next step forward is to consider a network of fractional-order dML neurons, for example Erd\"os-Renyi or ring-star. Another potential step is to consider higher-order coupling in these networks, also externally perturbed with different noise source, making the whole scenario stochastic. It would also be interesting to explore the dynamical responses of the three-dimensional slow-fast system of dML neurons~\citep{ScCa18, GhFa24} in a Caputo-type setting. We also beleive considering an incommensurate fractional-order system will induce more complexity in the model and is worth investigating.

\section*{Code Availability}
All computer codes are openly available in Github at: \url{https://github.com/indrag49/fractional-Order-dML}.


\appendix

\section{Software Implementation}
\label{app:Software}
In order to simulate the dynamics of a system of Caputo-type fractional order differential equations we employ \texttt{FdeSolver.jl} (v1.0.8), a \texttt{Julia} package for solving fractional differential equations, by Khalighi {\em et al.}~\citep{KhBe24,BeKh24}. This is a high-performing open-source \texttt{Julia} package built on the algorithms by Diethelm {\em et al.}~\citep{DiFo02, DiFo04} and the \texttt{MATLAB} routines by Garrappa~\citep{Ga18}. There are four main steps going on inside \texttt{FdeSolver.jl}:
\begin{enumerate}
    \item conversion of the fractional differential equations problem to a {\em Volterra integral equation},
    \item discretization of the Volterra integral equation via {\em product-integration} rules,
    \item numerical solution using {\em predictor-corrector} algorithms, and
    \item reduction of the computation cost using {\em fast Fourier transform}.
\end{enumerate}
Readers are encouraged to refer to~\citep{KhBe24} for more details. A typical function to solve a Caputo-type system of fractional differential equations looks like
\begin{tcolorbox}
\texttt{FDEsolver(F, tSpan, y0, $\beta$, par, h)}
\end{tcolorbox}
where \texttt{F} represents the right-hand side of the system of differential equations, denoted as a function. This should return a vector. The input \texttt{tSpan} represents the time span written as a vector \texttt{[initial time, final time]}. Next, we have a row vector of real numbers \texttt{y0} representing the initial values. Then, we have $\beta$ which is also a row vector of the order of the Caputo differential operators. Each element of this vector determines the order of one differential equation from a system of fractional differential equations. If all the elements of $\beta$ are the same, then the system of fractional differential equations we are numerically simulating is commensurate, otherwise it is incommensurate. This vector can take both integer and numerical values. We keep all the elements of $\beta$ in the range $(0, 1]$. The row vector of real numbers \texttt{par} is a collection of parameters that govern the dynamics of the fractional differential equations. Finally \texttt{h} is a real scalar that sets the step size of the numerical simulation.

Once we simulate a system of fractional differential equations, we would need to visualize the results. For this, we first collect the simulation results into a CSV datafile (using \texttt{Julia}'s \texttt{DataFrames} and \texttt{CSV} packages) and then read the datafile using \texttt{Python}'s (v3.10.9) \texttt{Pandas} library, specifically the \texttt{pd.read\_csv()} function. After that, we visualize the results using the \texttt{matplotlib} library. All computer codes are openly available at: \url{https://github.com/indrag49/fractional-Order-dML}.

\bibliographystyle{apalike}
\bibliography{main}

\end{document}